\documentclass[oneside, a4paper,11pt,reqno]{amsart}
\usepackage{tikz}
\usepackage{amsmath,amssymb}
\usetikzlibrary{decorations.markings}

\usepackage{bbm}
\usepackage[T1]{fontenc}

\usepackage{verbatim}
\usepackage{amsmath}

\usepackage{amsfonts}
\usepackage{color}

\title[Second order cubic corrections]{Second order  cubic corrections of  large deviations for perturbed  random walks}
\author{Giancarlos Oviedo, Gonzalo Panizo and Alejandro F. Ram\'\i
  rez}

\email{g.oviedovalverde@gmail.com, gonzalo.panizo@gmail.com, \newline \ \ \ \ \ aramirez@mat.uc.cl}

\address{Instituto de Matem\'atica y Ciencias Afines, Universidad 
  Nacional de Ingenier\'\i a, Lima, Per\'u, Facultad de Matem\'aticas, Pontificia Universidad Cat\'olica 
  de Chile, Santiago, Chile and NYU-ECNU Institute of Mathematical Sciences at NYU Shanghai, 
  Shanghai, China}

\date{\today}

\usepackage{amsmath,amssymb,amsfonts,amsthm}

\usepackage[colorlinks=true,linkcolor=black,citecolor=black,urlcolor=blue,pdfborder={0 0 0}]{hyperref}

\usepackage{graphicx}
\usepackage{enumitem}
\usepackage{color}

\usepackage[margin=1cm]{caption}
\usepackage[protrusion=true
]{microtype}

\usepackage[nameinlink]{cleveref}
  \crefname{theorem}{Theorem}{Theorems}
  \crefname{thm}{Theorem}{Theorems}
  \crefname{mainthm}{Theorem}{Theorems}
  \crefname{lemma}{Lemma}{Lemmas}
  \crefname{lem}{Lemma}{Lemmas}
  \crefname{remark}{Remark}{Remarks}
  \crefname{prop}{Proposition}{Propositions}
  \crefname{defn}{Definition}{Definitions}
  \crefname{corollary}{Corollary}{Corollaries}
  \crefname{section}{Section}{Sections}
  \crefname{figure}{Figure}{Figures}

\newtheorem{thm}{Theorem}[section]
\newtheorem{theorem}[thm]{Theorem}

\newtheorem{lemma}[thm]{Lemma}
\newtheorem{corollary}[thm]{Corollary}
\newtheorem{prop}[thm]{Proposition}

\numberwithin{equation}{section}

\usepackage{bbm}
\usepackage{mathtools,amsfonts,amsthm}
\usepackage{xcolor}
\usepackage{lipsum,upgreek}
\usepackage{physics}

  \newcounter{cnstcnt}
\newcommand{\newconstant}{%
\refstepcounter{cnstcnt}%
\ensuremath{C_{\thecnstcnt}}}
\newcommand{\oldconstant}[1]{\ensuremath{C_{{\ref{#1}}}}}

\newcounter{cnst}
\newcommand{\newc}{%
\refstepcounter{cnst}%
\ensuremath{c_{\thecnst}}}
\newcommand{\oldc}[1]{\ensuremath{c_{{\ref{#1}}}}}

\newcounter{cnsta}
\newcommand{\newca}{%
\refstepcounter{cnsta}%
\ensuremath{{\mathfrak a}_{\thecnsta}}}
\newcommand{\oldca}[1]{\ensuremath{{\mathfrak a}_{{\ref{#1}}}}}

\thanks{Alejandro Ram\'\i rez  has been 
partially supported by Iniciativa Cient\'\i fica Milenio NC120062 
and by Fondo Nacional de Desarrollo Cient\'\i fico 
y Tecnol\'ogico grant 1180259}

\newtheorem{definition}{Definition}
\newtheorem{proposition}{Proposition}

\begin{document}

\maketitle

\begin{abstract}
  We prove that the Beta random walk, introduced in \cite{Barr-Cor} 2017, has
  second order cubic fluctuations
    from the large deviation principle  
 of the GUE Tracy-Widom
 type for arbitrary values $\upalpha>0$ and $\upbeta>0$ of the
 parameters of the Beta distribution, removing previous restrictions
 on their values. Furthermore, we prove that the GUE Tracy-Widom fluctuations
  still hold in the intermediate disorder regime. We also show that any random walk in space-time random
  environment that matches certain moments with the
  Beta random walk also has GUE Tracy-Widom
  fluctuations in the intermediate disorder regime. As a corollary we
  show  the emergence of GUE Tracy-Widom fluctuations
  from the large deviation principle for trajectories ending at boundary points for random walks
  in space (time-independent) i.i.d. Dirichlet random environment
  in dimension $d=2$ for a class of asymptotic behavior of the parameters. \end{abstract}

\section{Introduction} The Beta-random walk, introduced in
\cite{Barr-Cor}, is a random walk in 
space-time i.i.d. random environment in $\mathbb Z$ which is exactly 
solvable. Given $\upalpha,\upbeta\in (0,\infty)$, for each space time point $(t,x)\in\mathbb N\times\mathbb Z$, 
a Beta-random variable $B_{t,x}$ of parameters $\upalpha$ and $\upbeta$
is associated to the space-time point $(t,x)$, so that $(B_{t,x})_{(t,x)\in\mathbb N\times\mathbb Z}$
are i.i.d. The Beta-random walk $(X_t)_{t\in\mathbb N}$ in the
environment $(B_{t,x})_{(t,x)\in\mathbb N\times\mathbb Z}$
is then defined by its transition 
probabilities 

$$
P\left(X_{t+1}=x+1|X_t=x\right)=B_{t,x} \qquad {\rm for}\ x\in\mathbb Z.
$$
A general quenched large deviation principle for the position of the walk, which includes the Beta-random 
walk, was proved in \cite{RasA-Sepp-Y}. Subsequently, Barraquand and Corwin in 
\cite{Barr-Cor}, obtained an explicit formula for the quenched
large deviation rate function and also
proved for the case $\upalpha=\upbeta=1$, second order cube-root scale corrections to this large 
deviation principle, with convergence to the GUE Tracy-Widom
distribution. This was
recently extended for the case $\upalpha=1$ and $\upbeta>0$ in \cite{K21}.

Directed polymers in random environment in dimension $1+1$ are
in strong disorder for all inverse temperatures $\beta>0$, and it is conjectured that
for general distributions, there
are second order cube-root fluctuations for the rescaled free energy
with GUE Tracy-Widom
statistics.
Random walks in space-time i.i.d. random environments share several
features with directed polymers in random environment. It is hence
natural to  conjecture that a similar behavior would occur: second order corrections to the quenched
large deviation principle for random walk in space-time i.i.d. random
environment
for a large class of distributions of the random environment, still
 of scale cube-root, with convergence to the GUE Tracy-Widom
distribution. Here, the strength of the disorder given by the random
environment, plays the role of the inverse temperature (for
dimensions $d\ge 3+1$,   regimes analogous to
the strong and weak disorder regimes of directed polymers
exist for the random walk in space-time i.i.d. environment
\cite{BMRS19}). The intermediate disorder regime probes the
strong-weak
disorder transition of directed polymers in dimensions $1+1$, by taking the limit $\beta\to 0$
as the length of the polymer is increased. For the special case
$\beta_N=\hat\beta N^{-1/4}$, it was shown in \cite{AKQ14}, that
the fluctuations crossover between the Edwards-Willkinson regime to
the
GUE Tracy-Widom regime as $\hat\beta\to\infty$, and the
Kardar-Parisi-Zhang equation (KPZ) appears
in the limit. Subsequently Krishnan and Quastel in \cite{KQ18}, showed
through
a perturbation argument involving matching a certain number
of moments with the log-gamma polymer \cite{S12}, the universality of
the GUE Tracy-Widom
distribution when $1\gg\beta_N\gg O(N^{-1/4})$ (see similar ideas used
within the context of chaos phenomena and ultrametricity in the mixed
$p$-spin model in \cite{AC16}).

The intermediate disorder has also been studied within the context of  random walks in space-time
i.i.d. random
environment in $\mathbb Z$ which are perturbations of the simple symmetric random
walk. Indeed, in \cite{CG17}, the intermediate disorder regime was
proven for a random walk in space-time i.i.d. random environment
defined by its probability to jump to the right at a given time $t$ from
a site $x$ as
$\frac{1}{2}+t^{-1/4}\xi(x,t)$ with $(\xi(x,t))_{x\in\mathbb Z,
  t\in\mathbb N}$  i.i.d. with values in $[0,1]$. This corresponds to
the case $\beta_N=\hat\beta N^{-1/4}$ of directed polymers. A second result in \cite{CG17} gives an intermediate
disorder
regime limit of the same kind for the logarithmic fluctuations of
the transition probability $P_{0,\omega}(X_t=y)$
of the Beta random walk, with scaling  $y=\gamma 
t+xt^{1/2}$, $\gamma\in (0,1/2)$, $x\in\mathbb R$ with time-dependent parameters
$\upalpha_t=\upbeta_t=
t^{1/2}$, as $t\to\infty$ is time.

In this article we first prove convergence to the GUE Tracy-Widom 
distribution of the second order fluctuations of the Beta random walk 
for any $\upalpha>0$ and $\upbeta>0$. This result removes
the restrictions of \cite{Barr-Cor}, where the
convergence was proven for $\upalpha=\upbeta=1$ and of \cite{K21}
where
it was extended to the
case $\upalpha=1$ and $\upbeta>0$. Furthermore, our result shows
that we can keep the same range $\theta\in (0,0.5)$ parametrizing the
target points $x(\theta)$ (which is expressed as a rational function
invlolving polygamma functions) of the large deviation event $\{X_t\ge
x(\theta) t\}$
as in \cite{Barr-Cor}, as long
as $\upalpha\ge 0.7$ and $\upbeta>0$. As a second result we 
prove convergence to the GUE Tracy-Widom distribution  for the logarithmic fluctuations of the
probability  $P_{0,\omega}(X_t\ge x_tt)$, for $\upalpha_t\to\infty$,
$\upbeta_t\to\infty$
and  $x_t\to 1$, under an appropriate condition on the growth of
$\upalpha_t$ and $\upbeta_t$. For the case
$\upalpha=t^r$,
$\upbeta=t^s$, this condition reduces to
 $ r+\max(r-s,0)<1$.
In the third result of this article
we   show that the
corresponding
result for random walks in space-time i.i.d. environments which are
perturbations of the Beta random walk still holds, along the lines
of \cite{KQ18}.
An interesting corollary of our perturbation results, is the
appearance of the GUE Tracy-Widom fluctuations for random walks in space
i.i.d. (static) Dirichlet environment in dimension $d=2$, for certain
asymptotic behaviour of the parameters, corresponding to events
which force the random walk to move through
directed trajectories.
 The main challenge of the
proofs
 is doing a sophisticated steep-descent
analysis  involving polygamma functions.

In what follows we give a precise formulation of the main results of
this article.

\medskip
\section{Main results}
We first define the model of random walk in random environment on
$\mathbb Z^d\times\mathbb N$, where the factor $\mathbb Z^d$ will
represent the space where it moves, while $\mathbb N$ is the time. Let $|\cdot |_1$ denote the $l_1$ norm.
Let $U:=\{e\in\mathbb Z^d:|e|_1=1\}$ and $\mathcal P:=\{(p(e))_{e\in U}\in [0,1]^{2d}: \sum_{e\in U}
p(e)\le 1\}$.
Let $\Omega=\mathcal P^{\mathbb Z^d\times\mathbb N}$. We call $\Omega$
the
{\it environmental space} and each element
$\omega:=(\omega(x,t))_{x\in\mathbb Z^d, t\in \mathbb N}\in\Omega$, with
$\omega(x,t)=(\omega_e(x,t))_{e\in U}\in\mathcal P$, an {\it
  environment}. Note that  we do not assume necessarily that
$\sum_{e\in U}\omega_e(x,t)=1$, which can be interpreted
as jump probabilities having a non-vanishing probability of absorption at each
step. Given $\omega\in\Omega$, consider the (sub)-Markov chain
$(X_t)_{t\ge 0}$ with state space $\mathbb Z^d$ starting from
$x\in\mathbb Z^d$,
defined through its transition probabilities

$$
P_{x,\omega}(X_{t+1}=y+e|X_t=y)=\omega_e(y,t),
$$
for $y\in\mathbb Z^d$ and $t\ge 0$, with $P_{x,\omega}(X_0=x)=1$. We call the
(sub)-Markov process $(X_t)_{t\ge 0}$ a random walk in the space-time  environment
$\omega$ on $\mathbb Z^d$ and denote by $P_{x,\omega}$ its law. If $\mathbb
P$ is a probability measure defined on $\Omega$, we denote the law
$P_{x,\omega}$
the {\it quenched law} of the {\it random walk in  random
  environment}, and by $E_{x,\omega}$ the expectation corresponding to $P_{x,\omega}$.

A special case of a random walk in random environment is
the Beta random walk,  defined for $d=1$, and where the
environment
is space-time i.i.d. Recall that a random variable $B$ is a Beta random variable of parameters $\upalpha >0$ and $\upbeta >0$ if for every $r\in [0,1]$ we have that
\[
P(B\le r) = \int_0^r x^{\upalpha-1}(1-x)^{\upbeta-1}\frac{\Gamma(\upalpha+\upbeta)}{\Gamma(\upalpha)\Gamma(\upbeta)}dx
\]
Let $(B_{x,t})_{x\in\mathbb{Z},t\ge 0}$  be an i.i.d. collection of
Beta 
random variables of parameter $\upalpha>0$ and $\upbeta>0$.
Let $\mathbb P_{\upalpha, \upbeta}$ be the joint law of $\omega$ with
$\omega_1(x,t)= B_{x,t}$ and $\omega_{-1}(x,t)=1-B_{x,t}$. Then, the random walk in  random environment on $\mathbb Z$ with law
$\mathbb P_{\upalpha,\upbeta}$ is called a {\it Beta random walk}.
We will denote by $\mathbb E_{\upalpha,\upbeta}$ the corresponding expectation.

An  important class of a random walks in random environment
corresponds to the case in which the law $\mathbb P$  of the
environment
is concentrated on environments $\omega$ which are
time-independent, so that $\omega(x):=\omega(t,x)$ for all $t\ge 0$
and $x\in\mathbb Z^d$.
In this case we will use the notation $\omega(x)=(\omega(x,e))_{e\in
  U}$.
A particular example of a random walk in (time-independent) random
environment on $\mathbb Z^d$, is the random walk in Dirichlet
environment (RWDE). To define the RWDE, let us use the
notation $U=\{e_1, e_2,\ldots, e_{2d}\}$ with the convention
$e_{d+i}=-e_i$ for $1\le i\le d$. For each $1\le i\le 2d$, let $\upalpha_i>0$. The Dirichlet
distribution with parameters $\upalpha:=(\upalpha_i)_{i\in\{1,\ldots, k\}}$ is the
distribution
on $\mathcal P_1 =\{(p(e))_{e\in U}\in [0,1]^{2d}: \sum_{e\in U}p(e)\le 1\}$
 which has a density with respect to
the Lebesgue measure in $\mathcal P_1$ given by

$$
\frac{\Gamma(\upalpha_1+\cdots+\upalpha_{k})}{\Gamma
  (\upalpha_1)\cdots\Gamma (\upalpha_{k})} u_1^{\upalpha_1-1}\cdots
u_{k}^{\upalpha_{k}-1}\prod_{i\ne i_0} d u_i,
$$
where $i_0$ is an irrelevant choice of index among $\{1,\ldots,k\}$
and $u_{i_0}=1-\sum_{i\ne i_0}u_i$.
The random walk in Dirichlet environment of parameter
$\upalpha:=(\upalpha_i)_{i\in\{1,\ldots, 2d\}}$ is defined as the random walk on $\mathbb Z^d$ whose
environment $\omega$ is time-independent and has a law $\mathbb
P_{\upalpha}$ under
which $(\omega(x))_{x\in\mathbb Z^d}$ are i.i.d. and have Dirichlet
distribution
of parameter $\upalpha$.

In what follows we will use the
notation
$P_\omega:=P_{0,\omega}$ for the quenched law of the Beta random walk. Define $P_\omega(t,x):=P_\omega(X_t\ge
x)$. In \cite{RasA-Sepp} and \cite{RasA-Sepp-Y}, it was shown that $\mathbb P_{\upalpha, \upbeta}$-a.s.

$$
\lim_{t\to\infty} \frac{1}{t}\log P_\omega(t,xt)=-I(x),
  $$
  where $I$ is the Legendre transform of

  $$
\lambda(s):=\lim_{t\to\infty}\frac{1}{t}\log\left(E_\omega\left[e^{sX_t}\right]\right)\quad
s\in\mathbb R,
$$
where the right-hand side limit exists $\mathbb P_{\upalpha,\upbeta}$-a.s.  
In \cite{Barr-Cor} a closed formula for $I$ was obtained using
critical Fredholm determinant asymptotics, so that $I$ is
implicitely defined by

\begin{equation}
  \label{xtheta}
x(\theta)=\frac{\Psi_1(\theta+\upalpha+\upbeta)+\Psi_1(\theta)-2\Psi_1(\theta+\upalpha)}{\Psi_1(\theta)-
  \Psi_1(\theta+\upalpha+\upbeta)}
\end{equation}
and

\begin{align*}
I(x(\theta))&=\frac{\Psi_1(\theta+\upalpha+\upbeta)-\Psi_1(\theta+\upalpha)}{\Psi_1(\theta)-
  \Psi_1(\theta+\upalpha+\upbeta)}
\left(\Psi(\theta+\upalpha+\upbeta)-\Psi(\theta)\right)\\
&+\Psi(\theta+\upalpha+\upbeta)-\Psi(\theta+\upalpha),
\end{align*}
where $\Psi$ is the digamma function defined as
$\Psi(z)=\Gamma'(z)/\Gamma(z)$,
$\Psi_1(z)=\Psi'(z)$  is the trigamma function, $z\in\mathbb C$,  $\theta\in (0,\infty)$, and as $\theta$ ranges from $0$ to
$\infty$, $x(\theta)$ ranges from $1$ to
$(\upalpha-\upbeta)/(\upalpha+\upbeta)$.
Define $\sigma(\theta)$ by the relation

\begin{align}
  \nonumber
  &2\sigma(\theta)^3=
  \Psi_2(\theta+\upalpha)-\Psi_2(\theta+\upalpha+\upbeta)\\
  \label{sigma}
  &    +\frac{\Psi_1(\theta+\upalpha)-\Psi_1(\theta+\upalpha+\upbeta)}{\Psi_1(\theta)-
  \Psi_1(\theta+\upalpha+\upbeta)}
\left(\Psi_2(\theta+\upalpha+\upbeta)-\Psi_2(\theta)\right),
\end{align}
where $\Psi_2(z):=\Psi_1'(z)$.
By Lemma 5.3 of \cite{Barr-Cor}, the right hand side 
of (\ref{sigma}) is positive so that $\sigma(\theta)>0$.
Our first result is an extension of Theorem 1.15 of \cite{Barr-Cor}
and Theorem 1.2 of \cite{K21}
which includes the case  $\upalpha\ge 0.7$, $\upbeta>0$  and
$\theta\in (0,0.5)$.
Recall that the GUE Tracy-Widom distribution is defined by
$F_{GUE}(x)=\det(I-K_{Ai})_{L^2(x,+\infty)}$, $x\in\mathbb R$, where
$\det(I-K_{A_i})_{L^2(x,+\infty)}$ is the Fredholm determinant of the
Airy kernel $K_{Ai}$ 
kernel defined as

$$
K_{Ai}(u,v)=\frac{1}{(2\pi i)^2}\int_{e^{-2\pi i/3}\infty}^{e^{2\pi
    i/3}\infty}
\int_{e^{-\pi i/3}\infty}^{e^{\pi i/3}\infty}
\frac{e^{z^3/3-zu}}{e^{w^3/3-wv}}\frac{1}{z-w}dzdw,
$$
where the contours for $z$ and $w$ do not intersect.

Our first result is an extension of Theorem 1.15 of \cite{Barr-Cor} and
the implication of Theorem 1.2 of \cite{K21} for the Beta random walk.

\medskip
\begin{theorem}
  \label{one} For all $\upalpha>0$, $\beta>0$ and 
  $\theta\in (0,\min\{0.5, 0.72\times\upalpha\})$, we have that

  \begin{equation}
\nonumber
\lim_{t\to\infty}\mathbb 
P_{\upalpha,\upbeta}\left(\frac{\log\left(P_\omega(t,x(\theta)t)\right)+I(x(\theta))t}{t^{1/3}\sigma(\theta)}\le 
  y 
  \right)=F_{GUE}(y). 
  \end{equation}
\end{theorem}

\medskip

It should be noted that  the range of target
points
$x(\theta)$ varies from $x(\theta)=1$ for $\theta=0$ to
$x(\theta)=\frac{\upalpha-\upbeta}{\upalpha+\upbeta}$, which
is the speed of the random walk, as $\theta\to\infty$. Since
the convergence in the above theorem is
values of $\theta$ not larger than $\min\{0.5, 0.72\upalpha\}$, and
$x(\theta)$ is decreasing with $\theta$, the range of target points
is far from the velocity $\frac{\upalpha-\upbeta}{\upalpha+\upbeta}$.

Also, Theorem \ref{one} shows in particular that the convergence to the GUE
Tracy-Widom distribution occurs for all $\theta\in (0,0.5)$ and $\upbeta>0$ as long as
$\upalpha\ge 0.7$.
The case $\upalpha=\upbeta=1$,  was proven for all $\theta\in (0,0.5)$
in \cite{Barr-Cor} and extended to $\upalpha=1$ and $\upbeta>0$ in \cite{K21}.  
Removing these restrictions to obtain  Theorem \ref{one} is  
technically  
challenging and requires sophisticated  estimates involving the polygamma functions. Our bounds on $\theta$ and $\upalpha$ are not optimal, and the methods
presented here
could give better estimates.

   Our second result shows that the intermediate disorder
  regime
  holds for the Beta random walk with parameters tending to $\infty$.
  It should be noted that since $\theta$ will be fixed, by (\ref{xtheta}),
  we will also have that $x(\theta)\to 1$.
  To state the theorem, we introduce the following  
function  
which will play a key role in the rate at which $\upalpha$ and  
$\upbeta$ 
can tend to infinity,  

$$
\mathfrak g(x,y):=\frac{y}{x (x+y)}\qquad {\rm for}\ x>0, y>0.  
$$
This function will appear from the difference between the trigamma 
functions 
$\Psi_1(\theta+\upalpha)-\Psi_1(\theta+\upalpha+\upbeta)$. 
We will assume that the parameters $\upalpha$ and $\upbeta$ depend on
the terminal time $t$,
and we will denote them by $(\upalpha_t,\upbeta_t)$: this means that
the random environment is given by space-time i.i.d. beta distributed random
variables of parameters $(\upalpha_t,\upbeta_t)$. We will assume that

\begin{equation}
  \label{as1}
\lim_{t\to\infty}\upalpha_t=\infty, \qquad \lim_{t\to\infty}\upbeta_t=\infty
\end{equation}
and that

\begin{equation}
  \label{as2}
  \lim_{t\to\infty}t \mathfrak g(\upalpha_t,\upbeta_t)=\infty.
\end{equation}
   In the following theorem, the quantities $x(\theta)$ and 
  $\sigma(\theta)$ will be time dependent since they will be evaluated 
  using time-dependent  parameters $(\upalpha_t,\upbeta_t)$ for 
  the beta random walk. It will turn out that under conditions
  (\ref{as1}) and (\ref{as2}), we will have that

  \begin{equation}
    \label{txt}
tx(\theta)\sim t-C(\theta)t\mathfrak g(\upalpha_t,\upbeta_t),
\end{equation}
where $C(\theta)>0$ is a constant depending on $\theta$, with the
correction term $t\mathfrak g(\upalpha_t,\upbeta_t)=o(t)$ but
$t\mathfrak g(\upalpha_t,\upbeta_t)\to\infty$ as $t\to\infty$. This
ensures that the entropy tends to $\infty$ in the sense that the number of trajectories involved in the probability
$P_\omega(t,x(\theta)t)$ will tend to $\infty$ as $t\to\infty$. Furthermore, we will also have under (\ref{as2}) that
$\sigma(\theta)\sim C'(\theta)\mathfrak g(\upalpha_t,\upbeta_t)$, for
some constant $C'(\theta)>0$, so that
$$
t^{1/3}\sigma(\theta)\to\infty.
$$

  \medskip
  \begin{theorem}
    \label{int-dis} Consider a family of Beta random walks of
    parameters $(\upalpha_t,\upbeta_t)$.  Assume that conditions
    (\ref{as1}) and (\ref{as2}) are satisfied. Then, for all
    $\theta\in (0,0.5)$ we have that 

  $$
\lim_{t\to\infty}\mathbb 
P_{\upalpha_t,\upbeta_t}\left(\frac{\log\left(P_\omega(t,x(\theta)t)\right)+I(x(\theta))t}{t^{1/3}\sigma(\theta)}\le 
  y 
  \right)=F_{GUE}(y). 
  $$
  \end{theorem}

  \medskip

  \noindent 
In the above theorem, the target point $x(\theta)$ is approaching $1$
as $t\to\infty$, through its dependence on time from $\upalpha_t$ and
$\upbeta_t$ (see (\ref{txt})). If we want to achieve a result where
$\upalpha_t\to\infty$, $\upbeta_t\to\infty$, but the target point $x(\theta)$
is fixed, we would need to choose also the parameter $\theta$ as a
function of time, so that $\theta_t\to\infty$ as $t\to\infty$. For the
moment, this regime is out of the reach of the methods used in this
article because the deformations of contours used are justified only
for $\theta<\min\{1/2,\upalpha+\upbeta\}$.

In the particular case in which
  $\upalpha_t=c_1t^r$, $\upbeta_t=c_2t^s$, for some constants
 $c_1>0$ and $c_2>0$, and the following 
  condition is satisfied, 

  \begin{equation}
    \label{gcond}
    r+\max(r-s,0)<1,
  \end{equation}
  then assumptions (\ref{as1}) and (\ref{as2}) are satisfied. If also
  $r=s$, the law of the Beta distribution with parameter
  $(\upalpha_t,\upbeta_t)$
  converges to the atom at $c_1/(c_1+c_2)$.

  \noindent The proof of Theorems \ref{one} and \ref{int-dis} is based on an exact 
  formula for the Laplace transform of the probabilities 
  $P_\omega(t,x)$
   from where an asymptotic analysis through 
  steep-descent methods of integrals on complex contours is carried 
  out. The asymptotic analysis approach is within the spririt of what
  is  carried out 
  in the context of the log-gamma polymer in \cite{BCR13}, the 
  O'Connell-Yor semi-discrete polymer and the continuum random polymer 
  in \cite{BCF14} and 
  for the Beta random walk (or Beta polymer) in \cite{Barr-Cor} and 
  \cite{K21}
  for $\upalpha=1$ and $\upbeta>0$. For the proof of Theorem
  \ref{one},
  we  have to do some challenging computations involving the polygamma 
  functions, leading to the required steep-descent estimates
  and involving the following complex function

\begin{equation}
  \label{hz}
h(z):=I(x(\theta))z+\frac{1-x(\theta)}{2}
\log\left(\frac{\Gamma(\upalpha+z)}{\Gamma(z)}\right) 
+\frac{1+x(\theta)}{2}\log\left(\frac{\Gamma(\upalpha+z)}{\Gamma(\upalpha 
    +\upbeta+z)}\right), 
\end{equation}
which turns out to have a critical point at $z=\theta$. We then do
a steep-descent analysis near this critical point, so that the choice
of
contours near this point will be important. Here we choose them
in a way similar to the choice of \cite{K21}. It appears that this
approach is robust enough to extend the validity of the convergence
to the GUE Tracy-Widom of the second order correction of the large
deviation probabilities to the range of parameters stated in Theorem
\ref{one}. Nevertheless, the first limitation of this approach
 is the impossibility of deforming contours to achieve
 a circular contour centered at the
origin of radius larger than $1/2$, restricting the range
of parameters to $\theta\in (0,1/2)$.
The second limitation  shows up
in the necessity of proving two key inequalities involving $h$ and
its derivatives, which ensure the steep-descent properties with the
chosen contours, which actually turn out to be false for some
values of $\upalpha>0$, $\upbeta>0$ and $\theta\in (0,0.5)$. At any
rate, we expect that an approach which enables us to go to values of
$\theta\ge 0.5$, might be possible with
 a clever choice of contours,  and could give an extension of Theorem
\ref{one}
to all $\theta>0$, $\upalpha>0$ and $\upbeta>0$.

  For Theorem \ref{int-dis}, we have to keep track of the dependence of
  several estimates on $\upalpha_t$ and $\upbeta_t$, and essentially
  the time scale changes from $t$ to $\sigma t^3$, but now $\sigma$
  depends
  on time through $\upalpha_t$ and $\upbeta_t$.
  
  Our final result extends Theorem \ref{int-dis} to random walks in
  space time i.i.d. environments which are close to the Beta random
  walk.
  To state it, we need to define the concept of matching moments 
  between two families of environments, one of which has Beta 
  distributions. 
  Given an environment $\omega$, define $\xi=(\xi(x,t))_{x\in\mathbb
    Z,t\ge 0}$ by
$\xi_+(x,t)=-\log \omega_1(x,t)$ and
$\xi_-(x,t) =-\log\omega_{-1}(x,t)$, so that
\[
e^{-\xi_+(x,t)} = \omega_1(x,t) \text{\quad and\quad}e^{-\xi_-(x,t)} = \omega_{-1}(x,t).
\]
We will also use the multi-index notation for $t=(t_1,t_2)\in\mathbb R^2$,

$$
|\alpha|=\alpha_1+\alpha_2, \qquad t^\alpha=t_1^{\alpha_1}t_2^{\alpha_2}.
$$

\medskip

\begin{definition}[Moment matching condition]
  Consider a family of Beta probability measures 
$(\mathbb P_{\upalpha_t,\upbeta_t})_{t\ge 0}$
such that (\ref{as1}), (\ref{as2}) are satisfied. Assume also that 

$$
M_1:=\lim_{t\to\infty}\mathbb E_{\upalpha_t,\upbeta_t}[\xi_+(0,t)]
=-\lim_{t\to\infty}\left(\Psi(\upalpha_t)-\Psi(\upalpha_t+\upbeta_t)\right) 
$$
and 

$$
M_2:=\lim_{t\to\infty}\mathbb E_{\upalpha_t,\upbeta_t}[\xi_-(0,t)]
=-\lim_{t\to\infty}\left(\Psi(\upbeta_t)-\Psi(\upalpha_t+\upbeta_t)\right) 
$$
exist. Define

$$
\bar\xi_+(x,t)=\bar\xi_1(x,t):=\xi_+(x,t)-M_1\quad{\rm and}\quad
\bar\xi_-(x,t)=\bar \xi_2(x,t):=\xi_-(x,t)-M_2.
$$
Let $f(t):[0,\infty)\to [0,\infty)$. Given a family of probability measures
$(\mathbb P_t)_{t\ge 0}$ defined on the environmental space
$\Omega$ (with corresponding expectations $(\mathbb E_t)_{t\ge 0}$),
we say that it {\it matches moments
up to order $k$ at rate $f$} with the family $(\mathbb
P_{\upalpha_t,\upbeta_t})_{t\ge 0}$,
if we have that
for all $x\in\mathbb Z$ and $t\ge 1$,
\[
\big\vert\mathbb E_t[\bar\xi^\alpha (x,t)]-\mathbb
E_{\upalpha_t,\upbeta_t}[\bar\xi^\alpha (x,t)] \big\vert \le f(t),
\qquad\text{for } |\alpha|\le k,
\]
and
\[
  \big\vert\mathbb E_t[\bar\xi^\alpha(x,t)]\big\vert \le f(t),
  \qquad\text{for } |\alpha|= k.
\]
\end{definition}

\medskip
We can now state the third result of this article.
  Again, the quantities $x(\theta)$ and 
  $\sigma(\theta)$ will be time dependent through the time-dependent  parameters $(\upalpha_t,\upbeta_t)$ for 
  the beta random walk.
  To state this theorem, we define for time $t\ge 0$, the set 

$$
D_t:=\{z\in\mathbb Z^2: |z|_1\le t\}. 
$$
Note that the convex hull of the range of the Dirichlet random walk at 
time $t$ in $\mathbb Z^2$ is 
$D_t$. Define its boundary 

$$
\partial D_t:=\{z\in\mathbb Z^2: |z|_1=t\}. 
$$
We then can define $\partial D_t^{+,+}:=\{z\in \partial D_t: z_1\ge 
0,z_2\ge 0\}$, $\partial D_t^{+,-}:=\{z\in \partial D_t: z_1\ge 
0,z_2\le 0\}$, $\partial D_t^{-,+}:=\{z\in \partial D_t: z_1\le 
0,z_2\ge 0\}$ and $\partial D_t^{-,-}:=\{z\in \partial D_t: z_1\le 
0,z_2\le 0\}$. We have that $\partial D_t=\partial D_t^{+,+}\cup 
\partial D_t^{+,-}\cup\partial D_t^{-,+}\cup\partial D_t^{-,-}$. 
For $t\ge 0$ and $y\in (0,1)$ consider the set 

$$
A_{t,y}:=\{z\in\mathbb Z^2: |z|_1=t, z_1\ge (t+y)/2, z_2\ge 0\}\subset\partial D^{+,+}_t. 
$$
See Figure \ref{figure0} to visualize the sets $D_t$, $\partial D_t$, 
$\partial D_t^{+,+}$, $\partial D_t^{+,-}$, $\partial D_t^{-,+}$, 
$\partial D_t^{-,-}$ and $A_{t,x(\theta)t}$. 
We will adopt the notation 

\begin{equation}
  \label{adopt}
P_{0,\omega}(t,y)=P_{0,\omega}(X_t\in A_{t,y}). 
\end{equation}
\medskip

\begin{theorem}
  \label{moments} Consider a family 
  of parameters  $(\upalpha_t,\upbeta_t)$ that 
satisfy (\ref{as1}) and (\ref{as2}). Let $\theta\in (0,0.5)$. Let 
$(\mathbb P_t)_{t\ge 0}$ be a family of environmental laws 
which  matches moments up 
to order $k$ at rate $\upalpha_t^{-\left\lceil\frac{k}{2}\right\rceil}$ with  $(\mathbb 
P_{\upalpha_t,\upbeta_t})_{t\ge 0}$ and such that

\begin{equation}
  \label{tta}
\lim_{t\to\infty}\frac{t^2\upalpha^{-\left\lceil\frac{k}{2}\right\rceil}_t}{\sigma(\theta)t^{1/3}}=0.
  \end{equation}
  Then

  $$
\lim_{t\to\infty}\mathbb 
P_t \left(\frac{\log\left(P_{0,\omega}(t,x(\theta)t)\right)+I(x(\theta))t}{t^{1/3}\sigma(\theta)}\le 
  y 
  \right)=F_{GUE}(y). 
  $$

  \end{theorem}

\medskip

\noindent Our approach to prove Theorem \ref{moments} is to
express the weak convergence in terms of the convergence of
a large enough family of expectations involving smooth enough
functions,
and then making a Taylor expansion, and apply Theorem \ref{int-dis}.
This is similar to what is presented in \cite{KQ18}, although here we
have
to deal with perturbations which involve two parameters instead of
one.

Theorem \ref{moments} has the following  corollary for
random walks in Dirichlet random environment in dimension $d=2$.
  As in Theorem \ref{moments}, the quantities $x(\theta)$ and 
  $\sigma(\theta)$ will be time dependent through
  time-dependent parameters $(\upalpha_t,\upbeta_t)$ and
  $P_{0,\omega}(t,x(\theta)t)$
  is defined in (\ref{adopt}). 

  \medskip
\begin{corollary}
  \label{corollary1} Consider a family of random walks
  in Dirichlet environment on $\mathbb Z^2$ of parameters
  $(\upalpha_t)_{t\ge 0}$ with
  $\upalpha_t=(\upalpha_{t,i})_{i\in\{1,\ldots, 4\}}$
  and $\upalpha_{t,1}=\upalpha_{t,2}=t^r$,
  $\upalpha_{t,3}\le t^{-p}$ and $\upalpha_{t,4}\le t^{-p}$ for
  $t\ge 1$,
  for some $r\in (0,1)$ and

\begin{equation}
  \label{as3}
p\ge r \left\lceil\frac{5}{3r}-\frac{1}{3}\right\rceil -r. 
\end{equation}
Then for all $\theta\in (0,0.5)$ we have that

    $$
\lim_{t\to\infty}\mathbb 
P_{\upalpha_t}\left(\frac{\log\left(P_{0,\omega}(t,x(\theta)t)\right)+I(x(\theta))t}{t^{1/3}\sigma(\theta)}\le 
  y 
  \right)=F_{GUE}(y). 
  $$

  \end{corollary}
  \medskip

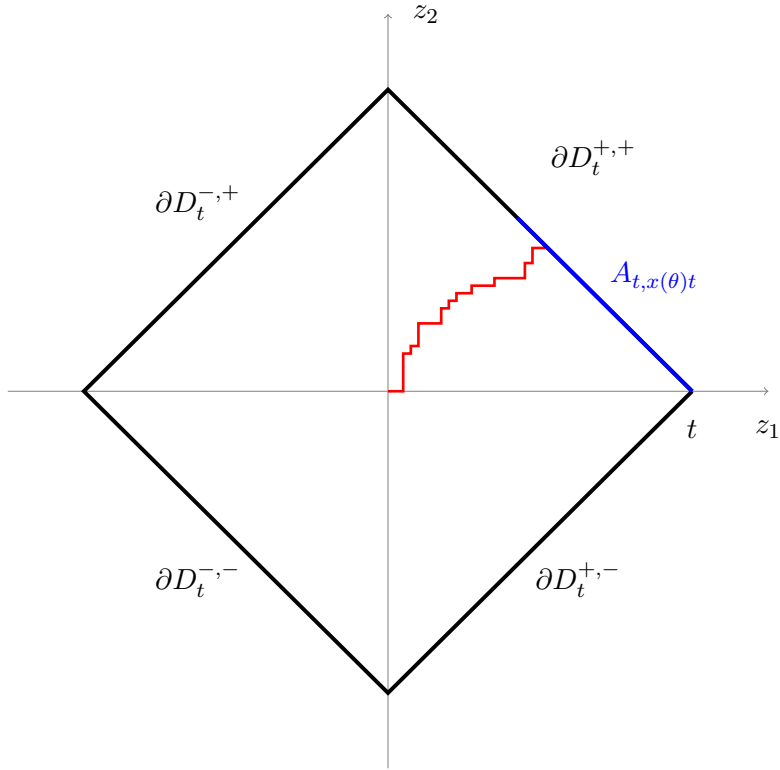
\begin{figure}[!h]

  \begin{tikzpicture}

\draw[help lines,->] (-5,0) -- (5,0) coordinate (xaxis);
\draw[help lines,->] (0,-5) -- (0,5) coordinate (yaxis);



\path[draw, red, line width=1pt,postaction=decorate]
(0,0)--(0.1,0)--(0.2,0)--(0.2,0.1)--(0.2,0.2)--(0.2,0.3)--(0.2,0.4)--(0.2,0.5)--(0.3,0.5)--(0.3,0.6)--(0.4,0.6)--(0.4,0.7)--(0.4,0.8)--(0.4,0.9)--(0.5,0.9)--(0.6,0.9)--(0.7,0.9)--(0.7,1.0)--(0.7,1.1)--(0.8,1.1)--(0.8,1.2)--(0.9,1.2)--(0.9,1.3)--(1.0,1.3)--(1.1,1.3)--(1.1,1.4)--(1.2,1.4)--(1.3,1.4)--(1.4,1.4)--(1.4,1.5)--(1.5,1.5)--(1.6,1.5)--(1.7,1.5)--(1.8,1.5)--(1.8,1.6)--(1.8,1.7)--(1.9,1.7)--(1.9,1.8)--(1.9,1.9)--(2,1.9)--(2.1,1.9); 

\path[draw,line width=1.5pt,postaction=decorate]
(4,0) -- (0,4)--(-4,0)--(0,-4)--(4,0);

\path[draw,blue, line width=1.5pt,postaction=decorate]
(1.7,2.3)--(4,0);

\node at (4,-0.5) {$t$};
\node at (5,-0.5) {$z_1$};
\node at (0.5,5) {$z_2$};
\node at (2.7,3.1) {$\partial D_t^{+,+}$};
\node[blue] at (3.5,1.5) {$A_{t,x(\theta) t}$};
\node at (-2.5,-2.5) {$\partial D_t^{-,-}$};
\node at (2.5,-2.5) {$\partial D_t^{+,-}$};
\node at (-2.5,2.5) {$\partial D_t^{-,+}$};
\end{tikzpicture}

\caption{The set $D_t$ and its boundary $\partial D_t$ divided into
  the four quadrants, the large deviation event $A_{t,x(\theta)t}$ (in
  blue) of Theorem \ref{moments} and Corollary
  \ref{corollary1} and a typical directed trajectory (in red) of a random
  walk starting from $0$. In the case of Corollary \ref{corollary1},
  the length of the blue segment $A_{t,x(\theta)t}$ grows like
  $O(t^{1-r})$ as $t\to\infty$.}
\label{figure0}
\end{figure}

Corollary \ref{corollary1} is a particular case of perturbations of random walks
in i.i.d. random environment in dimension $d=2$ which are
perturbations
of the Beta random walk (see Figure \ref{figure0}), and can be derived
from Theorem \ref{moments} comparing the law of the Dirichlet random
walk (for directed trajectories) with the law of a Beta random walk
with specific parameters
$(\upalpha_t,\upbeta_t)$ (see Lemma \ref{parametersbeta} for the choice of these parameters).  Its proof requires obtaining good enough
estimates
for the moments of Dirichlet random variables.
It should be emphasized that  even though 
the result concerns a random walk in a static i.i.d. environment, the 
random walk is  not allowed to backtrack simply because the event
$\{X_t\in A_{t,x(\theta)t}\}$ contains only directed trajectories. On the other hand, if
$\upalpha_{t,3}$ and $\upalpha_{t,4}$ decay to $0$ exponentially fast, it would be
possible to state a variation of this corollary where the large
deviation event includes trajectories which could backtrack . Nevertheless,
in a sense, this situation would correspond to a case in which these
trajectories
do not contribute at all to the large deviations, nor to its
correction. It might  be
possible to extend the corollary to large deviation events which
genuinely allow backtracking to some extent, but this would be
a task for a separate article and it would require different
methods
besides the moment matching technique.

In Section \ref{sthree}, we present the proof of Theorems \ref{one}
and \ref{int-dis} in a unified way. In Section \ref{sfour}, we will prove
Theorem
\ref{moments}. Throughout the rest of this article we will
adopt the notation $C_1,C_2,\ldots,$ and $c_1,c_2,\ldots,$
to denote constants, which in general might depend on
some of the parameters involved. In particular, they might
depend on $\upalpha$ and $\upbeta$, and it will be important
to keep track of this dependence in order to prove Theorem \ref{int-dis}.

  \medskip
  \section{Proof of Theorems \ref{one} and \ref{int-dis}}
  \label{sthree}
To prove Theorems \ref{one} and \ref{int-dis} we will
do a careful steep-descent analysis, whose starting point is
the connection between the Laplace transform and the
distribution functions. This connection is explained in Section
\ref{sublaplace}. In Section \ref{subdet}, we present an
exact determinantal formula derived in \cite{Barr-Cor},
which will then lend itself to do an asymptotic analysis.
In Section 
\ref{he}
we explain the general strategy to do the aymptotic analysis of the 
determinantal formula to prove Theorems \ref{one} and \ref{int-dis},
summarizing the main step as Proposition \ref{p1}.
In Section \ref{sde}, several estimates will be proven showing that
the contours involved have the steep-descent property,
starting from Lemmas \ref{fo} and \ref{im}. This will be
applied
in Section \ref{sde} to obtain the necessary bound on the integrands
over these complex contours. In Section \ref{sp1},
Proposition
\ref{p1} will be proven. Finally in Section \ref{lemmasfo}, Lemmas \ref{fo}
and \ref{im} are proven.

\subsection{Connection to the Laplace transform}
\label{sublaplace}
The main conclusion of this section will be the following lemma
showing how to obtain the limiting distribution function of the
fluctuations of $\log P_\omega(t,x(\theta)t)$ through its Laplace
transform. Part $(i)$ of the Lemma is  stated in \cite{Barr-Cor},
although it is a standard method already used in similar contexts
by other authors (se for example \cite{BCR13} or \cite{BCF14}).

For $y\in \mathbb R$ define

\begin{equation}
  \label{uy}
u(y):=-e^{t I(x(\theta))-t^{1/3}\sigma(\theta)y}. 
\end{equation}
Our question is under what assumptions on the parameters
$(\upalpha_t,\upbeta_t)_{t\ge 0}$  is
the following equation valid:

\begin{align}
 \nonumber 
&\lim_{t\to\infty}\mathbb E_{\upalpha_t,\upbeta_t}\left[e^{u(y) 
    P_\omega(t,x(\theta)t)}\right]\\
& \label{aabbb}=
\lim_{t\to\infty}\mathbb P_{\upalpha_t,\upbeta_t}\left(
  \frac{\log\left(P_\omega(t,x(\theta)t)\right)+I(x(\theta))t}{t^{1/3}\sigma(\theta)}\le
                                       y\right), 
\end{align}
The statement (\ref{aabbb}) says that also the right-hand side limit
exists. But we would also like to ensure that this limit corresponds
to
some convergence in distribution.
We will adopt the convention that for the case of a Beta
distribution
with fixed parameters $(\upalpha,\upbeta)$, $\upalpha_t=\upalpha$ and $\upbeta_t=\upbeta$
for all $t\ge 0$.

Let us explain how we will prove (\ref{aabbb}).
 We will need the
 following lemma whose proof we omit (see for example  \cite{Bor-Cor}). 

\medskip 

\begin{lemma}
  \label{lbc}
  Consider a sequence of functions 
  $(f_t)_{t\in\mathbb N}$ mapping $\mathbb R\to[0,1]$ such that

\begin{itemize}

\item[(i)] For each $t$, 
  $f_t(x)$ is strictly decreasing in $x$. 

\item[(ii)] For each $t$, $\lim_{x\to-\infty}f_t(x)=1$ and 
$\lim_{x\to\infty}f_t(x)=0$. 

\item[(iii)] For each $\delta>0$, on $\mathbb 
  R\backslash [-\delta,\delta]$, $f_t$ converges uniformly to 
  $1(x\le 0)$. 

\end{itemize}
 Define the $r$-shift of $f_t$ as 
  $f_t^r(x)=f_t(x-r)$. Consider 
  a sequence of random variables $X_t$ such that for each $r\in\mathbb 
  R$, 

  $$
\lim_{t\to\infty}E[f^r_t(X_t)]=p(r), 
$$
and assume that $p(r)$ is a continuous probability distribution 
function. 
Then $X_t$ converges weakly in distribution to a random variable $X$
which is distributed according to $P(X\le r)=p(r)$. 
  \end{lemma}

\medskip

\noindent We will apply Lemma \ref{lbc} to prove  (\ref{aabbb})
making a specific choice of functions and random variables. Let $\theta>0$. Consider the family of functions $(f_t)_{t\in\mathbb N}$ defined by 

\begin{equation}
  \label{ff}
f_t(w):=e^{-e^{w t^{1/3}\sigma(\theta)}}. 
\end{equation}
In the case in which $\upalpha$, $\upbeta$ and $\theta$ are fixed and
satisfy
the assumptions $\upalpha>0$, $\upbeta>0$ and $\theta\in (0,\min\{0.5,0.72\times\upalpha\})$, it is obvious that this family 
satisfies conditions $(i)$, $(ii)$ and $(iii)$ of Lemma \ref{lbc}.
For the case in which $(\upalpha_t)_{t\ge 0}$ and $(\upbeta_t)_{t\ge
  0}$ satisfy (\ref{as1}) and (\ref{as2}), by parts $(iv)$ and $(v)$
of Corollary \ref{corcor}
(which is stated and proven in Section \ref{polyest})
we have that

\begin{equation}
  \label{limti}
\lim_{t\to\infty}t\sigma(\theta) ^3=\infty,
\end{equation}
so that (\ref{ff}) still satisfies $(i)$, $(ii)$ and $(iii)$ of Lemma \ref{lbc}.
Let $(X_t)_{t\in\mathbb N}$ be defined by 

\begin{equation}
  \label{frv}
X_t=\frac{\log P_\omega(t,x(\theta)t)+I(x(\theta))t}{t^{1/3}\sigma(\theta)}. 
\end{equation}
Consider now the statement

\begin{equation}
    \label{ll}
  \lim_{t\to\infty}E[f^r_t(X_t)]=\lim_{t\to\infty}\mathbb 
  E_{\upalpha_t,\upbeta_t}\left[
    e^{u(r) P_\omega(t,x(\theta)t)}\right]=\lim_{t\to\infty}\mathbb 
  P_{\upalpha_t,\upbeta_t}(X_t\le r).
  \end{equation}
  The following proposition summarizes what we have explained
  in relation to (\ref{aabbb}), and will be proven below. 
\medskip 

\begin{proposition}
  \label{con}
Consider the family of functions $(f_t)_{t\in\mathbb N}$
  defined  
  in (\ref{ff}) and the random variables $(X_t)_{t\in\mathbb N}$
  defined in (\ref{frv}).

  \begin{itemize}
  \item[(i)]
      For all $\upalpha>0$, $\beta>0$ and 
  $\theta\in (0,\min\{0.5,0.72\times\upalpha\})$, we have that 
 (\ref{ll}) is satisfied,
  and  the limit  exists for all $r$ and $(X_t)_{t\in\mathbb N}$ converges 
  in distribution to some random variable $X$.

  \item[(ii)] Let $(\upalpha_t)_{t\ge 0}$ and $(\upbeta_t)_{t\ge 0}$
    satisfy (\ref{as1}) and (\ref{as2}). Then,  (\ref{ll}) is satisfied,  
   the limit  exists for all $r$ and $(X_t)_{t\in\mathbb N}$ converges 
  in distribution to some random variable $X$.

  \end{itemize}
\end{proposition}
\medskip 

\noindent The first equality in (\ref{ll}) under the assumption of
parts $(i)$
 or $(ii)$   of Proposition 
\ref{con} is immediate from the definitions. The second statement of
parts $(i)$ and $(ii)$ Proposition 
\ref{con} will be proven in Section \ref{sp1}.

\medskip

\subsection{Determinantal formula for the Laplace transform}
\label{subdet}
The second step in the proof of Theorem \ref{one} will be the
following exact formula for the Laplace transform of $P(t,x)$ proved in
\cite{Barr-Cor}. One way of deriving this formula is to obtain
integral formulas for the moments of $P(t,x)$ using a variant of the
Bethe ansatz and a
non-commutative binomial identity, throug a recurrent
system of equations for the moments of $P(t,x)$. 

\medskip
\begin{theorem}[Barraquand-Corwin, 2017]
  \label{bc17}
For $u\in\mathbb C\backslash \mathbb R_{>0}$, fix $t\in\mathbb Z_{\ge 0}$,
$x\in\{-t,\ldots,t\}$ with the same parity,
and $\upalpha,\upbeta>0$. Then one has

$$
E\left[e^{uP(t,x)}\right]=\det(I-K^{RW}_u)_{L^2(C_0)},
$$
where $C_0$ is a small positively oriented circle containing $0$ but
not $-\upalpha-\upbeta$ nor $-1$,
and $K^{RW}_u:L^2(C_0)\to L^2(C_0)$ is defined by its integral kernel

\begin{equation}
  \label{kernel}
K_u^{RW}(v,v')=\frac{1}{2\pi 
  i}\int_{1/2-i\infty}^{1/2+i\infty}\frac{\pi}{\sin(\pi s)}(-u)^s 
\frac{g^{RW}(v)}{g^{RW}(v+s)}\frac{ds}{s+v-v'},
\end{equation}
where

$$
g^{RW}(v)=\left(\frac{\Gamma(v)}{\Gamma(\alpha+v)}\right)^{(t-x)/2}
\left(\frac{\Gamma(\alpha+\beta+v)}{\Gamma(\alpha+v)}\right)^{(t+x)/2}\Gamma(v).
$$
\end{theorem}

\medskip

\subsection{Asymptotic analysis of the Laplace transform}
\label{he}
Here we will show how to finish the proof of Theorems \ref{one} and \ref{int-dis},
proving
the convergence of the Laplace transform of the normalized
fluctuations
of $\log P_\omega(t,x(\theta)t)$ to the GUE Tracy-Widom distribution.
From Proposition \ref{con}, we can see that
to prove these theorems, we can use the determinantal
formula of Theorem \ref{bc17}. The asymptotic steep-descent
analysis that will be later used will be similar to
the proof of Theorem 5.2
of \cite{Barr-Cor} or Theorem 2.1 of \cite{K21}. Nevertheless,
to prove the necessary steep-descent properties,
a technical analysis of high complexity involving the polygamma
functions must be executed, both to deal with
 arbitrary values of $\upalpha>0$ and
 $\upbeta>0$ in  Theorem \ref{one} (specially in order to
 achieve the range $\theta\in (0,0.5)$ of validity for all
 $\upalpha\ge 0.7$ and $\upbeta>0$), and to deal
 with values of $\upalpha$ and $\upbeta$ tending to
 $\infty$ in Theorem \ref{int-dis}.

We first rewrite the kernel $K^{RW}_u(v,v')$ choosing $u=u(y)$
as in (\ref{uy}) so that

\begin{equation}
  \label{ku}
K_u^{RW}(v,v')=\frac{1}{2\pi 
  i}\int_{1/2-i\infty}^{1/2+i\infty}\frac{\pi}{\sin(\pi (z-v))}
e^{t(h(z)-h(v))-t^{1/3}\sigma(\theta) y(z-v)}\frac{\Gamma(v)}{\Gamma(z)}
\frac{dz}{z-v'}.
\end{equation}
where we recall the definition of $h$ in (\ref{hz}).
It can be checked that  $\theta$ is a critical point
for the function $h$, so that $h'(\theta)=h''(\theta)=0$. We hence
follow the steepest-descent method deforming the integration
contours so that they go across this critical point.
As in \cite{Barr-Cor}, we deform the contour $C_0$
in Theorem \ref{bc17} (the Fredholm contour) to the contour

\begin{equation}
\label{ct}
C_\theta:=\{z\in\mathbb C: |z|=\theta\},
\end{equation}
defined as the circle centered at $0$ with radius $\theta$,
and the vertical line in $\mathbb C$ passing through $1/2$
in the definition of the kernel (\ref{kernel}) to

\begin{equation}
\label{dt}
D_\theta=\{\theta+iy:y\in\mathbb R\}.
\end{equation}
There is no problem in doing this as long as we avoid the poles, which
happens if

\begin{equation}
\label{tm}
\theta<\min\{\upalpha+\upbeta,1/2\}.
\end{equation}
Let also for $\epsilon>0$, $B(\theta,\epsilon)$ be the ball
of radius $\epsilon$ centered at $\theta$ in $\mathbb C$, 

\begin{equation}
\label{ce}
C^\epsilon_\theta:=C_\theta\cap B(\theta,\epsilon),
\end{equation}
 the part of the contour $C_\theta$ inside
$B(\theta,\epsilon)$,
while 

\begin{equation}
\label{de}
D^\epsilon_\theta:=D_\theta\cap B(\theta,\epsilon),
\end{equation}
 the part of the contour $D_\theta$
inside $B(\theta,\epsilon)$.
As it is standard in several results proving
convergence to the GUE Tracy-Widom distribution (see for example
\cite{Barr-Cor}  or \cite{K21}), the proof will be based in steep-descent
methods
for the complex contours. We summarize this in a proposition that
we state below.
To state it, we will modify the the contour
$C_\theta^\epsilon$, to a contour which is  a piece of a wedge, which in the limit becomes
a full wedge defined by a certain angle $\phi$. Let us explain how to
choose $\phi$. For $L>0$ define the contour for arbitrary $\phi\in (\pi/6,\pi/2)$,

$$
W^L_\theta:=\left\{\theta+|y|e^{i(\pi-\phi) {\text{sgn}}(y)}:y\in [-L,L]\right\}.
$$
Note that $C_\theta^\epsilon$ is an arc of circle which crosses
$\theta$ vertically. We now choose $L$ and $\phi$ so that the
endpoints
of $W^L_\theta$ and of $C_\theta^\epsilon$ coincide. Then note that for
$\epsilon$
small enough we can replace $C_\theta^\epsilon$ by $W^L_\theta$.
We will also introduce the extension of the piece of wedge $W_\theta^L$
outside of the ball $B(\theta,\epsilon)$ through the contour

$$C_\theta^{\epsilon,+}:=C_\theta-B(\theta,\epsilon),
$$
defined as (see Figure \ref{figure1})

$$
V_\theta^\epsilon:=W_\theta^L\cup C_\theta^{\epsilon,+}.
$$

\begin{figure}[!h]

\begin{tikzpicture}[decoration={markings, 
mark=at position 7cm with {\arrow[line width=3pt]{>}}
}
]
\draw[help lines,->] (5,0) -- (12,0) coordinate (xaxis);



\path[draw, dashed, gray, line width=1pt,postaction=decorate]
(8,-3.4641016) arc (-60:60:4); 

\path[draw,line width=1.5pt,postaction=decorate]
(8,-3.4641016) arc (-60:-30:4)-- (10,0) node[below right]
{$\theta$} 
--(9.464101615,2) 
arc(30:60:4);

\node at (3,2.5) {$\ $}; 
\node at (8,3) {$C_\theta^{\epsilon,+}$}; 
\node at (9,1) {$W^L_\theta$}; 
\node at (8,-3) {$C_\theta^{\epsilon,+}$}; 
\end{tikzpicture}

\caption{The contour $V_\theta^\epsilon$ composed by the union of the 
  contours $C_\theta^{\epsilon,+}$ and $W^L_\theta$, and in dashed 
  line the contour 
  $C_\theta^\epsilon$.}
\label{figure1}
\end{figure}
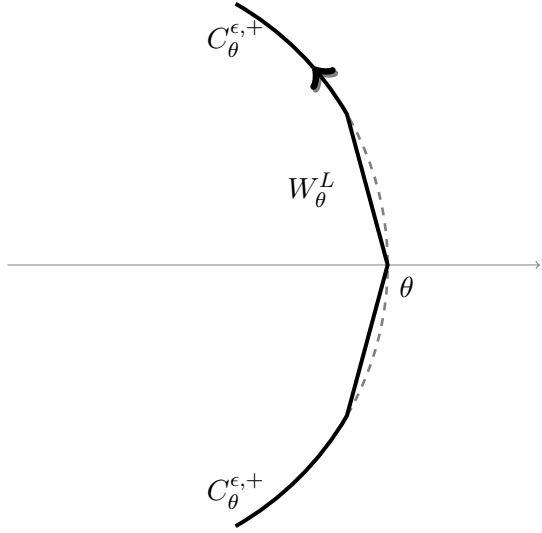

\medskip

With the
choice of $\phi$ explained above,  we also define

\begin{equation}
  \label{cy}
W_\theta^\infty:=\left\{\theta+ |y| e^{i(\pi-\phi)
    {\text{sgn}}(y)}:y\in\mathbb R\right\}. 
\end{equation}
We will need to define for $y\in\mathbb R$ the kernel  $K_y$,

  $$
  K_y(w,w'):=\frac{1}{2\pi i}\int_{e^{-i\pi/3}\infty}^{e^{i\pi/3}\infty}
  \frac{1}{(z-w')(w-z)}\frac{e^{z^3/3-yz}}{e^{w^3/3-yw}}dz, 
  $$
  where the contour for $z$ is a wedge-shaped contour constituted of 
  two 
  rays going to infinity in the directions $e^{-i\pi/3}$ and $e^{i\pi 
    /3}$
  that do not intersect $W_\theta^\infty$. 

\medskip
\begin{prop}
  \label{p1}
  Consider the equality
  
  \begin{equation}
    \label{liminfty}
  \lim_{t\to\infty}
 \det\left(I-K_u^{RW}\right)_{L^2(C_\theta)}
=\det (I+K_y)_{L^2(W_\theta^\infty)}. 
\end{equation}

    \begin{itemize}
    \item[(i)]
        For all $\upalpha>0$, $\beta>0$ and 
  $\theta\in (0,\min\{0.5,0.72\times\upalpha\})$, we have that 
        (\ref{liminfty}) is satisfied.

        \item[(ii)] Assume that $(\upalpha_t)_{t\ge 0}$ and
          $(\upbeta_t)_{t\ge 0}$ satisfy (\ref{as1}) and
          (\ref{as2}). Then
           for all
          $\theta\in (0,0.5)$,  (\ref{liminfty})
          is
          satisfied.
          \end{itemize}

    \end{prop}
  \medskip

  The proof of Proposition \ref{p1} will proceed in two steps. Firstly,
  we will truncate all the integrations to contours at a distance
  $\epsilon$
  of $\theta$ for some $\epsilon\in (0,\theta/2)$, showing that
  
  $$
\lim_{t\to\infty} \det\left(I-K_u^{RW}\right)_{L^2(C_\theta)}
=\lim_{t\to\infty} \det\left(I-K_{y,\epsilon}^{RW}\right)_{L^2(C^\epsilon_\theta)}, 
$$
where 

\begin{align}
  \nonumber 
&K_{y,\epsilon}^{RW}(v,v')\\
  \label{keps}
  &=\frac{1}{2\pi i}\int_{D^\epsilon_\theta}
\frac{\pi}{\sin (\pi (z-v))} e^{t(h(z)-h(v))-t^{1/3}\sigma(\theta) 
y(z-v)}\frac{\Gamma (v)}{\Gamma(z)}\frac{1}{z-v'}dz. 
\end{align}
Secondly, we will prove that

  $$
  \lim_{t\to\infty}
 \det\left(I-K_{y,\epsilon}^{RW}\right)_{L^2(C^\epsilon_\theta)}
=\det (I+K_y)_{L^2(C)}.
$$
The details of this proof will be given in Section \ref{sp1}.
From Proposition \ref{p1},
  Theorem \ref{one} now follows as in \cite{Barr-Cor} from the identity

  $$
\det (I+K_y)_{L^2(C)}=\det (I-K_{Ai})_{L^2(y,+\infty)},
$$
valid for $y\in\mathbb R$.
It remains to  prove Proposition \ref{p1}. We will first prove
  steep-descent estimates needed for the proof of Theorem \ref{one} and quantitative versions of them  needed for the proof of Theorem \ref{int-dis}.

  \medskip

  \subsection{Preliminary estimates involving the polygamma functions}
  \label{polyest}
  Here we will derive several estimates which will eventually
  prove the necessary steep-descent properties of the
  integrals along the complex contours in the Fredholm determinants
  and its kernel.

  Recall that the polygamma function of order $k$, for $k\ge 1$, is defined as

  $$
\Psi_k(z):=\frac{d^k}{dz^k}\Psi(z).
$$
  We start with several estimates involving the polygamma functions.
\medskip

\begin{lemma}
  \label{poly} The following inequalities are satisfied.

  \begin{itemize}

\item[(i)]  For all  $k\ge 1$ and $x>0$,

\begin{align}
  &\label{pk}
    k!  \left(\frac{1}{x^{k+1}}+\frac{1}{k}\frac{1}{(x+1)^k}\right) 
      \le (-1)^{k+1}\Psi_k(x)\le   k!    \left(\frac{1}{x^{k+1}}+\frac{1}{k}\frac{1}{x^k}\right). 
\end{align}

\item[(ii)]  For every $x>0$, $y>0$ and $k\ge 1$ we have that

  \begin{align}
 \nonumber
&   k! \mathfrak 
                         g(x+1,y)\left(\frac{1}{x^k}+\frac{1}{k}\frac{1}{(x+1)^{k-1}}\right)\\
\nonumber &  \le  k! \mathfrak 
                         g(x,y)\frac{1}{x^k}+(k-1)!\frac{1}{(x+1)^{k-1}}\mathfrak
    g(x+1,y)\\
\nonumber &                         \le (-1)^{k+1}(\Psi_k(x)-\Psi_k(x+y)) \\
     \label{inpsi1}      &     \le(k+1)! \mathfrak g(x,y)\left(\frac{1}{x^k}+\frac{1}{k+1}\frac{1}{x^{k-1}}\right). 
\end{align}

\end{itemize}
  \end{lemma}
  \begin{proof} 
    First note that for all $k\ge 1$,

    \begin{equation}
      \label{expansion}
\Psi_k(z)=(-1)^{k+1} k!\sum_{j=0}^\infty\frac{1}{(z+j)^{k+1}}.
\end{equation}

\medskip

\noindent {\it Proof of part $(i)$.} From 
(\ref{expansion}) we have that 

\begin{align*}
&(-1)^{k+1}\Psi_k(x)=
k!\left(\frac{1}{x^{k+1}}+\sum_{j=1}^\infty\frac{1}{(x+j)^{k+1}}\right) 
    \le 
    k!\left(\frac{1}{x^{k+1}}+\frac{1}{k}\frac{1}{x^k}\right). 
  \end{align*}
Similarly, 

\begin{align*}
&(-1)^{k+1}\Psi_k(x)=
k!\left(\frac{1}{x^{k+1}}+\sum_{j=1}^\infty\frac{1}{(x+j)^{k+1}}\right)\ge 
      k!\left(\frac{1}{x^{k+1}}+\frac{1}{k}\frac{1}{(x+1)^k}\right) 
  \end{align*}

\medskip
\noindent {\it Proof of part $(ii)$.}
From the upper bound of part $(i)$ note that

\begin{align*}
  &
    (-1)^{k+1}(\Psi_k(x)-\Psi_k(x+y)) 
    =\int_x^{x+y}(-1)^{k+2}\Psi_{k+1}(u)du\\
  &  \le 
    (k+1)!\int_x^{x+y}\left(\frac{1}{u^{k+2}}+\frac{1}{k+1}\frac{1}{u^{k+1}}\right)du \\
&    \le(k+1)! \mathfrak g(x,y)\left(\frac{1}{x^k}+\frac{1}{k+1}\frac{1}{x^{k-1}}\right). 
\end{align*}
For the lower bound,
  first note that

  $$
  \int_x^{x+y}u^{-(k+2)}du=
  \frac{1}{k+1}\left(\frac{1}{x^{k+1}}-\frac{1}{(x+y)^{k+1}}\right)
  \ge\frac{1}{k+1}\frac{1}{x^k}\mathfrak g(x,y).
  $$
  Hence,
  
\begin{align*}
  &
    (-1)^{k+1}(\Psi_k(x)-\Psi_k(x+y)) 
   \ge 
    (k+1)!\int_x^{x+y}\left(\frac{1}{u^{k+2}}+\frac{1}{k+1}\frac{1}{(u+1)^{k+1}}\right)du
  \\
  &\ge 
    (k+1)!\left(\frac{1}{k+1}\frac{1}{x^k}\mathfrak g(x,y)+\frac{1}{k(k+1)}\frac{1}{(x+1)^{k-1}}
    \mathfrak g(x+1,y)\right)\\
&  \ge k! \mathfrak g(x+1,y)\left(\frac{1}{x^k}+\frac{1}{k}\frac{1}{(x+1)^{k-1}}\right). 
\end{align*}

\end{proof}
\medskip

We will now apply Lemma \ref{poly} to derive several crucial
properties involving the function $h$ [cf. (\ref{hz})], which plays
a central role in the steep-descent analysis that will be made.
For the record, we write the
following expression for $h^{(k)}(\theta)$, valid when $k\ge 2$,
  \begin{align}
  \nonumber 
  &h^{(k)}(\theta)=\Psi_{k-1}(\theta+\upalpha)-\Psi_{k-1}(\theta+\upalpha+\upbeta)\\
  \label{fracp} 
    &
               +\frac{\Psi_1(\theta+\upalpha)-\Psi_1(\theta+\upalpha+\upbeta)}{\Psi_1(\theta)-\Psi_1(\theta+\upalpha 
               +\upbeta)}\left(\Psi_{k-1}(\theta+\upalpha+\upbeta)-\Psi_{k-1}(\theta)\right). 
\end{align}

\medskip
\begin{corollary}
 \label{corcor} 
  The following estimates are satisfied,

  \begin{itemize}

  \item[(i)] For all $\upalpha>0$ and $\upbeta>0$,

    $$
\Psi_1(\upalpha)-\Psi_1(\upalpha+\upbeta)\ge\mathfrak g(\upalpha+1,\upbeta).
    $$
  \item[(ii)] For all $\upalpha>0$, $\upbeta>0$, $\theta>0$ and $k\ge 1$, we have that

    $$
    |\Psi_{k}(\theta+\upalpha)-\Psi_{k}(\theta+\upalpha+\upbeta)|\le
   (k+1)!\mathfrak 
    g(\upalpha,\upbeta)\frac{1+\theta^{-1}}{\theta^{k-1}}. 
$$
    \item[(iii)] For all $\upalpha>0$, $\upbeta>0$, $\theta\in (0,0.5)$ and $k\ge 3$
      we have that

$$
|h^{(k)}(\theta)|\le     64\mathfrak
g(\upalpha,\upbeta)\frac{k!}{\theta^{k+2}}.
$$

\item[(iv)]
  For all $\upalpha_0>0$ there is a $\theta_0>0$
  such that for  $\upalpha\ge\upalpha_0$, $\upbeta>0$
  and $\theta\in (0,\theta_0)$ one has that 

  $$
  \frac{\sigma^3(\theta)}{2\theta}+\frac{h^{(4)}(\theta)}{4!}\ge 
  \oldconstant{11}(\theta,\upalpha_0)\mathfrak g(\upalpha,\upbeta). 
  $$
  where $\newconstant\label{11}(\theta,\upalpha_0)>0$ depends only 
  on $\theta$ and $\upalpha_0$.  Furthermore $\theta_0(\alpha_0)$
  can be chosen so that $\lim_{\alpha_0\to\infty}\theta_0=\infty$.

  \item[(v)] For all $\upalpha>0$, $\upbeta>0$ and $\theta>0$
    we have that

    \begin{equation}
    \label{vii}
\sigma^3(\theta)=\frac{1}{2}h^{(3)}(\theta)>0\quad{\rm and}\quad
h^{(4)}(\theta)<0.
\end{equation}

\end{itemize}
\end{corollary}
\begin{proof} {\it Proof of part $(i)$}. This is immediate from
  the lower bound of part $(ii)$ of Lemma \ref{poly}.

  \medskip

  \noindent  {\it Proof of part $(ii)$}.
  From the upper bound of part $(ii)$ of Lemma \ref{poly} note that

  \begin{align*}
&  |\Psi_{k}(\theta+\upalpha)-\Psi_{k}(\theta+\upalpha+\upbeta)|\le
   (k+1)!\mathfrak 
                      g(\theta+\upalpha,\upbeta)\frac{1+\theta^{-1}}{\theta^{k-1}}.
      \end{align*}
  \medskip

  \noindent {\it Proof of part $(iii)$}.
From  part $(ii)$ of Lemma \ref{poly} note that

\begin{equation}
  \label{gteta}
 g(\theta+1,\upalpha+\upbeta) \left(\frac{1}{\theta}+1\right)
\le
 \Psi_1(\theta)-\Psi_1(\theta+\upalpha+\upbeta)
\le g(\theta,\upalpha+\upbeta) \left(\frac{2}{\theta}+1\right).
  \end{equation}
  Combining this with part $(ii)$, we have that for $\theta\in (0,0.5)$,

  \begin{align}
    \nonumber
&  \frac{|  \Psi_{k-1}(\theta+\upalpha+\upbeta)-\Psi_{k-1}(\theta)|}
  {  \Psi_1(\theta)-\Psi_1(\theta+\upalpha+\upbeta)}
  \le
  (k!)\frac{1}{\theta^{k-1}}\frac{\mathfrak g(\theta,\upalpha+\upbeta)}{
                      \mathfrak g(\theta+1,\upalpha+\upbeta)}\\
    \label{2kf}
&  \le (k!)\theta^{-(k-1)} 4\theta^{-2}=4(k!)\theta^{-(k+1)}.
  \end{align}
  On the other hand, using this bound and (\ref{fracp}), we get that
  
  \begin{align}
  \nonumber 
  &|h^{(k)}(\theta)|=|\Psi_{k-1}(\theta+\upalpha)-\Psi_{k-1}(\theta+\upalpha+\upbeta)|\\
  \nonumber 
    &
               +\frac{\Psi_1(\theta+\upalpha)-\Psi_1(\theta+\upalpha+\upbeta)}{\Psi_1(\theta)-\Psi_1(\theta+\upalpha 
               +\upbeta)}\left|\Psi_{k-1}(\theta+\upalpha+\upbeta)-\Psi_{k-1}(\theta)\right|\\
 \nonumber 
 &
   \le 
|\Psi_{k-1}(\theta+\upalpha)-\Psi_{k-1}(\theta+\upalpha+\upbeta)|\\
\nonumber
  &   +\frac{\Psi_1(\upalpha)-\Psi_1(\upalpha+\upbeta)}{\Psi_1(\theta)-\Psi_1(\theta+\upalpha 
    +\upbeta)}\left|\Psi_{k-1}(\theta+\upalpha+\upbeta)-\Psi_{k-1}(\theta)\right|\\
    &
        \label{hkt}
    \le (k-1)!\mathfrak
    g(\upalpha,\upbeta)\frac{1}{\theta^{k-1}}
    +
      32\mathfrak g(\upalpha,\upbeta)
    \frac{k!}{\theta^{k+2}}\le
    64\mathfrak g(\upalpha,\upbeta)\frac{k!}{\theta^{k+2}},
  \end{align}
  where in the  second to last inequality we have used (\ref{2kf}) and
  part $(ii)$ again. This proves part $(iii)$.

  \medskip

\noindent {\it Proof of part $(iv)$}.      Dropping out the positive terms of the first two lines of the 
      above display we obtain that,

            \begin{align*}
        &\frac{\sigma^3(\theta)}{2\theta}+\frac{h^{(4)}(\theta)}{4!}          \ge 
          \frac{1}{2\theta}\left(\Psi_2(\theta+\upalpha)-\Psi_2(\theta+\upalpha+\upbeta) 
          \right) \\
        &
+        \frac{1}{4!}
  \frac{\Psi_1(\theta+\upalpha)-\Psi_1(\theta+\upalpha+\upbeta)}
          {\Psi_1(\theta)-\Psi_1(\theta+\upalpha+\upbeta)}
                    \Psi_3(\theta+\upalpha+\upbeta)\\
        &
-                    \frac{\Psi_1(\theta+\upalpha)-\Psi_1(\theta+\upalpha+\upbeta)}
          {\Psi_1(\theta)-\Psi_1(\theta+\upalpha+\upbeta)}
          \Xi(\theta)
              \\
              &\ge 
          \frac{1}{2\theta}\left(\Psi_2(\upalpha)-\Psi_2(\upalpha+\upbeta) 
                \right) 
    -
                    \frac{\Psi_1(\theta+\upalpha)-\Psi_1(\theta+\upalpha+\upbeta)}
          {\Psi_1(\theta)-\Psi_1(\theta+\upalpha+\upbeta)}
                \Xi(\theta)\\
              &\ge 
                          \frac{1}{2\theta}\left(\Psi_2(\upalpha)-\Psi_2(\upalpha+\upbeta) 
                \right) 
    +
                    \frac{\Psi_1(1+\upalpha)-\Psi_1(1+\upalpha+\upbeta)}
          {\Psi_1(\theta)-\Psi_1(\theta+\upalpha+\upbeta)}
                \left(-\Xi(\theta)                \right)\\
                           &\ge 
-2\mathfrak g(\upalpha,\upbeta)\frac{1}{\upalpha}
    +
                    \frac{\mathfrak g(2+\upalpha,\upbeta)(
                             (1+\upalpha)^{-1}) 
                             +1)}
          {\mathfrak g(1,\upalpha+\upbeta)(\theta^{-1}+1)}
                \left(-\Xi(\theta)                \right), 
            \end{align*}
            where for the last inequality we have used parts $(i)$ and 
            $(ii)$
            of Lemma \ref{poly}, and we have used the fact that

            $$
\Xi(\theta)=-\sum_{n=0}^\infty\frac{2\theta+4n}{4\theta (\theta+n)^4}<0.
            $$
            Now assume that $\upalpha_0>0$ 
            is fixed and $\upalpha\ge\upalpha_0$, $\upbeta>0$ and $\theta>0$. 
            In this case 

            $$
            \frac{\mathfrak g(2+\upalpha,\upbeta)}{\mathfrak g(\upalpha,\upbeta)}
              \ge 
              \left(\frac{\upalpha_0}{2+\upalpha_0}\right)^2=: 
                \newconstant\label{uaba}. 
            $$
            We then have that 

            $$
                    \frac{\sigma^3(\theta)}{2\theta}+\frac{h^{(4)}(\theta)}{4!}
                    \ge 
\mathfrak                    g(\upalpha,\upbeta) 
\left(                    -2 
                    \frac{1}{\upalpha_0}
                    +2\oldconstant{uaba}\frac{1}{1+\theta}
                  \left(-\theta\Xi(\theta)\right)
                  \right). 
                  $$
                  Since $-\theta\Xi(\theta)\to\infty$ as $\theta\to 0_+$,
                  from this inequality it is clear that there is a 
                  $\theta_0(\alpha_0)>0$ (depending only on 
                  $\upalpha_0$), 
                  such that for all $\upalpha\ge\upalpha_0$, 
                  $\upbeta>0$
                  and $\theta\in (0,\theta_0)$, the right-hand side is 
                  positive. Furthermore, we can also easily check
                  $\theta_0(\alpha_0)$
                  can be chosen so that
                  $\lim_{\alpha_0\to\infty}\theta_0=\infty$.
                
                  \medskip

                  \noindent {\it Proof of part $(v)$}.
The first inequality of (\ref{vii}) follows from part $(iv)$ after we prove that
$h^{(4)}(\theta)<0$.
So we will just prove that $h^{(4)}(\theta)<0$. Recall the expression
(\ref{fracp}) for $h^{(4)}(z)$.
  From this expression, we see that it is enough to show that 
  
\begin{align*} \frac{\Psi_3(\theta+\upalpha)-\Psi_3(\theta+\upalpha+\upbeta)}{\Psi_1(\theta+\upalpha)-\Psi_1(\theta+\upalpha+\upbeta)}<\frac{\Psi_3(\theta)-\Psi_3(\theta+\upalpha+\upbeta)}{\Psi_1(\theta)-\Psi_1(\theta+\upalpha+\upbeta)}. 
\end{align*}
This is equivalent to $G=\frac{\Psi_3}{\Psi_1}$ being strictly convex 
on $(0,\infty)$, that is, 
\[G''=\frac{\Psi_3''\Psi_1'-\Psi_3'\Psi_1''}{(\Psi_1')^3}\circ\Psi_1^{-1}>0\]
or, as $\Psi_1<0$, 
\[
  (\Psi_3''\Psi_1'-\Psi_3'\Psi_1'')(x)=48\bigg(3\sum_{n=
    0}^\infty\frac1{x_n^5}
  \sum_{n= 0}^\infty\frac1{x_n^4}-5\sum_{n=0}^\infty\frac1{x_n^6}\sum_{n=
    0}^\infty\frac1{x_n^3} \bigg)<0, 
\]
where we used the notation $x_n=x+n$. But this is true, since 
\begin{equation}
  \nonumber 
\sum_{n=0}^\infty \frac1{x_n^6}\sum_{n= 0}^\infty\frac1{x_n^3}-
\sum_{n=0}^\infty\frac1{x_n^5}\sum_{n=0}^\infty\frac1{x_n^4}>0, 
\end{equation}
 as can be seen by writing the $m$-$n$ products (the $m=n$ terms are cero), 
\begin{equation}
  \nonumber 
\frac1{x_m^6}\,\frac1{x_n^3}+\frac1{x_n^6}\,\frac1{x_m^3}-
\frac1{x_m^5}\frac1{x_n^4}-\frac1{x_n^5}\frac1{x_m^4}
=\frac1{x_n^6x_m^6}(x_m-x_n)(x_m^2-x_n^2). 
\end{equation}
\end{proof}

\medskip

  \subsection{Steep-descent estimates}
  \label{sde}
Let us now state three lemmas which will give the steepest-descent
properties of $h$ on $C_\theta$ and $D_\theta$.
The first one,  is a quantitative version of Lemma 5.5 of
\cite{Barr-Cor} which had
 the restriction $\upalpha=\upbeta=1$, extending it for all
 values of $\upalpha$ and $\upbeta$,
 at the cost of having to choose $\theta$ small enough
 and is one of the main challenges in the proof of Theorem \ref{one}.
 
       \medskip

       \begin{lemma}
         \label{fo}
                 For all $\upalpha>0$, $\beta>0$, 
  $\theta\in (0,\min\{0.5,0.72\times\upalpha\})$
 and $\phi 
  \in[0,2\pi]$, one has that 
  \begin{equation}
    \label{5.5a}
      \Re\left(i\theta e^{i\phi } h'\left(\theta 
      e^{i\phi}\right)\right)\ge \oldconstant{gonzalo}\theta^2\sin\phi (1-\cos\phi) 
        \mathfrak g(\upalpha+1,\upbeta),
      \end{equation}
      for some constant $\newconstant\label{gonzalo}>0$ depending only
      on $\upalpha$, but increasing on $\upalpha$.
\end{lemma}

\medskip
Numerical computations 
suggest that Lemma \ref{fo} is false for some values of $\upalpha>0$,
$\upbeta>0$ and $\theta>0$ (even $\theta\in (0,0.5)$). Nevertheless,
we
expect Theorem \ref{one} to be valid for all $\upalpha>0$, $\upbeta>0$
and $\theta>0$. This would imply that different steep descent contours
would have to be chosen.

    \medskip

    \noindent The second lemma we present is a quantitative version
    of Lemma 5.4 of \cite{Barr-Cor}.
Define for $x,y\in\mathbb R$, 

\begin{equation}
  \label{phixy}
\Phi(x,y)=\sum_{n=0}^\infty\frac{1}{(n+x)^2+y^2}. 
\end{equation}

\medskip

\begin{lemma}
\label{im}	For all $\upalpha > 0$, $\upbeta > 0$ and $\theta>0$ we have that 

\begin{itemize}
\item[(i)] $\Im h'(\theta + iy) > 0$ for $y > 0$ and $\Im h'(\theta + iy) < 0$ for $y < 0$. 
\item[(ii)] $\Im h'(\theta + iy)=yH(\theta,y,\upalpha,\upbeta)$ where

  \begin{equation}
    \label{hthetay}
  H(\theta,y,\upalpha,\upbeta)\ge8\frac{\Psi_1(\theta+\upalpha)-\Psi_1(\theta+\upalpha+ \upbeta)}{
    \Psi_1(\theta)-\Psi_1(\theta+\upalpha+ \upbeta)}
  \int_\theta^{\theta+\upalpha}\sum_{n\ge0}\frac{y^2(x-\theta)dx}{((x+n)^2+y^2)^3}.
\end{equation}
  \end{itemize}

\end{lemma}

\medskip

Finally, the third lemma, is a requirement that has been used in the
latest version of \cite{Barr-Cor}, and which in our case gives the biggest
limitation
on the range of values of $\upalpha,\upbeta$ and $\theta$ for which we
can prove Theorem \ref{one}.

\medskip 

\begin{lemma}
  \label{iv}
 For all $\upalpha>0$, $\upbeta>0$ and $\theta\in (0,\min\{0.5,\upalpha\})$, we have 
  that 

  $$
  \frac{\sigma^3(\theta)}{2\theta}+\frac{h^{(4)}(\theta)}{4!}>0. 
  $$
\end{lemma}

\medskip Numerical computations also suggest that Lemma \ref{one} is
false for some values of $\upalpha>0$, $\upbeta>0$ and $\theta\in
(0,0.5)$.
Again, we expect Theorem \ref{one} to hold despite this, but
different steep descent contours would have to be chosen.
We will present the proof of Lemma \ref{one} in Section \ref{sp00}.

We will present the proofs of Lemmas \ref{fo} and \ref{im}
in Section \ref{lemmasfo}.
  Now we will continue developing several
consequences of these lemmas.
From Lemma \ref{fo} we obtain the following
corollary which extends Lemma 5.5 of \cite{Barr-Cor}
for arbitrary $\upalpha>0$ and $\upbeta>0$.

In what follows we will say that
the contour $C_\theta$ is steep-descent 
  for the function $-\Re(h)$  if $\Re(h(\theta 
  e^{i\phi}))$ is strictly increasing for $\phi\in (0,\pi)$ and 
  strictly decreasing for $(-\pi,0)$.
\medskip

\begin{corollary}
  \label{cor1}
                 For all $\upalpha>0$, $\beta>0$ and 
  $\theta\in (0,\min\{0.5,0.72\times\upalpha\})$
 the contour $C_\theta$ is steep-descent for $-\Re(h)$.
\end{corollary}

\medskip

\noindent On the other hand, exactly as in \cite{Barr-Cor}, from
part $(i)$ of Lemma \ref{im}, we obtain Lemma 5.4 of \cite{Barr-Cor},
which we state as the following corollary for convenience.

\medskip

\begin{corollary}
\label{corab}  For all $\upalpha>0$, $\upbeta>0$ and $\theta>0$,
  the contour $D_\theta$ is steep-descent for the function $\Re(h)$
  in the sense that $\Re(h(\theta+iy))$ is strictly decreasing for $y$ positive
  and
strictly   increasing for $y$ negative.
 \end{corollary}

 \medskip

 We will now need to modify the vertical line contour $D_\theta$
 so that it avoids the singularity at $\theta$. This will be a time
 dependent
 modification. For each $r>0$ define,

 $$
D_{\theta}^{r,+}:=D_\theta-B\left(\theta, r\right), 
$$

 $$
V_{\theta}^{r,+}:=V^\epsilon_\theta-B\left(\theta, r\right), 
$$
 the semicircle

$$
S_\theta^t:=\left\{z\in\mathbb C: |z-\theta|=\frac{1}{\sigma(\theta) t^{1/3}}, \Re
  (z-\theta)\ge 0\right\}
$$
and the contour

$$
D_{\theta}^t:=D_{\theta}^{\sigma^{-1}t^{-1/3},+}\cup S_\theta^t.
$$
See Figure \ref{figure3} for a representation of these contours
together
with the  contour $V_\theta^\epsilon$. The presence of the factor $\sigma^{-1}$
in the
definition of $D_\theta^t$ will be irrelevant in the proof of Theorem
\ref{one},
since it will only change the computations by a fixed
constant. Nevertheless,
it will be important in the proof of Theorem \ref{int-dis}, where
$\sigma(\theta)\to 0$ as $t\to\infty$.
In what follows we adopt the notation $\sigma_t=\sigma_t(\theta)$
to admit the possibility that $\sigma$ is time dependent because it
is a function of time dependent parameters $(\upalpha_t,\upbeta_t)$.

\begin{figure}[!h]

\begin{tikzpicture}[decoration={markings, 
mark=at position 7cm with {\arrow[line width=2pt]{>}}
}
]
\draw[help lines,->] (8,0) -- (14,0) coordinate (xaxis); 

\draw[black, dashed, thin] (10,0) circle (3.1 cm);


\path[draw,line width=1.5pt,postaction=decorate]
(8.596255,-3.856725) arc (-40:-30:6)-- (10,0) node[below right] {$\theta$} 
--(9.19615242,3) 
arc(30:40:6);

\path[draw, line width=1.5pt,postaction=decorate] (10,-4)--(10,-0.8)
arc(-90:90:0.8)
--
(10,4);

\node at (5,2.5) {$\ $}; 
\node at (9.2,1.5) {$V^\epsilon_\theta$}; 
\node at (11,0.7) {$D_\theta^t$};
\node at (11.7,-0.3) {$\theta+\frac{1}{\sigma t^{1/3}}$};

\node at (10.8,0) {$\bullet$};
\node at (10,0) {$\bullet$};

\node at (13.1,0) {$\bullet$};
\node at (13.6,-0.3) {$\theta+\epsilon$}; 
\end{tikzpicture}

\caption{The contours $V_\theta^\epsilon$  and $D_\theta^t$ with its
  $t$
dependent deformation to avoid the singularity at $\theta$.}
\label{figure3}
\end{figure}
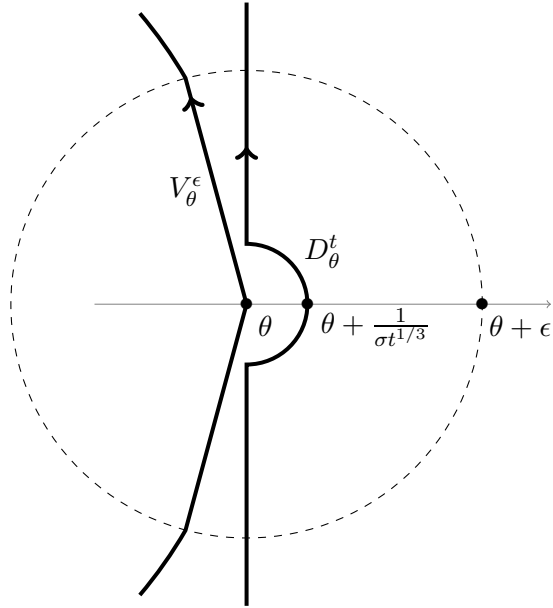

\medskip

\begin{lemma}
  \label{taylor1} For all $\upalpha>0$, $\upbeta>0$, $\theta>0$
  and $t\ge t_0$, where $\sigma_{t_0} t_0^{1/3}\theta\ge 2$, 
  we have that for all $z\in D_\theta^t$,

  $$
Re(h(z)-h(\theta))\le 128\mathfrak 
g(\upalpha,\upbeta)\frac{1}{\theta^5\sigma^3 t}. 
$$
\end{lemma}
\begin{proof}
  For $z\in D_\theta-B(\theta,\sigma^{-1} t^{-1/3})$, the lemma
  follows from Corollary \ref{corab}. For $z\in D_\theta^t\cap B(\theta,\sigma^{-1}
  t^{-1/3})$
  we have that, for $t\ge t_0$ where $t_0$ is such that $\sigma 
      t^{1/3}_0\theta\ge 2$, 

      \begin{align*}
&|h(z)-h(\theta)|\le \sum_{k=3}^\infty \frac{1}{k!}\frac{1}{\sigma^k 
  t^{k/3}} |h^{(k)}(\theta)|
\le 
64 \mathfrak g(\upalpha,\upbeta) 
\sum_{k=3}^\infty \frac{1}{\sigma^k 
  t^{k/3}\theta^{k+2}}\\
&\le 
64 \mathfrak g(\upalpha,\upbeta) \frac{1}{\theta^5}\frac{1}{\sigma^3 t}
\sum_{k=0}^\infty \frac{1}{\sigma^k 
                           t^{k/3}\theta^{k}}=128 
                           \mathfrak g(\upalpha,\upbeta) 
                           \frac{1}{\theta^5\sigma^3 t}, 
      \end{align*}
      where we have used part $(iii)$ of Corollary \ref{corcor}.

  \end{proof}

\begin{corollary}
  \label{c10}
               For all $\upalpha>0$, $\beta>0$ and 
  $\theta\in (0,\min\{0.5,0.72\times\upalpha\})$
 the following are satisfied.

    \begin{itemize}

    \item[(i)] There is a constant $\newconstant\label{7cc}>0$
      such that for every $z\in D_{\theta}^{\epsilon,+}$ and $v\in C_\theta$ we 
    have that 

    $$
    \Re(h(z)-h(v)) 
    \le     -
\oldconstant{7cc} \mathfrak h(\upalpha,\upbeta,\theta)
\epsilon^4, 
$$
where $\mathfrak h(\upalpha,\upbeta,\theta)>0$ and for $\upalpha\ge 1$
and $\upbeta\ge 1$ we have that

\begin{equation}
  \label{hfrakt}
  \mathfrak
h(\upalpha,\upbeta,\theta)\ge \oldconstant{80}(\theta)
\mathfrak g(\upalpha,\upbeta),
\end{equation}
for some
$\newconstant\label{80}(\theta)>0$
which does not depend on $\upalpha$ or $\upbeta$.

\item[(ii)] There is a constant $\newconstant\label{9cc}>0$ such that for every $z\in D_\theta^t$
    and $v\in C_\theta^{\epsilon,+}=V_{\theta}^{\epsilon,+}$
 we 
    have that 

    $$
    \Re(h(z)-h(v)) 
    \le     -\oldconstant{9cc}
\epsilon^4 \mathfrak g(\upalpha+1,\upbeta) 
      +128 
                           \mathfrak g(\upalpha,\upbeta) 
                           \frac{1}{\theta^5\sigma^3 t}.
$$
\end{itemize}
\end{corollary}

    \begin{proof}{\it Proof of part $(i)$}. From Corollary
      \ref{cor1} we can see that
       $\Re(h(\theta)-h(v))\le 
      0$ which implies that $\Re(h(z)-h(v))\le \Re(h(z)-h(\theta))$. On the other 
      hand, 
      by Corollary \ref{corab}, $\Re(h(z))$ is decreasing in $y$ for 
      $z=\theta+iy$, $y\ge \epsilon$ and 
      increasing in $y$ for 
      $z=\theta+iy$, $y\le -\epsilon$, so that 
      $\Re(h(z)-h(\theta))\le\Re(h(\theta+i\epsilon)-h(\theta))$
      for $z\in D_{\theta}^{\epsilon,+}$.
      It follows that

      \begin{align*}
        & \Re(h(z)-h(v)) =\Re(h(z)-h(\theta))+\Re(h(\theta)-h(v))\\
&        \le \Re(h(\theta+i\epsilon)-h(\theta))
        =\int_0^\epsilon 
\frac{d\Re (h(\theta+iy))}{dy} dy=-\int_0^\epsilon 
        \Im(h'(\theta+iy))dy\\
        &\le 
         - \frac{H(\theta,\epsilon,\upalpha,\upbeta)}{4}
\epsilon^4, 
      \end{align*}
      where in the last inequality we have used part $(ii)$ of Lemma 
      \ref{im}. Now,

      $$
      H(\theta,y,\upalpha,\upbeta)\ge 8 y^2 
      \frac{\Psi_1(\theta+\upalpha)-\Psi_1(\theta+\upalpha+\upbeta)}{\Psi_1(\theta)-
        \Psi_1(\theta+\upalpha+\upbeta)}
      \int_\theta^{\theta+\upalpha}
      \sum_{n\ge 0}\frac{(x-\theta)dx}{((x+n)^2+1)^3}.
      $$
      Defining,

      $$\mathfrak h(\upalpha,\upbeta,\theta):=
8 y^2 
      \frac{\Psi_1(\theta+\upalpha)-\Psi_1(\theta+\upalpha+\upbeta)}{\Psi_1(\theta)-
        \Psi_1(\theta+\upalpha+\upbeta)}
      \int_\theta^{\theta+\upalpha}
      \sum_{n\ge 0}\frac{(x-\theta)dx}{((x+n)^2+1)^3},
      $$
      we finish the proof.

      \medskip

      \noindent {\it Proof of part $(ii)$}. By Corollary \ref{cor1}
      and the fact that $v\in C_\theta^{\epsilon,+}$,
 we have that
      $\Re(h(\theta)-h(v ))\le\Re (h(\theta)-h(\theta e^{i\phi_*}))$, 
      where $\phi_*$ is such that $|\theta e^{i\phi_*}-\theta|=\epsilon$. 
Hence, by Lemma \ref{taylor1}, we have that

      \begin{align*}
        & \Re(h(z)-h(v))\\
        &=\Re(h(z)-h(\theta))+\Re(h(\theta)-h(v))\le 
          \Re(h(\theta)-h(\theta e^{i\phi_*}))\\
        &=-\int_0^{\phi_*}
\Re(i\theta e^{i\phi}h'(\theta e^{i\phi})) d\phi\le 
          - \epsilon^4 \oldc{1}\mathfrak g(\upalpha+1,\upbeta)
          +128 
                           \mathfrak g(\upalpha,\upbeta) 
                           \frac{1}{\theta^5\sigma^3 t},
      \end{align*}
for some constant $\newc\label{1}$,       where in the last inequality we have used
   Lemma 
      \ref{fo}. 
    \end{proof}
    \medskip

    \subsection{Steep-descent properties}
\label{sdp}
    Here we will apply the steep-descent estimates of the previous
    section
    to derive important steep descent properties of the Fredholm
    determinant. We will start proving a couple
    of properties about the functions involved.

\medskip

\begin{lemma}
  \label{th}
                 For all $\upalpha>0$, $\beta>0$,
                 $\theta\in (0,\min\{0.5,0.72\times\upalpha\})$
                 and       $\epsilon\in(0,\theta/2)$ 
 the following bounds are satisfied. 
\begin{itemize}

\item[(i)] For all $v\in C_\theta$, $z\in D_{\theta}^{\epsilon,+}$ and 
  $t\ge 0$ we 
  have that 

      \begin{equation}
        \label{boundd}
      \left|e^{t(h(z)-h(v))-t^{1/3}\sigma(\theta) y(z-v)}\right|
      \le e^{-t\oldconstant{7cc}\epsilon^4 \mathfrak h(\upalpha,\upbeta,\theta) 
        + t^{1/3}\sigma(\theta) |y|\epsilon}. 
      \end{equation}

      \item[(ii)] For all $v\in C_\theta^{\epsilon,+}=V_\theta^{\epsilon,+}$, $z\in D_{\theta}^t$ and 
  $t\ge 0$ we 
  have that 

      \begin{equation}
        \label{boundd1}
      \left|e^{t(h(z)-h(v))-t^{1/3}\sigma(\theta) y(z-v)}\right|
      \le e^{-t\oldconstant{9cc}\epsilon^4 \mathfrak g(\upalpha+1,\upbeta)
        +128\mathfrak g(\upalpha,\upbeta)\frac{1}{\theta^5\sigma^3 }
        + 2t^{1/3}\sigma(\theta) |y|}. 
      \end{equation}

\item[(iii)] There is a constant
  $\newconstant\label{100}(\theta,\epsilon,\upalpha,\upbeta)>0$, which is
  independent of $\upalpha$ and $\upbeta$ when both $\upalpha\ge 1$ and
  $\upbeta\ge 1$, such that for all $v\in W_{\theta,\epsilon}^L$, $z\in
  D_{\theta}^t$  or $v\in C_\theta^\epsilon$, $z\in D_\theta^t$, and $t\ge 1$ ,  we 
  have that 

      \begin{equation}
        \label{boundd2}
      \left|e^{t(h(z)-h(v))-t^{1/3}\sigma(\theta) y(z-v)}\right|
      \le \frac{1}{\oldconstant{100}(\theta,\epsilon,\upalpha,\upbeta)}e^{- \oldconstant{100}(\theta,\epsilon,\upalpha,\upbeta) t\sigma(\theta)^3 
        |v-\theta|^3}.
      \end{equation}

\end{itemize}

\end{lemma}
\begin{proof} {\it Proof of parts $(i)$ and $(ii)$}.   Part $(i)$
  follows from
  part $(i)$ of Corollary 
      \ref{c10} and the inequalities,

  \begin{align*}
  \left|e^{t(h(z)-h(v))-t^{1/3}\sigma(\theta) y(z-v)}\right|
                     & =     e^{t\Re((h(z)-h(v)))-t^{1/3}\sigma(\theta) y\Re((z-v))}\\
  &  \le     e^{-t\oldconstant{7cc}\epsilon^4\mathfrak h(\upalpha,\upbeta,\theta)
    +t^{1/3}\sigma(\theta) |y|}.
  \end{align*}
Similarly part $(ii)$ follows from   part $(ii)$ of Corollary 
      \ref{c10}.

      \medskip 

      \noindent {\it Proof of part $(iii)$}. Let us use the inequality
      $|z-v|\le |v-\theta|+\sigma^{-1}t^{-1/3}$ and Lemma \ref{taylor1} to conclude that

      \begin{align}
        \nonumber 
        &      \Re\left( t(h(z)-h(v))-t^{1/3}\sigma(\theta) y(z-v)\right)\\
        \nonumber 
        &    =\Re (t(h(z)-h(\theta))+
      \Re\left( t(h(\theta)-h(v))-t^{1/3}\sigma(\theta) 
                                                                      y(z-v)\right)\\
        \label{tttt}
        \!\!\!\!\! &\le  \!   \Re\left( t(h(\theta)-h(v))\right) 
                                                                                    +
                                                                                    128\mathfrak 
                                                                                    g(\upalpha,\upbeta)\frac{1}{\theta^5\sigma^3 }
                    +
                                                                        t^{1/3}\sigma(\theta)|y||v-\theta|+|y|.
      \end{align}
To estimate the first term of the right-hand side in display 
(\ref{tttt}) we will make a Taylor series expansion around $\theta$
of $h(v)$. Recall that $v\in W_{\theta,\epsilon}^L$ or $v\in 
C_\theta^\epsilon$.
Define $\bar v$ by 

$$
v=\theta+\frac{1}{\sigma(\theta)t^{1/3}}\bar v, 
$$
and $\bar h(\bar v):=h(v)$. Then,

\begin{align*}
&t\bar h(\bar 
                    v)\\
  &=th(\theta)+\frac{\bar v^3}{3!}\frac{1}{\sigma^3}h^{(3)}(\theta)+
\frac{\bar v^4}{4!}\frac{1}{\sigma^4 t^{1/3}}h^{(4)}(\theta)+
\sum_{k=5}^\infty \frac{\bar v^k}{k!}\frac{1}{\sigma^k t^{(k-3)/3}}
    h^{(k)}(\theta). 
\end{align*}
Thus, 

\begin{align}
  &
  \nonumber
\left|t\bar h(\bar v)-    th(\theta)-\frac{\bar v^3}{3!}\frac{1}{\sigma^3}h^{(3)}(\theta)-
    \frac{\bar v^4}{4!}\frac{1}{\sigma^4
    t^{1/3}}h^{(4)}(\theta)\right|\\
  &
      \label{last} \le 
   \frac{|\bar v^3|}{\sigma^3} \sum_{k=5}^\infty \frac{\epsilon^{k-3}}{k!}
    h^{(k)}(\theta). 
\end{align}
Now, from part $(iii)$ of Corollary \ref{corcor}, we have that

$$
\left|\frac{|\bar v^3|}{\sigma^3} \sum_{k=5}^\infty 
  \frac{\epsilon^{k-3}}{k!}h^{(k)}(\theta)\right|\le 
64\frac{\mathfrak g(\upalpha,\upbeta)}{\sigma^3} |\bar v^3|\sum_{k=5}^\infty\frac{\epsilon^{k-3}}{\theta^{k+2}}
\le 128 \frac{\mathfrak g(\upalpha,\upbeta)}{\sigma^3} \epsilon^2 \frac{1}{\theta^7}|\bar v^3|, 
$$
where in the last inequality we have used $\epsilon<\theta/2$.
It follows that 
\begin{equation}
  \label{tr}
\left|t\Re(\bar h(\bar v)-h(\theta))-\frac{\Re(\bar v^3)}{3}-
    \frac{\Re(\bar v^4)}{4!}\frac{1}{\sigma^4
      t^{1/3}}h^{(4)}(\theta)\right|\le
128  \frac{\mathfrak g(\upalpha,\upbeta)}{\sigma^3} \epsilon^2 \frac{1}{\theta^7}|\bar v^3|.
\end{equation}
Now, since $v\in W_{\theta,\epsilon}^L$ or $v\in 
C_\theta^\epsilon$,
 the argument
  of $\bar v$ is $\pm\left(\frac{\pi}{2}+\frac{\epsilon}{2\theta}+o(\epsilon)\right)$.
Then, from the fact that 
$h^{(4)}(\theta)<0$ (see part $(v)$ of Corollary \ref{corcor}) and that $(\sigma
t^{1/3})^{-1}|\bar v|<\epsilon$ (because $v\in W_{\theta,\epsilon}^L$), we have 

\begin{align}
\nonumber  &-\frac{\Re(\bar v^3)}{3}-
\frac{\Re(\bar v^4)}{4!}\frac{1}{\sigma^4 t^{1/3}}h^{(4)}(\theta)\\
\nonumber &=-\sin\left(\frac{3\epsilon}{2\theta}+o(\epsilon)\right)\frac{|\bar 
  v|^3}{3}-
\cos\left(\frac{2\epsilon}{\theta}+o(\epsilon)\right)\frac{h^{(4)}(\theta)}{\sigma^4 
                                                                      t^{1/3}}\frac{|\bar v|^4}{4!}\\
\nonumber &\le -\epsilon |\bar v|^3\frac{1}{\sigma^3}\left(
                                                                                                        \frac{\sigma^3}{2\theta}+\frac{h^{(4)}(\theta)}{4!}
                                                                                                        +o_1(\epsilon)\frac{\sigma^3}{\epsilon}+o_2(\epsilon) 
            h^{(4)}(\theta)\right)\\
           &
             \nonumber
             =
             -\epsilon |\bar v|^3\frac{1}{\sigma^3}\left(
             \frac{\sigma^3}{2\theta}+\frac{h^{(4)}(\theta)}{4!}\right)
                                                                                                        +o_1(\epsilon)|\bar
             v|^3 +o_2(\epsilon)|\bar v|^3 
           \frac{ h^{(4)}(\theta)}{\sigma^3}\\
\label{lee}  &\le -\epsilon |\bar v|^3 \oldc{6}(\theta,\upalpha,\upbeta), 
\end{align}
where in the last inequality we have assumed that $\epsilon$ is small
enough, and where
by parts $(iii)$ and $(vi)$ of Corollary \ref{corcor} and by Lemma \ref{iv}, we have that $\newc\label{6}(\theta,\upalpha,\upbeta)$ is
a constant independent of $\upalpha$ and $\upbeta$ when both
$\upalpha\ge 1$ and $\upbeta\ge 1$, and $\oldc{6}(\theta,\upalpha,\upbeta)>0$. We then have 
from (\ref{tr}) and (\ref{lee}), using again part $(iv)$ of Corollary
\ref{corcor},
that for $\epsilon$ small enough,

$$
-t\Re(\bar h(\bar v)-h(\theta))\le  -\oldc{7}(\theta,\upalpha,\upbeta)\epsilon |\bar
v|^3,
$$
where  $\newc\label{7}(\theta,\upalpha,\upbeta)$ does not depend on $\upalpha$ and
$\upbeta$
for $\upalpha\ge 1$ and $\upbeta\ge 1$.
Hence, 

$$
-t\Re( h( v)-h(\theta))\le  -\oldc{7}(\theta,\upalpha,\upbeta)\epsilon \sigma(\theta)^3t | v-\theta|^3 
$$
and from (\ref{tttt}) and Lemma \ref{taylor1}, we get that

\begin{align*}
  &\left|e^{t(h(z)-h(\theta))-t^{1/3}\sigma(\theta)y(z-v)}\right|\le
    \left|e^{t(h(z)-h(\theta))+t(h(\theta)-h(v))-t^{1/3}\sigma(\theta)y(z-v)}\right|\\
  &
    \le
    e^{128\mathfrak g(\upalpha,\upbeta)\frac{1}{\theta^5\sigma^3}}
    e^{t(\Re(h(\theta)-h(v)))+t^{1/3}\sigma(\theta)|y(z-v)|}\\
  &\le 
e^{-\oldc{7}(\theta,\upalpha,\upbeta)\epsilon  t\sigma(\theta)^3
  |v-\theta|^3+                                                              128\mathfrak 
                                                                                    g(\upalpha,\upbeta)\frac{1}{\theta^5\sigma^3 }
  +\newc t^{1/3}\sigma(\theta)|y||v-\theta| }, 
\end{align*}
which proves part $(iii)$. 

\end{proof}

    \medskip

\subsection{Proof of Proposition \ref{p1}}
\label{sp1}
We will now present the proof of Proposition \ref{p1} (see
also Proposition 5.6 of
\cite{Barr-Cor} and Section 8 of \cite{K21}) using the steep descent
properties
of Section \ref{sdp}
along appropriate contours. At several steps,
in order to deal with the singularity of the kernel $K_u^{RW}(v,v')$ for
$v=v'=\theta$,
we will deform the contours $C_\theta$ to $V_\theta^\epsilon$
(see Figure \ref{figure1}) and we will also have to
deform the contour $D^\epsilon_\theta$ so that it avoids $\theta$.
We will then show that the integration outside the ball $B(\theta,\epsilon)$
has a negligible contribution in the limit.
In what follows we will assume that either $\upalpha$ and $\upbeta$
are fixed (time-independent) or that these parameters are time-dependent,
$(\upalpha_t)_{t\ge 0}$ and $(\upbeta_t)_{t\ge 0}$, and satisfy
(\ref{as1}) and (\ref{as2}).

 \medskip

\noindent {\bf Step 0.}
 As explained above, we first deform the contour of integration in the 
 definition of the kernel $K_u^{RW}(v,v')$ from $D_{1/2}$ to 
 $D_\theta^t$ (see Figure \ref{figure3}), so that

\begin{equation}
  \label{ku2}
K_u^{RW}(v,v')=\frac{1}{2\pi 
  i}\int_{D_\theta^t}\frac{\pi}{\sin(\pi (z-v))}
e^{t(h(z)-h(v))-t^{1/3}\sigma(\theta) y(z-v)}\frac{\Gamma(v)}{\Gamma(z)}
\frac{dz}{z-v'}. 
\end{equation}

\medskip 

\noindent {\bf Step 1.} We will now show that 
whenever $v,v'\in C_\theta$ we have that 

\begin{equation}
  \label{limkke}
    | K_u^{RW}(v,v')- K_{y,\epsilon}^{RW}(v,v')|\le 
   \frac{1}{\oldc{9}} e^{-\oldca{1a}t}, 
  \end{equation}
for some constant $\newc\label{9}=\oldc{9}(\theta,\epsilon,y)>0$
independent 
of $\upalpha$ and $\upbeta$ and 
   $\newca\label{1a}=\oldca{1a}(\theta,\epsilon,\upalpha,\upbeta,y)>0$, 
  such that 

  $$
\oldca{1a}(\theta,\epsilon,\upalpha,\upbeta,y)\ge \oldc{9} \mathfrak 
g(\upalpha,\upbeta) 
\quad {\rm for}\ \upalpha\ge 1\ {\rm and}\ \upbeta\ge 1, 
$$
where the kernel $K_{y,\epsilon}$ was defined in (\ref{keps}). 
Note that

\begin{align}
  \nonumber 
&   K_u^{RW}(v,v')- K_{y,\epsilon}^{RW}(v,v')=\\
  \label{ku3}
&\int_{D_\theta^{\epsilon,+}}\frac{\pi}{\sin(\pi (z-v))}
e^{t(h(z)-h(v))-t^{1/3}\sigma(\theta) y(z-v)}\frac{\Gamma(v)}{\Gamma(z)}
\frac{dz}{z-v'}. 
\end{align}
       From part $(i)$ of Lemma \ref{th}, we have that since $v\in 
       C_\theta$ and $z\in D_\theta^{\epsilon,+}$, 

      \begin{equation}
        \label{limit2}
\left| e^{t(h(z)-h(v))-t^{1/3}\sigma(\theta) y(z-v)}\right|\le e^{-t 
  \oldconstant{7cc} \epsilon^4\mathfrak h(\upalpha,\upbeta,\theta)+t^{1/3}\sigma|y|\epsilon}. 
\end{equation}
Now, recall the  asymptotics 
$|\Gamma(x+iy)| e^{\frac{\pi}{2}|y|}|y|^{1/2-x}\to\sqrt{2\pi}$ as 
$y\to\pm\infty$
for $x$ and $y$ real. This implies that 

\begin{equation}
  \label{gammaz}
|\Gamma(z)|\ge \frac{1}{\oldc{30}} e^{-\frac{\pi}{2}|\Im (z)|}, 
  \end{equation}
  for $z\in D_\theta$ for some 
  $\newc\label{30}(\epsilon)>0$. 
  Furthermore, we have that 

  \begin{equation}
    \label{sinz}
  |\sin(\pi z)|\ge \oldc{40} e^{\pi |\Im 
    (z)|}, 
  \end{equation}
  for some $\newc\label{40}>0$. 
From this and 
 (\ref{limit2}) we get now using that $v,v'\in C_\theta$ and $z\in 
 D_\theta^{\epsilon,+}$ that 

\begin{align}
  \nonumber 
&        \left|\frac{\pi}{\sin (\pi (z-v))}
          e^{t(h(z)-h(v))-t^{1/3}\sigma(\theta) y(z-v)}
        \frac{\Gamma(v)}{\Gamma(z)}\frac{1}{z-v'}\right|\\
  &
            \label{boundd3}
\le \oldc{31}|\Gamma(v)| e^{-\frac{\pi}{2}|\Im(z)|}e ^{-t\oldconstant{7cc}\epsilon^4 \mathfrak h(\upalpha,\upbeta,\theta) 
        + t^{1/3}\sigma(\theta) |y|\epsilon}, 
\end{align}
for some constant $\newc\label{31}=\oldc{31}(\theta,\epsilon)>0$. 
Note that since $v\in C_\theta$, $|\Gamma(v)|\le \newc(\theta)$. 
Hence, from  (\ref{boundd3}),  using the fact that

$$
\mathfrak h(\upalpha,\upbeta,\theta)\ge \oldconstant{80}(\theta) \mathfrak g(\upalpha,\upbeta)\quad 
{\rm for}\ \upalpha\ge 1\ {\rm and}\ \upbeta\ge 1, 
$$
(see (\ref{hfrakt}) of part $(i)$ of Corollary \ref{c10} and part $(i)$ of Lemma \ref{th}), 
and part $(iii)$  of Corollary \ref{corcor} and Lemma \ref{iv} to obtain an 
upper bound for $\sigma$, 
we obtain (\ref{limkke}).

\medskip

\noindent {\bf Step 2.}
Here we will show that for 
$v\in C_\theta^{\epsilon,+}$ and $v'\in 
C_\theta$ one has that 

\begin{equation}
  \label{leftk}
  \left|K_{y,\epsilon}^{RW}(v,v')\right|\le 
  \left|K_u^{RW}(v,v')\right|\le   
\frac{1}{\oldc{15}} e^{-\oldca{2a}t}, 
\end{equation} 
for a pair of constants $\oldca{2a}(\theta,\upalpha,\upbeta,y)>0$ and 
$\oldc{15}=\oldc{15}(\theta,\epsilon,y)>0$
such that

  $$
\newca\label{2a}=\oldca{2a}(\theta,\epsilon,\upalpha,\upbeta,y)\ge \newc\label{15} \mathfrak 
g(\upalpha,\upbeta) 
\quad {\rm for}\ \upalpha\ge 1\ {\rm and}\ \upbeta\ge 1. 
$$
In fact, using again the estimates (\ref{gammaz}) and (\ref{sinz}) and  part $(ii)$ of Lemma (\ref{th}) we have that 
(here $z\in D_\theta^t$) 

\begin{align*}
  &     \left|\frac{\pi}{\sin (\pi (z-v))}
          e^{t(h(z)-h(v))-t^{1/3}\sigma(\theta) y(z-v)}
        \frac{\Gamma(v)}{\Gamma(z)}\frac{1}{z-v'}\right|\\
&\le \newc\label{17} |\Gamma(v)| e^{-\frac{\pi}{2}|\Im(z)|}
                                                             e^{-t\oldconstant{9cc}\epsilon^4 \mathfrak g(\upalpha+1,\upbeta) 
                                                             +128\mathfrak 
                                                             g(\upalpha,\upbeta) 
                                                             \frac{1}{\theta^5\sigma^3 }
        + 2t^{1/3}\sigma(\theta) |y|}, 
\end{align*}
for some constant $c_{11}=c_{11}(\theta,\epsilon)>0$. 
Integrating over $z$ and using  part $(iii)$ of Corollary \ref{corcor} to bound  $\sigma$, 
we obtain (\ref{leftk}). 

\medskip 

\noindent {\bf Step 3.}
Here we will show that for all 
$v\in C_\theta^\epsilon$ and $v'\in C_\theta$, we have that

\begin{equation}
  \label{leftk2}
  \left|K_{y,\epsilon}^{RW}(v,v')\right|\le 
  \left|K_u^{RW}(v,v')\right|\le   
\frac{1}{\oldc{18}} \sigma t^{1/3} e^{-\oldc{18} t\sigma^3 |v-\theta|^3}, 
\end{equation} 
for some constant $\newc\label{18}=\oldc{18}(\theta,\epsilon,\upalpha,\upbeta)>0$
such that 

$$
\oldc{18}(\theta,\epsilon,\upalpha,\upbeta)\ge \newc\label{20}\quad 
{\rm for}\ \upalpha\ge 1,\upbeta\ge 1, 
$$
for some constant $\oldc{20}>0$ independent of $\upalpha$ and $\upbeta$. 
Indeed, by part $(iii)$
of Lemma \ref{th}, (\ref{gammaz}), (\ref{sinz}), the inequality 

$$
\left|\frac{(z-v)}{\sin(\pi(z-v))}\right|\le \newc, 
$$
and the fact that $|z-v'|\ge \newc (\sigma t^{1/3})^{-1}$, we have that

\begin{align*}
  &     \left|\frac{\pi}{\sin (\pi (z-v))}
          e^{t(h(z)-h(v))-t^{1/3}\sigma(\theta) y(z-v)}
        \frac{\Gamma(v)}{\Gamma(z)}\frac{1}{z-v'}\right|\\
&\le \newc\label{211}\sigma t^{1/3} |\Gamma(v)| e^{-\frac{\pi}{2}|\Im(z)|}
\frac{1}{\oldconstant{100}}e^{- \oldconstant{100} t\sigma(\theta)^3 
        |v-\theta|^3}, 
\end{align*}
for some constant $c_{16}>0$, where we recall that 
$\oldconstant{100}(\theta,\epsilon,\upalpha,\upbeta)>0$ does not 
depend on $\upalpha$ and $\upbeta$ for $\upalpha\ge 1$ and $\upbeta\ge 1$. 
which proves (\ref{leftk2}). 

\medskip

\noindent {\bf Step 4.} 
Here we will prove that 

\begin{equation}
  \label{step2}
\lim_{t\to\infty}\det(I-K^{RW}_u)_{L^2(C_\theta)}=\lim_{t\to\infty}\det(I-K^{RW}_u)_{L^2(C_\theta^\epsilon)}, 
\end{equation}
as long as the limit in the right-hand side exists. 
 Consider the Fredholm determinantal expansion 

\begin{align}
\nonumber 
  &\det(I-K^{RW}_u)_{L^2(C_\theta)}\\
  \label{fexp}  &=1+\sum_{n=1}^\infty 
\frac{(-1)^n}{n!}\int_{(C_\theta)^n} 
\det\left(K_u^{RW}(w_i,w_j)\right)_{i,j=1}^ndw_1\ldots dw_n. 
\end{align}
Now note that

\begin{align*}
&\int_{(C_\theta)^n} 
\det\left(K_u^{RW}(w_i,w_j)\right)_{i,j=1}^ndw_1\ldots dw_n\\
&=
\int_{(C_\theta^\epsilon)^n} 
\det\left(K_u^{RW}(w_i,w_j)\right)_{i,j=1}^ndw_1\ldots dw_n\\
&+\int_{(C_\theta^\epsilon)^n\backslash (C_\theta)^n}
\det\left(K_u^{RW}(w_i,w_j)\right)_{i,j=1}^ndw_1\ldots dw_n. 
\end{align*}
Let us now show that the limit as $t\to\infty$ of the second term 
above vanishes.

Combining (\ref{leftk}) and (\ref{leftk2}) with 
Hadamard's inequality, and using part $(iii)$ of Corollary 
\ref{corcor}, 
we have that

 \begin{equation}
 \label{leftdet}
\left|\det\left(K_u^{RW}(w_i,w_j)\right)_{i,j=1}^n\right|\le 
n^{n/2}(\oldc{22} \sigma^3 t)^{n/3}e^{-\oldc{23}\sigma^3 t}, 
\end{equation}
for some constants 
$\newc\label{22}=\oldc{22}(\theta,\epsilon,\upalpha,\upbeta)>0$
and $\newc\label{23}=\oldc{23}(\theta,\epsilon,\upalpha,\upbeta)>0$
with 
the property that both constants are independent of $\upalpha$ and 
$\upbeta$
for $\upalpha\ge 1$ and $\upbeta\ge 1$. 
But note that the right-hand side of (\ref{leftdet}) defines 
a series 

\begin{equation}
  \label{sumn}
\sum_{n=1}^\infty \frac{n^{n/2}}{n!}(\oldc{22}\sigma^3 
t)^{n/3}e^{-\oldc{23}\sigma^3 t}\le  e^{-\oldc{23}\sigma^3 t 
  +2 (\oldc{22})^2(\sigma^3 t)^{2/3}}. 
\end{equation}
Now, in the case in which $\upalpha$ and $\upbeta$ are fixed it is 
obvious that the right-hand side of (\ref{sumn}) converges to zero as 
$t\to\infty$. 
On the other hand, under the assumption that $(\upalpha_t)_{t\ge 0}$
and 
$(\upbeta_t)_{t\ge 0}$ satisfy (\ref{as1}) and (\ref{as2}), we have 
that $\lim_{t\to\infty}\sigma^3 t=\infty$ (see (\ref{limti})), so that 
we also have that the right-hand side of (\ref{sumn}) tends to $0$
as $t\to\infty$. It follows that 

$$
\lim_{t\to\infty}\sum_{n=1}^\infty \frac{(-1)^n}{n!}\int_{(C_\theta)^n\backslash (C_\theta^\epsilon)^n}
\det\left(K_u^{RW}(w_i,w_j)\right)_{i,j=1}^ndw_1\ldots dw_n=0, 
$$
which proves (\ref{step2}). 
\medskip 

\noindent {\bf Step 5.} Here we will show that

\begin{equation}
  \label{step3}
\lim_{t\to\infty}\det(I-K^{RW}_u)_{L^2(C_\theta^\epsilon)}=\lim_{t\to\infty}\det(I-K^{RW}_{y,\epsilon})_{L^2(C_\theta^\epsilon)}. 
\end{equation}
Now we will use the following inequality for the difference 
between the determinants of two $n\times n$ matrices 
$A=(A_1,\ldots,A_n)$ and $B=(B_1,\ldots,B_n)$, 
where $(A_i)_{1\le i\le n}$ and $(B_i)_{1\le i\le n}$ are the columns 
of $A$ and $B$ respectively, 

\begin{equation}
  \label{dadb}
|\det (A)-\det (B)|\le  \sum_{j=1}^n |\det(B_1,\ldots, B_{j-1},A_j-B_j,A_{j+1},\ldots,A_n)|. 
\end{equation}
Now choose $A=(K_u^{RW}(w_i,w_j))_{i,j=1}^n$ and 
$B=(K_{y,\epsilon}^{RW}(w_i,w_j))_{i,j=1}^n$, 
and apply (\ref{dadb}), the bounds (\ref{limkke}) of Step 1 and 
(\ref{leftk2}) of Step 4, and Hadamard's inequality 
to conclude that 

\begin{align*}
  &\left|\det(K_u^{RW}(w_i,w_j))_{i,j=1}^n-\det(K_{y,\epsilon}^{RW}(w_i,w_j))_{i,j=1}^n\right|\\
  &\le 
  n \left(n\frac{1}{\oldc{9}} e^{-2\oldca{1a}  t}\right)^{1/2}(n   \oldc{24}(\sigma^3 t)^{2/3})^{n/2}=
  (\oldc{25})^{n/2} n^{(n+1)/2} (\sigma^3 t)^{n/2} e^{-\oldca{1a} t}, 
\end{align*}
for some constants $\newc\label{24}=\oldc{24}(\theta,\epsilon, y, 
\upalpha, \upbeta)>0$ and $\newc\label{25}=\oldc{25}(\theta,\epsilon, y, 
\upalpha, \upbeta)>0$, which are independent of $\upalpha$ and $\upbeta$ for $\upalpha\ge 1$
and $\upbeta\ge 1$. 
As in Step 4, this implies that 

\begin{align}
\label{sumn1}&\sum_{n=1}^\infty\frac{1}{n!}\left|\int_{(C_\theta^\epsilon)^n}\det(K_u^{RW}(w_i,w_j))_{i,j=1}^ndw_1\cdots 
                    dw_n\right.\\
 \nonumber &-\left.\int_{(C_\theta^\epsilon)^n}\det(K_{y,\epsilon}^{RW}(w_i,w_j))_{i,j=1}^ndw_1\cdots 
    dw_n\right|\le 
 e^{-\oldc{26}\sigma^3 t 
  +2 (\oldc{27})^2 (\sigma^3t)^{2/3}}, 
\end{align}
for a pair of constants $\newc\label{26}=\oldc{26}(\theta,\epsilon, y, 
\upalpha, \upbeta)>0$ and 
$\newc\label{27}=\oldc{27}(\theta,\epsilon,y,\upalpha,\upbeta)>0$, 
which are independent of $\upalpha$ and $\upbeta$ for $\upalpha\ge 1$
and $\upbeta\ge 1$, 
As in Step 3, we conclude that either in the case in which $\upalpha$
and $\upbeta$ are constant, or in the case in which $(\upalpha_t)_
{t\ge 0}$ and $(\upbeta_t)_{t\ge 0}$ satisfy (\ref{as1}) and 
(\ref{as2}), 
the right-hand side of (\ref{sumn1}) tends to $0$ as $t\to\infty$
which combined with Step 4,  by the Fredholm determinant expansion 
implies (\ref{step3}). 

\medskip

\noindent {\bf Step 6.} We will now show that

$$
\lim_{t\to\infty}\det(I-K_{y,\epsilon}^{RW})_{L^2(C_\theta^\epsilon)}=
\det(I+K_y)_{L^2(C)}.
$$
First note that there is no problem in deforming the contour
$C_\theta^\epsilon$ to $W_\theta^L$, so that
it is enough to prove that

$$
\lim_{t\to\infty}\det(I-K_{y,\epsilon}^{RW})_{L^2(W_\theta^L)}
=\det(I+K_y)_{L^2(C)}.
$$
To prove this, we will first do a change of coordinates in the integration variables
of the Fredholm determinant expansion and in the integral defining the
kernel $K_{y,\epsilon}^{RW}$, introducing $\bar z$, $\bar v$ and $\bar
v'$ defined by

\begin{equation}
  \label{zt}
z=\theta+\frac{1}{\sigma(\theta)t^{1/3}}\bar z, 
v=\theta+\frac{1}{\sigma(\theta)t^{1/3}}\bar v\qquad {\rm and}\quad 
v'=\theta+\frac{1}{\sigma(\theta)t^{1/3}}\bar v'. 
\end{equation}
We then have that

$$
\det(I-K_{y,\epsilon}^{RW})_{L^2(W_\theta^L)}=
\det(I-\bar K_\epsilon^t)_{L^2(W^\infty_\theta)},
$$
where

$$
\bar K_\epsilon^t(\bar v,\bar v'):={\mathbf 1}_{|\bar v|,|\bar v'|\le
  \epsilon \sigma t^{1/3}}\frac{1}{\sigma(\theta)t^{1/3}}
K_{y,\epsilon}^{RW}\left(\theta+\frac{1}{\sigma(\theta)t^{1/3}}\bar
  v,\theta
  +\frac{1}{\sigma(\theta)t^{1/3}}\bar v'\right).
$$
We will first prove the pointwise convergence of the kernel $\bar 
K_\epsilon^t$.
Consider the contour $L_\epsilon:=D_0^{1,\epsilon \sigma t^{1/3},+}\cup S_0^1$ formed by the two vertical
lines $D_0^{1,\epsilon\sigma t^{1/3},+}:=\{yi: y\in [1,\epsilon\sigma t^{1/3}]\}\cup \{
yi:y\in [-1,-\epsilon\sigma t^{1/3}]\}$ and the
semicircle
$S_0^1:=\{z: |z|=1, \Re z\ge 0\}$. We adopt the convention $L_\infty$
as the contour $L_\epsilon$ with $\epsilon=\infty$. We then have

\begin{align*}
  &
\bar 
K_\epsilon^t(\bar v,\bar v')\\
&=\frac{{\mathbf 1}_{|\bar v|,|\bar v'|\le\epsilon\sigma
    t^{1/3}}}{2\pi i}
\int_{L_\epsilon}
\frac{\sigma^{-1}t^{-1/3}\pi}{\sin \left(\sigma^{-1}t^{-1/3}\pi (\bar 
                               z-\bar v)\right)} e^{t(\bar h(\bar 
                               z)-\bar h(\bar v))-
y(\bar z-\bar v)}\frac{\bar \Gamma (\bar v)}{\bar \Gamma(\bar 
                               z)}\frac{1}{\bar z-\bar v'}d\bar z,
\end{align*}
where

$$
\bar h(w)=h\left(\theta+\sigma^{-1}t^{-1/3}w\right) 
\quad {\rm and}\quad 
\bar \Gamma (w)=\Gamma (\theta+\sigma^{-1}t^{-1/3}w). 
$$
Now, from the limits (where again we use that  is (\ref{limti}) satisfied
both in the case $\upalpha$ and $\upbeta$ constant or
$(\upalpha_t)_{t\ge 0}$ and $(\upbeta_t)_{t\ge 0}$ satisfying
(\ref{as1}) and (\ref{as2})),

\begin{align*}
\lim_{t\to\infty}    \frac{\sigma(\theta)^{-1}t^{-1/3}\pi}{\sin (\sigma^{-1}t^{-1/3}\pi 
    (\bar z-\bar v))}&=\frac{1}{\bar z-\bar v}\\
   \lim_{t\to\infty}\frac{\bar\Gamma(\bar v)}{\bar\Gamma(\bar z)}&=1\\
    \lim_{t\to\infty}t(\bar h(\bar z)-\bar h(\bar v))&=\frac{1}{3}(\bar 
    z^3-\bar v^3), 
\end{align*}
 it follows that

 \begin{align*}
&\lim_{t\to\infty}\frac{\sigma^{-1}t^{-1/3}\pi}{\sin \left(\sigma^{-1}t^{-1/3}\pi (\bar 
                               z-\bar v)\right)} e^{t(\bar h(\bar 
                               z)-\bar h(\bar v))-
y(\bar z-\bar v)}\frac{\bar \Gamma (\bar v)}{\bar \Gamma(\bar 
z)}\frac{1}{\bar z-\bar v'}\\
&=\frac{1}{(\bar z-\bar v)(\bar z-\bar v')}e^{\frac{1}{3}(\bar z^3-\bar v^3)-y(\bar 
  z-\bar v)}. 
 \end{align*}
 We want to justify next that the above limit can be commuted with the
 integration over $\bar z$. Note that $v\in W_{\theta,\epsilon}^L$
 and $z\in D_\theta^t$ is implies that $\bar v\in W^\infty$ and $\bar
 z\in L_\infty$, so we can apply part $(iii)$ of Lemma \ref{th}
 to bound the exponential and conclude that

 \begin{align}
   \nonumber
   &\left|\frac{\sigma^{-1}t^{-1/3}\pi}{\sin \left(\sigma^{-1}t^{-1/3}\pi (\bar 
                               z-\bar v)\right)} e^{t(\bar h(\bar 
                               z)-\bar h(\bar v))-
y(\bar z-\bar v)}\frac{\bar \Gamma (\bar v)}{\bar \Gamma(\bar 
                     z)}\frac{1}{\bar z-\bar v'}\right|\\
\label{lastlast}   &\le
\oldc{28}\frac{|\Gamma(\bar v)|}{|\bar z- \bar v||\bar
z-\bar v'|} e^{-\oldconstant{100} |\bar v|^3}.
\end{align}
for some constant $\newc\label{28}>0$,
which is integrable in $\bar z$ at $\infty$ because the right-hand
side decays quadratically. We conclude then by the dominated convergence theorem that

\begin{align}
       \label{convconv}             \lim_{t\to\infty} K_\epsilon^t(\bar v, \bar v')
  =\int_{L_\infty} \frac{1}{(\bar z-\bar v)(\bar z-\bar v')}e^{\frac{1}{3}(\bar z^3-\bar v^3)-y(\bar 
  z-\bar v)}d\bar z.
\end{align}
Now, from (\ref{lastlast}), we can conclude that

$$
K_\epsilon^t(\bar v,\bar v')\le \oldc{fds} e^{-\oldconstant{100} |\bar v|^3},
$$
for some constant $\newc\label{fds}>0$.
By Hadamard's inequality for the determinant,
this implies that

$$
{\rm det} (K_\epsilon^t(\bar w_i,\bar w_j)_{i,j=1}^n\le 
n^{n/2} \prod_{i=1}^n\oldc{fds} e^{-\frac{\oldconstant{100}}{2} |\bar w_i|^3}. 
$$
It follows from this bound, the convergence in (\ref{convconv}) and
the dominated convergence theorem that

\begin{align*}
&\lim_{t\to\infty}\int_{(W_\theta^L)^n}\det 
(K^{RW}_{y,\epsilon}(w_i,w_j)) 
                    dw_1\cdots dw_n\\
  &=
\int_{(W_\theta^\infty)^n}\det 
(K_{y,}(\bar w_i,\bar w_j)) 
d\bar w_1\cdots d\bar w_n.
\end{align*}
Applying a second time the dominated convergence theorem to
interchange
the summation of the Fredholm determinant with the limit, we
finish the proofs of both parts $(i)$ and $(ii)$ of Proposition \ref{p1}.

\medskip

\subsection{Proof of Lemma \ref{iv}}
\label{sp00}
Here we will prove Lemma \ref{iv}.
Define

$$
\Xi(x)=\frac{\psi_2(\gamma)}{4\theta}+\frac{\psi_3(\gamma)}{4!}=
\frac{1}{4}\sum_{n=0}^\infty\frac{1}{x_n^4}-\frac{1}{2}
\sum_{n=0}^\infty \frac{1}{\theta}\frac{1}{x_n^3}.
$$
 As 
 \begin{align*}
 &   \frac{\sigma^3(\theta)}{2\theta} + 
    \frac{h^{(4)}(\theta)}{4!} = \frac{h^{(3)}(\theta)}{4\theta} + 
                                 \frac{h^{(4)}(\theta)}{4!}\\
   &                           = \Xi(\theta+\upalpha)-\Xi(\theta)+\frac{\psi_1(\theta)-\psi_1(\theta+\upalpha)}{\psi_1(\theta)-\psi_1(\theta+\upalpha+\upbeta)}\big(\Xi(\theta)-\Xi(\theta+\upalpha+\upbeta)\big) 
  \end{align*}
  goes to 0 when $\upbeta\to 0$, it is enough to prove that its derivative with respect to $\upbeta$ is positive or, making the change $y=\theta+\upalpha+\upbeta$, that 
  \[
    \frac{\psi_1(\theta)-\psi_1(\theta+\upalpha)}{(\psi_1(\theta)-\psi_1(y))^2}\Big[-\Xi'(y)\big(\psi_1(\theta)-\psi_1(y)\big)+\psi_2(y)\big(\Xi(\theta)-\Xi(y)\big)\Big]>0, 
  \]
  for $y>2\theta$. 
  
  The first factor is positive and the second one equals 
  \begin{align*}
   & -\sum_{m\ge0}\bigg(\frac{3}{2\theta}-\frac{1}{y_m}\bigg)\frac{1}{y_m^4}\sum_{n\ge0}\bigg(\frac{1}{\theta_n^2}-\frac{1}{y_n^2}\bigg)\\
   & +\sum_{m\ge0}\frac{1}{y_m^3}\Bigg[\sum_{n\ge0}\bigg(\frac{1}{\theta}-\frac{1}{2\theta_n}\bigg)\frac{1}{\theta_n^3} - \sum_{n\ge0}\bigg(\frac{1}{\theta}-\frac{1}{2y_n}\bigg)\frac{1}{y_n^3}\Bigg]\\
   & =-\sum_{m\ge0}\frac{3y_m-2\theta}{2\theta y_m^5}\sum_{n\ge0}\frac{y_n^2-\theta_n^2}{\theta_n^2y_n^2} + \sum_{m\ge0}\frac{1}{y_m^3}\Bigg[\sum_{n\ge0}\frac{2\theta_n-\theta}{2\theta\theta_n^4} - \sum_{n\ge0}\frac{2y_n-\theta}{2\theta y_n^4}\Bigg]. 
  \end{align*}
  We will prove the positivity of the resulting products of terms in the sums. First products of same index terms ($m=n$) 
  \begin{align*}
   -\frac{3y_n-2\theta}{2\theta y_n^5}\frac{y_n^2-\theta_n^2}{\theta_n^2y_n^2} + \frac{1}{y_n^3}\frac{2\theta_n-\theta}{2\theta\theta_n^4} - \frac{1}{y_n^3}\frac{2y_n-\theta}{2\theta y_n^4}. 
  \end{align*}
  After factoring the common denominator we are left with 
  \begin{align*}
  &  -(3y_n-2\theta)(y_n^2-\theta_n^2)\theta_n^2+(2\theta_n-\theta)y_n^4-(2y_n-\theta)\theta_n^4\\
   & = -(3a+3t-2\theta)((a+t)^2-a^2)a^2+(2a-\theta)(a+t)^4-(2a+2t-\theta)a^4, 
  \end{align*}
  where we have abbreviated $\theta_n$ as $a$ and $\upalpha+\upbeta$ as $t$, so that $y_n=a+t$ and $t>\theta$. After expanding the fourth power binomial we see that the constant and linear in $t$ terms cancel out 
  \begin{align*}
   & -(3a+3t-2\theta)(2a+t)ta^2+(2a-\theta)(4a^3t+6a^2t^2+4at^3+t^4)-2ta^4\\
     &   = -(9at+3t^2-2t\theta)ta^2+(2a-\theta)(6a^2t^2+4at^3+t^4)\\
& =t^2\big[3a^3-4a^2\theta+(5a^2-4a\theta)t+(2a-\theta)t^2\big]. 
\end{align*}
As $a\ge\theta$ it is clear that this expresion is positive for $t>\theta$. 

Now the products of terms with different index in the above sums 
\begin{align*}
  &-\frac{3y_m-2\theta}{2\theta y_m^5}\ \,\frac{y_n^2-\theta_n^2}{\theta_n^2y_n^2} + \frac{1}{y_m^3}\ \frac{2\theta_n-\theta}{2\theta\theta_n^4} -  \frac{1}{y_m^3}\ \frac{2y_n-\theta}{2\theta y_n^4}\\
  &-\frac{3y_n-2\theta}{2\theta y_n^5}\ \,\frac{y_m^2-\theta_m^2}{\theta_m^2y_m^2} + \frac{1}{y_n^3}\ \frac{2\theta_m-\theta}{2\theta\theta_m^4} -  \frac{1}{y_n^3}\ \frac{2y_m-\theta}{2\theta y_m^4}. 
\end{align*}
The first half can be written as 
\begin{align*}
 & -\frac{(3y_m-2\theta)(y_n^2-\theta_n^2)y_n^2}{2\theta y_m^5\theta_n^2y_n^4}+\frac{(2\theta_n-\theta)y_n^4-(2y_n-\theta)\theta_n^4}{2\theta y_m^3\theta_n^4y_n^4}\\
 & =  \frac{-(3y_m-2\theta)(y_n^2-\theta_n^2)y_n^2+(3y_n-2\theta)(y_n^2-\theta_n^2)y_m^2}{2\theta y_m^5\theta_n^2y_n^4}\\
 & +\,\frac{-(3y_n-2\theta)(y_n^2-\theta_n^2)\theta_n^2+(2\theta_n-\theta)y_n^4-(2y_n-\theta)\theta_n^4}{2\theta y_m^3\theta_n^4y_n^4}. 
\end{align*}
The last numerator was already shown to be positive. By a similar reasoning for the second half of the above expression, all that is left to finish the proof is to show that 
\begin{align*}
  &\frac{-(3y_m-2\theta)(y_n^2-\theta_n^2)y_n^2+(3y_n-2\theta)(y_n^2-\theta_n^2)y_m^2}{2\theta y_m^5\theta_n^2y_n^4}\\
  &+\frac{-(3y_n-2\theta)(y_m^2-\theta_m^2)y_m^2+(3y_m-2\theta)(y_m^2-\theta_m^2)y_n^2}{2\theta y_n^5\theta_m^2y_m^4}
\end{align*}
is positive for $y>2\theta$. 

This is equivalent to showing the positivity of 
\begin{align*}
  &-(3y_m-2\theta)(y_n^2-\theta_n^2)y_n^3\theta_m^2+(3y_n-2\theta)(y_n^2-\theta_n^2)y_m^2y_n\theta_m^2\\
  &-(3y_n-2\theta)(y_m^2-\theta_m^2)y_m^3\theta_n^2+(3y_m-2\theta)(y_m^2-\theta_m^2)y_n^2y_m\theta_n^2 
\end{align*}
Factorizing and using the abreviations $a=\theta_n$ and $b=\theta_m$, 
\begin{align*}
 &   \big[(y_n^2-\theta_n^2)\theta_m^2y_n - (y_m^2-\theta_m^2)\theta_n^2y_m\big]\big[(3y_n-2\theta)y_m^2 - (3y_m-2\theta)y_n^2\big]\\
   & =\big[(2a+t)tb^2(a+t) - (2b+t)ta^2(b+t)\big]\\
    &\times\big[3(a+t)(b+t)(b-a)-2\theta(a+b+2t)(b-a)\big]. 
  \end{align*}
  A further algebraic manipulation leads us to the following expresion, which is positive under the stated conditions. 
  \begin{align*}
    \big[(a+b)t^3+3abt^2\big]\big[3ab+3t^2+(a+b)(3t-2\theta)-4\theta t\big](b-a)^2 
  \end{align*}

    \medskip
    
\subsection{ Proof of Lemmas \ref{fo} and \ref{im}}
\label{lemmasfo}
 In what follows we present the proofs of Lemma \ref{fo} in Section
\ref{sfo} and of Lemma \ref{im} in Section \ref{sim}.

\medskip
\subsubsection{Proof of Lemma \ref{fo}}
\label{sfo}
The main technical ingredient in the proof of Lemma \ref{fo} will
be to extract the zero of $\Re(iz h'(z))$ for $z=e^{i\phi}$ at $\phi=0$,
through a subtraction of appropriate functions. This enables us to
extract two key factors ($\sin\phi$ and $(1-\cos\phi)$) from
$\Re(izh'(z))$. We will start
deducing several properties of a resulting function in these computations, defined
for  $x>0$,

\begin{equation}
  \label{mathcalp}
\mathcal P(x)=-\sum_{n=0}^\infty 
\frac{\theta^2+2\theta x_n\cos\phi}{              (\theta^2+2\theta x_n\cos\phi+x_n^2)(\theta+x_n)^2}, 
\end{equation}
where we adopt the convention $x_n=n+x$.

\medskip

\begin{lemma}
  \label{pestimate} Let $\theta\in (0,0.5)$. Then, the following are satisfied,

  \begin{itemize}

  \item[(i)] For all $\phi\in (0,\pi)$ such that $\cos\phi\ge 0$, the 
    function $-\mathcal P(x)$ is positive and decreasing in $x>0$. 

      \item[(ii)] For all $\phi\in (0,\pi)$ such that $\cos\phi\le-\frac{\theta}{2x} $, the 
    function $-\mathcal P(x)$ is negative.

    \item[(iii)] For all $\phi\in (0,\pi)$ and $x>0$ such that
      $\cos\phi\le-\frac{\theta}{x}$,     $-\mathcal P(x+y)$ is  increasing in $y>0$. 

    \item[(iv)] For all $\phi\in (0,\pi)$, $x\ge\theta$ and $y>0$ we
      have that
      
  \begin{align*}
    &  -(\mathcal P(x)-\mathcal P(x+y))\\
    &\le v(\rho,\phi) \mathbf 1(\cos\phi>-\rho) 
      \left(\Psi_1(\theta+x)-\Psi_1(\theta+x+y)\right), 
  \end{align*}
  where 

  $$
  v(\rho,\phi)=\frac{\rho^2+2\rho\cos\phi+\rho}{\rho^2+2\rho\cos\phi+1}\le \frac{\rho^2+3\rho}{(\rho+1)^2}, 
    $$
    and $\rho=\frac{\theta}{x}$. 
\end{itemize}
  \end{lemma}

  \begin{proof} The positivity of $-\mathcal P(x)$ for $\cos\phi\ge 0$
    of part $(i)$ is immediate. Also the negativity of $-\mathcal
    P(x)$
    for $\cos\phi\le-\frac{\theta}{2x}$, which proves part $(ii)$. We
    will now prove that $-\mathcal P(x)$ is decreasing in $x$ for
    $\cos\phi\ge 0$ (part $(i)$), part $(iii)$ and part $(iv)$. To simplify notation define 

    $$
C(u):=\theta+2 u\cos\phi, 
$$

$$
A(u):=\theta^2+2\theta u\cos\phi+u^2 
$$
and 

$$
B(u)=(\theta+u)^2. 
$$
Then note that 
    \begin{align*}
      &    \frac{ C(x_n) 
        }{A(x_n)B(x_n)}-\frac{C( (x+y)_n) 
                        }{A((x+y)_n)B((x+y)_n)}\\
&                        =
                        \frac{\theta+2x_n\cos\phi}{A(x_n)}\left(\frac{1}{B(x_n)}-\frac{1}{B((x+y)_n)}\right)\\
&                      -\frac{2 y\cos\phi}{A(x_n)}\frac{1}{B((x+y)_n)}
                                                                                                    +\frac{\theta+2(x+y)_n\cos\phi }{A((x+y)_n)}
                                                                                                    \frac{2y(\theta\cos\phi +x_n)+y^2 
                                                                                                                 }{A(x_n)B((x+y)_n)}\\
       &=
        \frac{\theta+2 x_n\cos\phi}{A(x_n)}\left(\frac{1}{B(x_n)}-\frac{1}{B((x+y)_n)}\right) \\
&        -\frac{2\cos\phi}{2\theta +2x_n+y}
        \frac{B(x_n)}{A(x_n)}\frac{2y\theta +2x_ny+y^2 
        }{B(x_n)B((x+y)_n)}\\
      &        +\frac{\theta+2 (x+y)_n\cos\phi}{A((x+y)_n)}
        \frac{B(x_n)}{A(x_n)}\frac{ 2(\theta\cos\phi +x_n)+y}{2\theta+2x_n+y}
        \frac{2\theta y+2x_ny+y^2 
                                                                               }{B(x_n)B((x+y)_n)}\\
     &=
        \left( \mathfrak b+\mathfrak c+\mathfrak a\right) 
        \left(\frac{1}{B(x_n)}-\frac{1}{B((x+y)_n)}\right), 
    \end{align*}
    where 

    $$
\mathfrak b=\frac{\theta+2 x_n\cos\phi}{\theta^2+2 x_n\theta\cos\phi+x_n^2}, 
$$

$$
\mathfrak c=-\frac{2\cos\phi}{2\theta+2x_n+y}
\frac{(\theta+x_n)^2}{\theta^2+2\theta x_n\cos\phi+x_n^2}
$$
and 

\begin{align*}
  &\mathfrak a\\
  &=\frac{\theta+2 (x_n+y)\cos\phi}{\theta^2+2\theta(x_n+y)\cos\phi+(x_n+y)^2}
        \frac{(x_n+\theta)^2}{\theta^2+2\theta x_n\cos\phi+x_n^2}\frac{ 2(\theta\cos\phi 
          +x_n)+y}{2(\theta+x_n)+y}. 
        \end{align*}
        Note that

        \begin{align}
          \nonumber
                &\mathfrak a+\mathfrak c\\
            \label{a+c}=&\frac{(2(x_n+y)(\theta+x_n\cos\phi)-\theta 
        y)(x_n+\theta)^2}
        {(\theta^2+2\theta x_n\cos\phi+x_n^2) 
          (\theta^2+2\theta(x_n+y)\cos\phi+(x_n+y)^2)(2\theta+2x_n+y)}.
        \end{align}
        From (\ref{a+c}) we can easily check that
        $\mathfrak 
        a+\mathfrak c\ge 0$ whenever 
$\cos\phi\ge 0$,
 which combined with the fact that the condition on $\phi$ also
implies that
$\mathfrak b\ge 0$,
proves that $-\mathcal P(x)$ is decreasing when $\cos\phi\ge 0$,
and proves part $(i)$. On the other hand, (\ref{a+c})
also shows that $\mathfrak a+\mathfrak c\le 0$ when
$\cos\phi\le-\frac{\theta}{x}$,
which combined with the fact that this condition also implies
that $\mathfrak b\le 0$, implies that $-\mathcal P(x+y)$ is increasing
in $y$ when $\cos\phi\le -\frac{\theta}{x}$ (part $(ii)$).

        Let us now prove part $(iii)$. Note that

        \begin{align*}
        &
       \mathfrak a+\mathfrak c\le\frac{2(x_n+y)(\theta+x_n\cos\phi)(x_n+\theta)^2}
        {(\theta^2+2\theta x_n\cos\phi+x_n^2) 
          (\theta^2+2\theta(x_n+y)\cos\phi+(x_n+y)^2)(2\theta+2x_n+y)}. 
        \end{align*}
        Now note that for all $\phi\in (0,\pi)$, the function 

        $$
f_1(v)=\frac{v}{\theta^2+2\theta v\cos\phi+v^2 }, 
        $$
is  decreasing in $v$ as long as $v\ge\theta$. Therefore, since by 
assumption 
we have that $x\ge\theta$, it follows that $x_n+y\ge\theta$, so that 
for all $n\ge 0$, 

        $$
     (  \mathfrak a+\mathfrak c)\theta\le\frac{\theta x_n(\theta+x_n\cos\phi)(x_n+\theta)}
     {(\theta^2+2\theta x_n\cos\phi+x_n^2)^2}
     \le 
     \frac{\rho_n 
       (\rho_n+\cos\phi)(\rho_n+1)}{(\rho_n^2+2\rho_n\cos\phi+1)^2}
     1(\cos\phi>-\rho), 
     $$
     where $\rho_n:=\frac{\theta}{x_n}$. Now, consider the function

     $$
f_2(u,a):=    \frac{
      u+a}{u^2+2ua+1}. 
    $$
    Note that for $u\in (0,1)$, 

    $$
    \frac{\partial f_2(u,a)}{\partial a}
    =\frac{1-u^2}{(u^2+2ua+1)^2}>0, 
    $$
    which shows that for fixed $u$, $f_2(u,a)$  is increasing in 
    $a$. Hence, for all $n\ge 0$,

    $$
    \frac{\rho_n(\rho_n+\cos\phi)(\rho_n+1)}{(\rho_n^2+2\rho_n\cos\phi+1)^2}
    \le 
    \frac{\rho_n}{\rho_n^2+2\rho_n\cos\phi+1}. 
    $$
        It follows that

    $$
    (\mathfrak a+\mathfrak b+\mathfrak c)\theta\le 
    f_3(\rho_n,\cos\phi) {\mathbf 1}(\cos\phi>-\rho), 
    $$
    where we define for $u\in [0,1]$ and $a\in (-1,1)$,

    $$
f_3(u,a)=\frac{u^2+2ua+u}{u^2+2ua+1}. 
$$
Now note that for $u\in (0,1)$ and $a$ arbitrary,

$$
\frac{\partial f_3(u,a)}{\partial a}=\frac{2u(1-u)}{(u^2+2ua+1)^2}>0. 
$$
Hence,

$$
f_3(\rho_n,a)\le f_3(\rho,a).
$$
Now for $a\in  (-u,1)$ and $u\in (0,1)$, 

$$
\frac{\partial f_3(u,a)}{\partial u}=\frac{1+2u+2a-u^2}{(u^2+2ua+1)^2}>0. 
$$
It follows that $f_3(u,a)$ is increasing in $a$ for fixed $u$ and 
increasing in $u$ for fixed $a$ as long as $u,a\in (0,1)$. 
Therefore, since $\rho_n\le\rho$,  
 we conclude that 

$$
(\mathfrak a+\mathfrak b+\mathfrak c)\theta\le 
\frac{\rho^2+2\rho\cos\phi+\rho}{\rho^2+2\rho\cos\phi+1}
{\mathbf 1}(\cos\phi>-\rho)\le \frac{\rho^2+3\rho}{(\rho+1)^2}
{\mathbf 1}(\cos\phi>-\rho). 
$$
    \end{proof}

\medskip 

Let us now proceed to prove Lemma \ref{fo}. Let $z=\theta e^{i\phi}$. Note that 

\begin{align}
  \nonumber 
  \Re(iz h'(z))&=\Theta(\upalpha)-\Theta(\upalpha+\upbeta)\\
      \nonumber 
&+\frac{\Psi_1(\upalpha+\theta)-\Psi_1(\upalpha+\upbeta+\theta)}{
  \Psi_1(\theta)-\Psi_1(\upalpha+\upbeta+\theta)}\left(
\Theta(\upalpha+\upbeta)-\Theta(0)\right), 
\end{align}
where for $\gamma$ real we define 

$$
\Theta(\gamma):=\Re(iz(\Psi_1(z+\gamma)-\Psi_1(\theta+\gamma)). 
$$
Now, note that for $\gamma>0$, 
we have the following expansion valid for any $z\notin 
\{0,-1,-2,\dots\}$, 

  \begin{equation}
    \label{five}
\Psi(z+\gamma)-\Psi(\theta+\gamma)=\sum_{n\ge0}\frac{z-\theta}{(z+\gamma+n)(\theta+\gamma+n)}. 
  \end{equation}
Also, 
  \begin{align*}
    &Re\left(iz\frac{z-\theta}{\gamma_n+z}\right) \\
    &= Re\left( i\theta^2 
    \frac{(\cos\phi+i\sin\phi)(\cos\phi-1+i\sin\phi)(\theta\cos\phi+\gamma_n 
      -i\theta\sin\phi)}{(\theta\cos\phi+\gamma_n)^2 
      +\theta^2\sin^2\phi}\right)\\ & = Re\left( i\theta^2 
    \frac{\left(\cos^2\phi-\sin^2\phi-\cos\phi +i\sin\phi(2\cos\phi-1) 
      \right) (\theta\cos\phi+\gamma_n-i\theta\sin\phi) 
    }{(\theta\cos\phi+\gamma_n)^2 +\theta^2\sin^2\phi}\right)\\ &
    =\theta^2\frac{\theta\sin\phi(\cos^2\phi-\sin^2\phi-\cos\phi) 
      -\sin\phi(2\cos\phi-1)(\theta\cos\phi+\gamma_n)}{
      (\theta\cos\phi+\gamma_n)^2 +\theta^2\sin^2\phi}\\ &
    =\theta^2\sin\phi \frac{-\theta -\gamma_n(2\cos\phi-1)}{ \theta^2 
      +2\theta \gamma_n\cos\phi+\gamma_n^2}. 
  \end{align*}
This implies that 

\begin{equation}
  \label{eqth}
\Theta(\gamma)=\theta^2\sin\phi \mathcal R(\gamma), 
\end{equation}
where 

$$
\mathcal R(\gamma):=-\sum_{n=0}^\infty\frac{\theta+\gamma_n(2\cos\phi-1)}{
  (\theta^2+2\theta\gamma_n\cos\phi+\gamma_n^2)(\theta+\gamma_n)}. 
$$
Hence it is enough to prove that for $\phi\in (0,\pi)$, 

$$
\mathcal R (\upalpha)-\mathcal R(\upalpha+\upbeta) 
+\frac{\Psi_1(\upalpha+\theta)-\Psi_1(\upalpha+\upbeta+\theta)}{
  \Psi_1(\theta)-\Psi_1(\upalpha+\upbeta+\theta)}\left(
\mathcal R(\upalpha+\upbeta)-\mathcal R(0)\right)>0. 
$$
On the other hand note that 

\begin{align}
\nonumber &\mathcal R(\gamma)+\Psi_1(\theta+\gamma)\\
\nonumber &=\sum_{n=0}^\infty\left(\frac{1}{(\theta+\gamma_n)^2}
  -\frac{\theta+\gamma_n(2\cos\phi-1)}{
                                                            (\theta^2+2\theta\gamma_n\cos\phi+\gamma_n^2)(\theta+\gamma_n)}\right)\\
\label{eqr}  &=
2     (1-\cos\phi)\mathcal Q(\gamma), 
\end{align}
where we define 

$$
\mathcal Q(\gamma)=\sum_{n=0}^\infty\frac{\gamma_n^2}{
  (\theta^2+2\theta\gamma_n\cos\phi+\gamma_n^2)(\theta+\gamma_n)^2}. 
$$
Therefore it is now enough to prove that 
 for $\phi\in (0,\pi)$, 

$$
\mathcal Q (\upalpha)-\mathcal Q(\upalpha+\upbeta) 
+\frac{\Psi_1(\upalpha+\theta)-\Psi_1(\upalpha+\upbeta+\theta)}{
  \Psi_1(\theta)-\Psi_1(\upalpha+\upbeta+\theta)}\left(
\mathcal Q(\upalpha+\upbeta)-\mathcal Q(0)\right)>0. 
$$
Now note that 

\begin{align}
\nonumber &\mathcal Q(\gamma)-\Psi_1(\theta+\gamma)\\
\nonumber&=
 \sum_{n=0}^\infty\left(
  \frac{\gamma_n^2}{
                                              (\theta^2+2\theta\gamma_n\cos\phi+\gamma_n^2)(\theta+\gamma_n)^2}
                                              -\frac{1}{(\theta+\gamma_n)^2}\right)\\
 \label{eqq} &=
\mathcal P(\gamma), 
\end{align}
where we recall the definition of $\mathcal P$ in (\ref{mathcalp}) above. 
So at this point it is enough to prove that for $\phi\in (0,\pi)$, 

\begin{equation}
  \label{newproof}
\mathcal P (\upalpha)-\mathcal P(\upalpha+\upbeta) 
+\frac{\Psi_1(\upalpha+\theta)-\Psi_1(\upalpha+\upbeta+\theta)}{
  \Psi_1(\theta)-\Psi_1(\upalpha+\upbeta+\theta)}\left(
\mathcal P(\upalpha+\upbeta)-\mathcal P(0)\right)>0. 
\end{equation}
Indeed, the lower bound in (\ref{5.5a})  now follows from the fact that
(\ref{newproof}), 
(\ref{eqth}), (\ref{eqr}) and (\ref{eqq}) imply that

\begin{align*}
  &\Re (iz h'(z))=\\
  &2\theta^2\sin\phi(1-\cos\phi)\left(\mathcal
P(\upalpha)
-\mathcal
    P(\upalpha+\upbeta)+\right.\\
  &\left.\frac{\Psi_1(\upalpha+\theta)-
\Psi_1(\upalpha+\upbeta+\theta)}{\Psi_1(\theta)-\Psi_1(\upalpha+\upbeta+\theta)}(\mathcal
P(\upalpha+\upbeta)-\mathcal P(0))
\right).
\end{align*}
From here we see that for $\theta\in [0,0.5]$ ($\theta=0$ interpreted
as the limit when $\lim_{\theta\to 0^+}$), $\upalpha\le 1$ and
$\upbeta\le 1$, the factor multiplying $2\theta^2\sin\phi
(1-\cos\phi)$ in the above expression, is positive, while as
$\upalpha\to\infty$ and $\upbeta\to\infty$, the dominant term in this
factor
is $-\frac{\mathcal
  P(0)}{\Psi_1(\theta)}(\Psi_1(\upalpha+\theta)-\Psi_1(\upalpha+\upbeta+\theta))$,
which gives by compactness the lower bound of (\ref{5.5a}).
Let us now introduce the parameter $\rho:=\frac{\theta}{\upalpha}$.  

We will now continue with the proof of (\ref{newproof}) which
will be divided in four cases: Case 1, when $\cos\phi\le-\rho$;
Case 2, when $-\rho\le\cos\phi\le -\rho/2$;  Case 3, when $-\rho/2
\le\cos\phi\le 0$; and Case 4 when
$0\le\cos\phi\le 1$. 
To do this, first  note that

\begin{equation}
  \label{pos}
-\mathcal P(0)=\sum_{n=0}^\infty\frac{\theta^2+2\theta n\cos\phi}{
  (\theta^2+2\theta n\cos\phi +n^2)(\theta+n)^2}. 
\end{equation}
In what follows we will assume that $\rho\le 1$.
We will use the following consequence of part $(iv)$ of Lemma
\ref{pestimate},
valid for all 
$\theta<\min(0.5,\upalpha)$ and $\upbeta>0$, 

\begin{align}
 \nonumber &-\left(\mathcal P(\upalpha)-\mathcal P(\upalpha+\upbeta)\right)\\
   &
     \label{leftl}
     \le 
\frac{\rho^2+2\rho\cos\phi+\rho}{\rho^2+2\rho\cos\phi+1}{\mathbf 1}(\cos\phi>-\rho) 
                    \left(\Psi_1(\theta+\upalpha)\! -\!\Psi_1(\theta+\upalpha+\upbeta)\right). 
\end{align}

\medskip

\noindent {\bf Case 1} ($\cos\phi\le-\rho$).
By part $(ii)$ of Lemma \ref{pestimate}, as well as by
(\ref{leftl}),
the term $-(\mathcal P(\upalpha)-\mathcal
P(\upalpha+\upbeta))\le 0$ and $\mathcal P(\upalpha+\upbeta)\ge 0$.
Therefore it is enough to prove that $-\mathcal P(0)>0$ (whenever
$\theta\in (0,1)$).
Note that the series (\ref{pos})  achieves it's minimum value at $\phi=-\pi$, so 
we have 

\begin{align}
  \nonumber 
&-\mathcal P(0)\ge \sum_{n=0}^\infty\frac{\theta^2-2\theta 
  n}{(n^2-\theta^2)^2}\\
  \nonumber&=\frac{1}{\theta^2}+\frac{\theta^2-2\theta}{(1-\theta^2)^2}
-\sum_{n=2}^\infty 
\frac{2\theta n-\theta^2}{(n^2-\theta^2)^2}\\
  &
 \nonumber
    \ge 
\frac{1}{\theta^2}+\frac{\theta^2-2\theta}{(1-\theta^2)^2}
-\sum_{n=2}^\infty 
    \frac{ n-\frac{1}{4}}{\left(n^2-\frac{1}{4}\right)^2}
    \ge 
    \frac{1}{\theta^2}+\frac{\theta^2-2\theta}{(1-\theta^2)^2}
    -0.2\\
  \label{poo}
  &  \ge 4-\frac{4}{3}-\frac{1}{5}>0,
\end{align}
where we used in the last inequality the fact that the terms in the 
series are decreasing in $\theta$ and the minimum value is achieved 
for $\theta=0.5$. 

\medskip

\noindent {\bf Case 2} ($-\rho\le \cos\phi\le -\rho/2$).
By the assumption $\cos\phi\le-\frac{\rho}{2}$ and part $(ii)$
 Lemma \ref{pestimate}, we see that  $-\mathcal P(\upalpha+\upbeta)\le 0$.
 This time,  the
 series (\ref{pos})  achieves it's minimum value $\cos\phi=-\rho$. But
 we
 will assume that $\rho\le 0.72$, so  
we see from (\ref{poo}), 
 that it is enough to prove that

 \begin{align*}
   &\frac{1}{\Psi_1(\theta)-\Psi_1(\upalpha+\upbeta+\theta)}\left(\frac{1}{\theta^2}
     +\frac{\theta^2-2\theta\times 
   0.72}{\left(\theta^2-2\theta\times 
   0.72+1\right)(1+\theta)^2}
   -0.2\right)\\
   &\ge 
\frac{\rho^2+2\rho\cos\phi+\rho}
{\rho^2+2\rho\cos\phi+1}\ge\rho. 
 \end{align*}
Now,
$\Psi_1(\theta)-\Psi_1(\theta+\upalpha)\ge\Psi_1(\theta)$.
Hence, it would be enough to prove that

$$
\frac{1}{\theta^2\Psi_1(\theta)}
\left(1-\theta^2 \frac{2\theta\times 
   0.72-\theta^2}{\left(\theta^2-2\theta\times 
   0.72+1\right)(1+\theta)^2}
-0.2\times\theta^2\right)
\ge\rho. 
 $$
 Now, since $\theta^2\Psi_1(\theta)$ is increasing in $\theta$, we
 see that the inequality is satisfied for all $\rho\in (0,0.7]$ such
 that

 $$
 \frac{1}{\frac{1}{4}\Psi_1\left(\frac{1}{2}\right)}
\left(1-\frac{1}{4} \frac{ 
   0.72-\frac{1}{4}}{\left(\frac{1}{4}-
   0.72+1\right)\left(1+\frac{1}{2}\right)^2}
-0.2\times\frac{1}{4}\right)
\ge\rho,
 $$
 or

 $$
\rho\le 0.74124\cdots.
$$

\medskip

\noindent {\bf Case 3} ($-\rho/2\le\cos\phi\le 0$).
Let us first bound,

\begin{align}
    \nonumber 
&-\mathcal 
                 P(\upalpha+\upbeta)\le \frac{\theta^2}
                 {(\theta^2+(\upalpha+\upbeta)^2)(\upalpha+\upbeta+\theta)^2}\\
  \nonumber
  &                 +
\sum_{n=1}^\infty\frac{\theta^2}{(\theta^2+
               (\upalpha+\upbeta)_n^2)((\upalpha+\upbeta)_n+\theta)^2}\\
  \label{here}&
    =\frac{1}{\theta^2}\frac{\rho^2}{(1+\rho^2)(1+\rho)^2}
    +\frac{\theta^2}{\theta^2+1}\Psi_1\left(1\right).
\end{align}
On the other hand,

\begin{align*}
&-\mathcal
P(0)\ge\sum_{n=0}^\infty\frac{\theta^2-\theta n}{(\theta^2-\theta
  n+n^2)(\theta+n)^2}\\
&=\frac{1}{\theta^2}-\frac{\theta-\theta^2}{(1-(\theta-\theta^2))(\theta+1)^2}
-\sum_{n=2}^\infty\frac{n\theta-\theta^2}{(n^2-(n\theta
                          -\theta^2))(\theta+n)^2}\\
  &\ge
    \frac{1}{\theta^2}-\frac{1}{3}-\sum_{n=2}^\infty\frac{2n-1}
    {\left(4n^2-2n+1\right)n^2}\ge
     \frac{1}{\theta^2}-\frac{1}{3}-0.65.
\end{align*}
As in case 2, we now see that it would be enough to prove that

\begin{align*}
  &\frac{1}{\frac{1}{4}\Psi_1\left(\frac{1}{2}\right)}
\left(
  \frac{1}{\theta^2}\left(1-\frac{\rho^2}{(1+\rho^2)(1+\rho)^2}\right) 
    -\frac{1}{3}-0.65 
    -\frac{\theta^2}{\theta^2+1}\Psi_1(1)\right)\\
  &\ge\frac{\rho}{\rho^2+1}. 
\end{align*}
But the left-hand side is bounded from below by the case in which
$\theta=0.5$,
from where we see that we have to show that

\begin{align*}
  &\frac{1}{\frac{1}{4}\Psi_1\left(\frac{1}{2}\right)}
\left(
  4\left(1-\frac{\rho^2}{(1+\rho^2)(1+\rho)^2}\right) 
    -\frac{1}{3}-0.65 
    -\frac{1}{5}\Psi_1(1)\right)\\
  &\ge\frac{\rho}{\rho^2+1},
\end{align*}
which is satisfied for all $\rho>0$.

\medskip
\noindent {\bf Case 4} ($ 0\le\cos\phi\le 1$).
Since by part $(ii)$ of Lemma \ref{pestimate} the difference 
$\Psi_1(\upalpha+\theta)-\Psi_1(\upalpha+\upbeta+\theta)$ is positive, 
dividing (\ref{newproof}) by this quantity, we see that now it is enough to 
show that 

\begin{align}
\label{a2f}
&\frac{\mathcal 
  P(\upalpha+\upbeta)-\mathcal P(0)}{\Psi_1(\theta)-\Psi_1(\upalpha+\upbeta+\theta)}
              >\frac{\rho^2+2\rho\cos\phi+\rho}{\rho^2+2\rho\cos\phi+1}
              {\mathbf 1}(\cos\phi>-\rho).
\end{align}
To do this, first  note that

$$
-\mathcal P(0)=\sum_{n=0}^\infty\frac{\theta^2+2\theta n\cos\phi}{
  (\theta^2+2\theta n\cos\phi +n^2)(\theta+n)^2}. 
$$ 
Therefore, the minimum value
of $-\mathcal P(0)$ is achieved for $\phi=0$, so that

\begin{align}
  \nonumber 
&-\mathcal P(0)\ge \sum_{n=0}^\infty\frac{\theta^2}{
 \left(\theta^2
               +n^2\right)\left(\theta+n\right)^2}
             =\frac{1}{\theta^2}-\mathcal P(1),
\end{align}
where $-\mathcal P(1)$ in the left-hand side is evaluated
at $\phi=0$.
On the other hand,  for $-\mathcal P(\upalpha+\upbeta)$
 we have that 

\begin{align}
    \nonumber 
&-\mathcal 
                 P(\upalpha+\upbeta) 
                 =\frac{\theta^2+2\theta(\upalpha+\upbeta)\cos\phi}
                 {(\theta^2+2\theta(\upalpha+\upbeta)\cos\phi+(\upalpha+\upbeta)^2)(\upalpha+\upbeta+\theta)^2}\\
                 \nonumber &                 +
\sum_{n=1}^\infty\frac{\theta^2+2\theta (\upalpha+\upbeta)_n 
                 \cos\phi}{(\theta^2+2\theta(\upalpha+\upbeta)_n\cos\phi+
                 (\upalpha+\upbeta)_n^2)((\upalpha+\upbeta)_n+\theta)^2}\\
\label{pab}&\le 
    \nonumber 
               \frac{\theta^2+2\theta\upalpha\cos\phi}
                 {(\theta^2+2\theta\upalpha\cos\phi+\upalpha^2)(\upalpha+\theta)^2}
             -\mathcal P(\upalpha+1)\\
 & \le
     \frac{1}{\theta^2}   \frac{\rho^2+2\rho}
                 {(\rho+1)^4}
             -\mathcal P(\upalpha+1),
\end{align}
where we have used part $(i)$ of Lemma \ref{pestimate}. Using again part
$(i)$
of Lemma \ref{pestimate}, we then conclude that

\begin{align*}
&\frac{\mathcal P(\upalpha+\upbeta)-\mathcal 
  P(0)}{\Psi_1(\theta)-\Psi_1(\upalpha+\upbeta+\theta)}
\ge \frac{1}{\theta^2 \Psi_1(\theta)}
  \left(1-\frac{\rho^2+2\rho}{(\rho+1)^4}\right)\\
  &\ge
\frac{1}{\frac{1}{4} \Psi_1\left(\frac{1}{2}\right)}
\left(1-\frac{\rho^2+2\rho}{(\rho+1)^4}\right).
\end{align*}
Therefore, by part $(iv)$ of Lemma \ref{pestimate}, it is enough
to show that

$$
\frac{1}{\frac{1}{4} \Psi_1\left(\frac{1}{2}\right)}
\left(1-\frac{\rho^2+2\rho}{(\rho+1)^4}\right)
\ge \frac{\rho^2+3\rho}{(\rho+1)^4},
$$
which is satisfied for all $\rho>0$.

\medskip

\subsubsection{Proof of Lemma \ref{im}}
\label{sim}
The following lemma will be useful to prove Lemma \ref{im}.

\begin{lemma}
\label{integral_doble}
Let $f$ and $g$ be twice continuously differentiable real functions 
defined on an interval containing $u<v<w$. If $f$ is convex and 
strictly decreasing, and $(g''f'-g'f'')(x)\ge\rho(x)$, with $\rho\ge0$ measurable, then 
\begin{align*}
g(v)-g(w)-\frac{f(v)-f(w)}{f(u)-f(w)}(g(u)-g(w))\ge\frac{f(v)-f(w)}{f(u)-f(w)}\int_u^v(u-x)\frac{\rho(x)}{f'(x)}dx. 
\end{align*}

\end{lemma}
        
\begin{proof} 
By the chain rule and the inverse function theorem $G=g\circ f^{-1}$ is 
continuously differentiable on an open interval containing $a<b<c$, 
the images under $f$ of $w,v,$ and $u$, respectively. Also, as $G'=(g'/f')\circ f^{-1}$, $G'$ is also continuously differentiable there, with 
\[
G''= \frac{g''f'-g'f''}{(f')^3}\circ f^ {-1}. 
\]
By hypothesis $f'$ is negative, so that $G''\le \rho/(f')^3\le 0$, that is, we 
are assuming  $G$ is concave. 
Now let us integrate from $b$ to $y\ge b$ and apply the change of variables $u=f(x)$ to get 
\begin{align*}
G'(y)-G'(b)\le\int_b^y\frac{\rho}{(f')^3}\circ f^ {-1}(u)du=\int_{f^{-1}(b)}^{f^{-1}(y)}\frac{\rho(x)}{f'(x)^2}dx. 
\end{align*}
Integrating from $b$ to $c$ we have 
\[
G(c)-G(b)-(c-b)G'(b)\le\int_b^c\int_{f^{-1}(b)}^{f^{-1}(y)}\frac{\rho(x)}{f'(x)^2}\,dx\,dy. 
\]
Reversing the order of integration --remember $f$ is decreasing-- last 
integral equals 
\[
  \int_{f^{-1}(b)}^{f^{-1}(c)}\int_{f(x)}^c\frac{\rho(x)}{f'(x)^2}\,dy\,dx=\int_u^v(f(x)-f(u))\frac{\rho(x)}{f'(x)^2}\,dy\,dx. 
\]
Now, let $t\in[a,b]$ be such that $G'(t)=(G(b)-G(a))/(b-a)$. As $G'$ is 
decreasing, $G'(t)\ge G'(b)$. In a similar way we have 
$f(x)-f(u)=f'(\eta)(x-u)$ for an appropriate $\eta\in[u,x]$, and 
$f'(\eta)\le f'(x)$. With both these inequalities we can write 
\[
G(c)-G(b)-\frac{c-b}{b-a}(G(b)-G(a))\le\int_u^v(x-u)\frac{\rho(x)}{f'(x)}\,dx\,dy. 
\]
Finally, as 
\begin{align*}
  &G(b)-G(a)-\frac{b-a}{c-a}(G(c)-G(a))\\
  =&(G(b)-G(a))\frac{c-b}{c-a}-(G(c)-G(b))\frac{b-a}{c-a}, 
\end{align*}
multiplying last inequality by $-(b-a)/(c-a)$, we obtain the desired result. 
\end{proof}

\medskip

Let us now prove Lemma \ref{im}.
We have, for $x>0$ and $y$ real, 
\[
\Im\Psi(x+iy)=y\,\Phi(x,y), 
\]
where  
$\Phi$ is  
defined in (\ref{phixy}). Note that $\Im h'(\theta+iy)=yH(\theta,y,\upalpha,\upbeta)$, 
and
\begin{align*}
H(\theta,y,\upalpha,\upbeta)=\Phi(\theta+\upalpha,y)-\Phi(\theta+\upalpha+\upbeta,y)-K_1(\Phi(\theta,y)-\Phi(\theta+\upalpha+\upbeta,y)). 
\end{align*}
Applying again Lemma \ref{integral_doble} we just need to show 
\[
(\Phi''\Psi_1'-\Phi'\Psi_1'')(x)\ge -8y^2\Psi_1'(x)\sum_{n\ge0}\frac1{(x_n^2+y^2)^3}, 
\]
where the derivatives of $\Phi$ are taken with respect to its first
variable and the second variable is set as $y$.
Calculating the derivatives and replacing, 
in particular 
\[
\Phi''(x)=\sum_{n\ge0}\frac{6}{(x_n^2+y^2)^2}-\sum_{n\ge0}\frac{8y^2}{(x_n^2+y^2)^3}, 
\]
the above inequality is equivalent to 
\[
  \bigg(\sum_{n\ge0}\frac1{x_n^4}\bigg)\bigg(\sum_{n\ge0}\frac{x_n}{(x_n^2+y^2)^2}\bigg)-
  \bigg(\sum_{n\ge0}\frac{x_n}{x_n^4}\bigg)\bigg(\sum_{n\ge0}\frac1{(x_n^2+y^2)^2}\bigg)\ge0, 
\]
and this follows by looking at the $m$-$n$ products (the $m=n$ terms are cero), 
\begin{align*}
\frac1{x_m^4}\frac{x_n}{(x_n^2+y^2)^2}+\frac1{x_n^4}\frac{x_m}{(x_m^2+y^2)^2}-
\frac{x_m}{x_m^4}\frac1{(x_n^2+y^2)^2}-\frac{x_n}{x_n^4}\frac1{(x_m^2+y^2)^2}\\
=(x_m-x_n)\bigg(\frac1{x_n^4}\frac1{(x_m^2+y^2)^2}-\frac1{x_m^4}\frac1{(x_n^2+y^2)^2}\bigg)\\
=y^2(x_m+x_n)(x_m-x_n)^2\frac{2x_m^2x_n^2+y^2(x_m^2+x_n^2)}{x_n^4(x_m^2+y^2)^2x_m^4(x_n^2+y^2)^2}\ge0. 
\end{align*}

  \medskip

  \section{Perturbative results}
  \label{sfour}
    Here we will prove Theorem \ref{moments} and Corollary
    \ref{corollary1}. We first need to derive several estimates
    about Dirichlet and Beta random variables.
    
    Consider a  random vector $X=(X_1,\ldots,X_k)$ having
    a Dirichlet distribution of parameters
$\upalpha=(\upalpha_1,\ldots,\upalpha_k)$. Let $\bar\xi_i=\log
X_i-M_i$ for $1\le i\le k$,
where $M_i$ are constants whose value will be given later.
We want to obtain estimates for the moments,

\begin{equation}
  \label{logm}
  L_{i_1,\ldots,i_k}(\upalpha):=\mathbb
E_{\upalpha}[\bar\xi_1^{i_1}
\cdots \bar\xi_k^{i_k}],
\end{equation}
for the shifted logarithmic moments of the random vector $X$, with 
$k\ge 0$, and $i_1,\ldots,i_k\ge 0$. We will call $i_1+\cdots+i_k$
the {\it degree} of the shifted moment. Consider the
following function, which we will call the logarithmic
partition function, defined by

$$
A(\upalpha):=\sum_{i=1}^k \left(\log (\Gamma (\upalpha_i))-\upalpha_i
M_i\right)-\log\Gamma\left(\sum_{i=1}^k \upalpha_i\right).
$$
Let also

$$
B(\upalpha)=\frac{\prod_{i=1}^k\Gamma(\upalpha_i)}{\Gamma\left(\sum_{i=1}^k\upalpha_i\right)}
=e^{A(\upalpha)}.
$$
Note that

\begin{equation}
  \label{li1}
        L_{i_1,\ldots,i_k}(\upalpha):=\frac{1}{B(\upalpha)}
        \frac{\partial^{i_1+\cdots+i_k}}{\partial^{i_1} \upalpha_1 
          \cdots \partial^{i_k} \upalpha_k}B(\upalpha). 
        \end{equation}
  From this identity we can recursively compute the shifted moments
  $L_{i_1,\ldots, i_k}(\upalpha)$. Nevertheless, we want a statement
  giving us a sharp asymptotic bound on the decay of these moments as
  the parameters tend to $\infty$ or to $0$.
  We will now define the families of functions corresponding to the
  higher order derivatives of the logarithmic partition function as,

  $$
 A_{i_1,\ldots, i_k}(\upalpha):=\frac{\partial^{i_1+\cdots+i_k}}{\partial^{i_1} \upalpha_1 
          \cdots \partial^{i_k} \upalpha_k}A(\upalpha). 
  $$
    We will call $i_1+\cdots+i_k$ the {\it degree} of
    $A_{i_1,\ldots,i_k}$ and we will use the notation

    $$
\mathcal A_n:=\{A_{i_1,\ldots,i_k}: i_1+\cdots+i_k=n\},
$$
for the set of functions of degree $n\ge 1$.  
Furthermore, we define the {\it degree} of a product of functions,

$$
\prod_{j=1}^l f_j,
$$
where $f_j\in\mathcal A_{n_j}$ for some $n_j\ge 1$, as the sum
$n_1+\ldots+n_l$.
Now, note from (\ref{li1}), that we have that

$$
L_{1,0,\ldots,0}(\upalpha)=A_{1,0,\ldots,0},
$$
with analogous equalities for $L_{0,1,0,\ldots,0}$ up to
$L_{0,\ldots,0,1}$.
Furthermore we have
the following recursion
formulas,

\begin{align}
  \nonumber
&      L_{i_1+1,i_2,\ldots,i_k}(\upalpha)=L_{1,0,\ldots,0}(\upalpha) 
            L_{i_1,\ldots ,i_k}(\upalpha)+
      \frac{\partial}{\partial\upalpha_1}
      L_{i_1,\ldots,i_k}(\upalpha),\\
\nonumber &      L_{i_1,i_2+1,i_3,\ldots,i_k}(\upalpha)=L_{0,1,0\ldots,0}(\upalpha) 
            L_{i_1,\ldots ,i_k}(\upalpha)+
      \frac{\partial}{\partial\upalpha_2}
                                        L_{i_1,\ldots,i_k}(\upalpha), \\
       \nonumber  &\vdots\\
    \label{recursion}   & L_{i_1,\ldots, i_{k-1}, i_k+1}(\upalpha) =L_{0,\ldots,0,1}(\upalpha) 
            L_{i_1,\ldots ,i_k}(\upalpha)+
      \frac{\partial}{\partial\upalpha_k}
                                        L_{i_1,\ldots,i_k}(\upalpha). 
      \end{align}
      From these recursion formulas, we can prove the following lemma.

    \medskip
    \begin{lemma}
      \label{logmlemma}Consider the logarithmic moments
      $L(i_1,\ldots, i_k)$, $i_j\ge 0$, $1\le j\le k$, of a Dirichlet
      random variable $X$ of
      parameters $(\upalpha_1,\ldots,\upalpha_k)$. 
 We have the following representation,     

      $$
 L_{i_1,\ldots,i_k}(\upalpha)=\sum_{i=1}^n a_i\prod_{j=1}^{l_i} f_j, 
 $$
where $n=i_1+\cdots+i_k$ is the degree of $L_{i_1,\ldots,i_k}$,
$l_1,\ldots,l_n\ge 1$ and
$a_1,\ldots, a_n$ are real constants and each term of the sum has the
same degree $n$.
Furthermore, one of the terms of this expansion is a product of $n$
functions of
degree $1$ each.

      \end{lemma}
\begin{proof} 
It is easy to see that it is true for shifted moments of degree $1$. Now, by
induction
on $i_1+\cdots+i_k$ and the recursion (\ref{recursion}), one can  check
that
if the statement is true for all moments of degree $i_1+\cdots+i_k$,
it must also be true for moments of degree $i_1+\cdots+i_k+1$.
        \end{proof}

        \medskip
    
    We now have the following corollary.

    \medskip

    \begin{corollary}
      \label{ccc} Consider a family of Dirichlet random variables
       of parameters
      $\upalpha_{t,1}=\upalpha_{t,2}\to\infty$ as $t\to\infty$, while
      $\upalpha_{t,i}\le 1$ for $3\le i\le 4$.
 Then,
      for all $i_1,i_2\ge 0$
      there is a constant $\newconstant\label{14}>0$ such that

      $$
|\mathbb E_{\upalpha_t}[\bar \xi_1^{i_1}\bar\xi_2^{i_2}]|\le
\frac{\oldconstant{14}}{\upalpha_{t,1}^{\left\lceil\frac{k}{2}\right\rceil}},
$$
where $i_1+i_2=k$.

    \end{corollary}
    \begin{proof}  Note that
      
      $$
\mathbb      E_{\upalpha_t}[\xi_j]=\Psi(\upalpha_{t,j})-\Psi(\sum_{i=1}^4\upalpha_{t,i}),
      $$
      for $j=1,2$. 
Hence,
      from the fact that the digamma
      function has the asymptotics

      \begin{equation}
        \label{psixl}\Psi(x)= \log x-\frac{1}{2x}+o\left(\frac{1}{x}\right),
      \end{equation}
      when
      $x\to\infty$,
      we see that $M_1=\log(1/2)$. Similarly $M_2=\log (1/2)$.
      From (\ref{psixl}) we also see that

      $$
L_{1,0,0,0}(\upalpha_t)=\frac{\oldc{l1}}{\upalpha_{t,1}}+o\left(\frac{1}{\upalpha_{t,1}}\right),
$$
for some constant $\newc\label{l1}>0$. We can deduce a similar
identity
for $L_{0,1,0,0}$.
On the other hand, note that when $i_1+i_2=2$, we have that

$$
L_{i_1,i_2,0,0}=g+f_1f_2,
$$
where $g\in\mathcal A_2$ and $f_1,f_2\in\mathcal A_1$. But for
functions in $\mathcal A_2$ we have the asymptotics

\begin{equation}
  \label{gu}
g(\upalpha)=\frac{\oldc{l2}}{\upalpha_{t,1}}+o\left(\frac{1}{\upalpha_{t,1}}\right),
\end{equation}
for some constant $\newc\label{l2}>0$ depending on $g$.
In general, for $i_1+i_2=2m$ even, we have from Lemma \ref{logmlemma} the expansion

$$
 L_{i_1,i_2,0,0}(\upalpha)=\sum_{i=1}^{2m} a_i\prod_{j=1}^{l_i} f_j.
 $$
 Call $n_j$ the degree of $f_j$, so that $n_1+\cdots+n_{l_i}=2m$. Note
 that

 $$
 f_j=\frac{\partial^\alpha}{\partial^{\alpha_1}\upalpha_{t,1}\partial^{\alpha_2}\upalpha_{t,2}}
 A(\upalpha_t),
 $$
 for some multi-index $\alpha=(\alpha_1,\alpha_2)$ with
 $\alpha_1+\alpha_2=n_j$. For $n_j=1$, this function is of the form

 $$
\Psi(\upalpha_{t,i})-\Psi(2\upalpha_{t,1}),
$$
for $i=1,2$, so the decay of $f_j$ in this case is bounded by

$$
\frac{\oldc{l1}}{\upalpha_{t,i}}.
$$
For the case $n_j\ge 2$, the function $f_j$ is of the following two
possible
forms

$$
\newc(\Psi_{n_j-1}(\upalpha_{t,1})-\Psi_{n_j-1}(2\upalpha_{t,1}))\qquad
{\rm or}\quad \newc\Psi_{n_j-1}(\upalpha_{t,i}).
$$
From part $(ii)$ of Lemma \ref{poly}, in both cases this gives a decay
bounded by

$$
\frac{\newc}{\upalpha_{t,i}^{n_j-1}}.
$$
This means that the decay of

\begin{equation}
  \label{fact}\prod_{i=1}^{l_i}f_j,
\end{equation}
as $t\to\infty$ is

$$
\frac{\newc}{\upalpha_{t,i}^{2m-d}},
$$
where $d$ is the number of factors in (\ref{fact}) with $n_j\ge
2$. This mean that the dominating term of the form (\ref{fact}) is the
one which has the highest number of factors of degree larger than $1$.
It is not difficult to check that this term is

$$
\prod_{i=1}^m f_i, 
$$
where $f_i\in\mathcal A_2$ for $1\le i\le 2$, which gives, since in
this case $d=m$, the decay

$$
\frac{\newc}{\upalpha_{t,i}^{m}}.
$$
For the case in 
which $i_1+i_2=2m-1$, there will be several dominant terms of the same
degree,
inluding one of the form

$$
f\prod_{i=1}^{m-1} g_i, 
$$
where $g_1,\ldots,g_{m-1}\in\mathcal A_2$ and $f\in\mathcal A_1$,
which gives
the same decay as the one for the case of degree $2m$.
      \end{proof}

    Let us now continue
    with the proof of Theorem \ref{moments}.

    \medskip

    \subsection{Proof of Theorem \ref{moments}}
    We will first need to show that a natural family
    of Beta random walks matches moments with itself.

    \medskip

    \begin{lemma}
      \label{parametersbeta}
Consider a family of Beta probability measures 
$(\mathbb P_{\upalpha_t,\upbeta_t})_{t\ge 0}$
such that (\ref{as1}), (\ref{as2}) are satisfied. Assume also that 

$$
M_1:=\lim_{t\to\infty}\mathbb E_{\upalpha_t,\upbeta_t}[\xi_+(0,t)]
=\lim_{t\to\infty}\left(\Psi(\upalpha_t)-\Psi(\upalpha_t+\upbeta_t)\right) 
$$
and 

$$
M_2:=\lim_{t\to\infty}\mathbb E_{\upalpha_t,\upbeta_t}[\xi_-(0,t)]
=\lim_{t\to\infty}\left(\Psi(\upbeta_t)-\Psi(\upalpha_t+\upbeta_t)\right) 
$$
exist. Then, for every $k\ge 1$, and $\alpha$ such that $|\alpha|=k$,
we have that

$$
\left|\mathbb E_{\upalpha_t,\upbeta_t}[\bar\xi^\alpha(x,t)]\right|\le
\upalpha_t^{-\left\lceil\frac{k}{2}\right\rceil}.
$$
Hence, for every $k\ge 0$, $(\mathbb P_{\upalpha_t,\upbeta_t})_{t\ge 0}$ matches moments
up to order $k$ at rate $\upalpha_t^{-\left\lceil\frac{k}{2}\right\rceil}$ with itself.
      \end{lemma}
\begin{proof} The proof follows immediately from Corollary \ref{ccc}
  and
  the definition of matching moments.
  \end{proof}
      \medskip
      
      Let 
\[
h_t= \frac{\log\big(P_{0,\omega}(t, x(\theta)t)\big)-I(x(\theta))}{\sigma(\theta)t^{1/3}},
\]
where $P_{0,\omega}(t,y)$ is defined in (\ref{adopt}).

\medskip

Let $C^k(\mathbb R)$ be the set of functions $f:\mathbb R\to\mathbb R$
whose derivatives up to order $k$ are uniformly bounded. From the fact
that
the set $C^k(\mathbb R)$ is a convergence determining set of
functions (see for example \cite{EK86}),
Theorem \ref{moments} follows immediately from the following lemma.

\medskip

\begin{lemma}
Consider a family 
  of parameters  $(\upalpha_t,\upbeta_t)$ that 
satisfy (\ref{as1}), (\ref{as2}). Let $\theta>0$.
$(\mathbb P_t)_{t\ge 0}$ be a family of environmental laws 
which  matches moments up 
to order $k$ at rate $\upalpha_t^{-\left\lceil\frac{k}{2}\right\rceil}$ with  $(\mathbb 
P_{\upalpha_t,\upbeta_t})_{t\ge 0}$.
Let $\varphi \in C^k(\mathbb R)$. Then there is a
$\newconstant\label{13} >0$ such that 
\[
\big\vert\mathbb E_t[\varphi(h_t)]-\mathbb
E_{\upalpha_t,\upbeta_t}[\varphi( h_t)] \big\vert \le
\oldconstant{13}\frac{t^2 \upalpha_t^{-\left\lceil\frac{k}{2}\right\rceil}
}{\sigma(\theta)t^{1/3}}. 
\]
\end{lemma}
\begin{proof} Let  $z\in\mathbb Z\times\mathbb N$ be  a vertex with
 $z=(v,s)$ for some $v\in\mathbb Z$ and $0\le s\le t$.
For each    $y_1,y_2\in\mathbb R$, define 
    
      \begin{equation}
        \nonumber
        P_\omega(y_1,y_2)=P_\omega^c+  e^{y_1+M_1}P_\omega^++e^{y_2+M_2}P_\omega^-,
    \end{equation}    
    where 

    \begin{itemize}
    \item[(i)] $P_\omega^c$ is the probability that
      $X_t\in A_{t,x(\theta)}$ and that $X$ does not pass through
      the vertex $z$, 
    \item[(ii)] $P_\omega^-$  is the probability that $X_t=A_{t,x(\theta)}$ and $X$ passes through the edge $z$, so that 
      $X_s=v$, $X_{s+1}=v+1$, but with 
      $\omega_+(z)=1$.
    \item[(ii)] $P_\omega^+$  is the probability that $X_t=A_{t,x(\theta)}$ and $X$ passes through the edge $z$, so that 
      $X_s=v$, $X_{s+1}=v-1$, but with 
      $\omega_-(z)=1$.
      \end{itemize}
    Let 
    \[
    h(y_1,y_2):=\frac{\log P_\omega(y_1,y_2)+I(x(\theta))t}{t^{1/3}\sigma(\theta)}.
    \]
    Fixing all  weights in the disorder at edges different from $f$
    and defining $g(y_1,y_2):=\varphi(h(y_1,y_2))$ we now use  Taylor's 
    theorem for $g(y_1,y_2)$ expanding it at $(y_1,y_2)=(0,0)$
    to conclude that 
    
    \begin{align*}
  &  \varphi(h(\bar\xi_+(z),
    \bar\xi_-(z)))=g(\bar\xi_+(z),\bar\xi_-(z))=
    g(0,0)+
    \sum_{i=1}^2 g_{i}(0,0)\bar\xi_i\\
      &+\frac{1}{2}\sum_{i,j=1}^2 g_{i,j}(0,0)\bar\xi_i\bar\xi_j 
  + \cdots+ \sum_{i_1,\ldots, i_{k-1}=1}^{2}\frac{1}{(k-1)!}
        g_{i_1,\ldots,i_{k-1}}(0)\bar\xi_{i_1}\cdots\bar\xi_{i_{k-1}}\\
    &  + \sum_{i_1,\ldots, i_{k}=1}^{2}\frac{1}{k!} g_{i_1,\ldots,i_{k}}(s\bar\xi_+(z),s\bar\xi_-(z))\bar\xi_{i_1}\cdots\bar\xi_{i_{k}},
    \end{align*}
    for some $s\in(0,1)$, where for $i_1,\ldots,i_j\in \{0,1\}$,

    $$
g_{i_1,\ldots,i_j}=\frac{\partial }{\partial
  y_{i_1}}\cdots \frac{\partial }{\partial y_{i_j}} g,
$$
and
    $\bar\xi_1:=\xi_+(z)$ and $\bar\xi_2:=\xi_-(z)$. Taking expectation and using the 
    independence of $(\xi_{x,t})_{(x,t)\in\mathbb{Z}\times\mathbb N}$, 
    we get that 
    \begin{align}
\nonumber&      \mathbb E_t[\varphi(h(\bar \xi_+(z),\bar\xi_-(z)))]=
      a+
    \sum_{i=1}^2 a_{i}\mathbb E_t [\bar\xi_i]\\
     \nonumber &+\frac{1}{2}\sum_{i,j=1}^2 a_{i,j}\mathbb E_t\left[\bar\xi_i\bar\xi_j\right] 
  + \cdots+ \sum_{i_1,\ldots, i_{k-1}=1}^{2}\frac{1}{(k-1)!}
        a_{i_1,\ldots,i_{k-1}}\mathbb E_t[\bar\xi_{i_1}\cdots\bar\xi_{i_{k-1}}]\\
   \label{phie} &  + \sum_{i_1,\ldots, i_{k}=1}^{2}\frac{1}{k!}
   \mathbb E_t[   g_{i_1,\ldots,i_{k}}(s\bar\xi_+(z),s\bar\xi_-(z))\bar\xi_{i_1}\cdots\bar\xi_{i_{k}}],
    \end{align}
    where

    $$
a=\mathbb E_t[g(0,0)]
$$
and
$$
a_{i_1,\ldots,i_j}=\mathbb E_t[ g_{i_1,\ldots,i_j}(0,\ldots,0)].
$$
Now,  we will prove  that

\begin{equation}
  \label{gg1}
  |a_{i_1,\ldots,i_j}|=\lvert g_{i_1,\ldots,i_j}(0,0)\rvert\le 
  \frac{\newc}{t^{1/3}\sigma(\theta)}
  \end{equation}
for all $0\le j\le k-1$ and

\begin{equation}
  \label{gg2}
  |\bar a_{i_1,\ldots,i_k}|:=\lvert g_{i_1,\ldots,i_k}(y_1,y_2)\rvert\le 
  \frac{\newc}{t^{1/3}\sigma(\theta)}.
\end{equation}
We have an expansion analogous to (\ref{phie}) for $\mathbb
E_{\upalpha_t,\upbeta_t}[
\phi(h(\bar\xi_+(z),\bar\xi_-(z)))]$, with the same coeffiecients
$a_j$ for $0\le j\le k$.
    We use Fa\`{a} di Bruno's formula for the
    chain rule of a composition
    (here we use the multivariate version, see \cite{KQ18}),

    $$
    g_{i_1,\ldots,i_j}(y)=\sum_{\pi\in\Pi}\varphi^{(|\pi|)}(y)\prod_{B\in\pi}
    \frac{\partial^{|B|}h}{\prod_{j\in B}\partial y_j},
    $$
    where $\Pi$ is the set of partitions of $\{1,\ldots,j\}$, $\pi$
    is an arbitrary partition of $\Pi$, $|\pi|$ is the number of
    blocks in the partition $\pi$, $B\in\pi$ means that the variable
    $B$ runs through all the blocks of $\pi$, $|B|$
    is the size of block $B$ and the variables $y_1,\ldots,y_j$
    take only the values $y_1$ and $y_2$ (with some abuse of
    notation). Since by assumption $\phi\in C^k(\mathbb
    R)$,
    it is enough to bound the derivatives $
    \frac{\partial^{|B|}h}{\prod_{j\in B}\partial y_j}$.
    Now,

    $$
    \frac{\partial P_\omega(y_1,y_2)}{\partial y_i}
    =
    \frac{e^{y_i+M_i} P^+_\omega}{P^c_\omega+e^{y_1+M_1}P^+_\omega+e^{y_2+M_2}P^-_\omega}=:p_i(y_1,y_2),
    $$
    for $i=1,2$. We then obtain for the higher
    order derivatives with $k_1+k_2=k$,

    $$
\frac{\partial^k P_\omega(y_1,y_2)}{\partial^{k_1} y_1\partial^{k_2}
  y_2}
=\mathcal P_{k_1,k_2}(p_1(y),p_2(y)),
$$
where the following recursion formula holds,

$$
\mathcal P_{k_1+1,k_2}(p_1(y),p_2(y))
= \mathcal P^{(1,0)}_{k_1,k_2}(p_1(y),p_2(y))p^{(1)}_1(y)
+\mathcal P^{(0,1)}_{k_1,k_2}(p_1(y),p_2(y))p^{(1)}_2(y),
$$
and a similar recursion formula for $\mathcal P_{k_1,k_2+1}(p_1(y),p_2(y))$,
where $\mathcal P^{(1,0)}$ and $\mathcal P^{(0,1)}$ are the
partial derivatives of $\mathcal P$ with respect to its first
and second variable respectively, and $p^{(i)}_j$ is the
partial derivative of $p_j$ with respect to $y_i$, $1\le i,j\le 2$.
This proves (\ref{gg1}) and (\ref{gg2}).

Now, from (\ref{phie}), (\ref{gg1}) and (\ref{gg2}) (and the
corresponding
expansion for $\mathbb E_{\upalpha_t,\upbeta_t}[\varphi(h(\bar 
            \xi_+(z),\bar\xi_-(z)))]$, we get that

        \begin{align*}
          & \left|     \mathbb E_t[\varphi(h(\bar 
            \xi_+(z),\bar\xi_-(z)))]-\mathbb E_{\upalpha_t,\upbeta_t}[\varphi(h(\bar 
            \xi_+(z),\bar\xi_-(z)))]\right|
          \\
          &\le
            \left(|a|+\sum_{i=1}^2 |a_i|+\sum_{i,j=1}^2 a_{i,j}+
            \sum_{i_1,\ldots,i_{k-1}}
            |a_{i_1,\ldots,i_{k-1}}|\right.\\
          &+\left. \sum_{i_1,\ldots,i_{k-1}}
            \bar a_{i_1,\ldots,i_{k-1}}
            \right) \upalpha_t^{-\left\lceil\frac{k}{2}\right\rceil}\\
          &\le \frac{\newc \upalpha_t^{-\left\lceil\frac{k}{2}\right\rceil}
            }{t^{1/3}\sigma(\theta)}.
    \end{align*}
    Summing over all $z$ that can be reached  up to time $t$
    starting
    from $(0,0)$ we finish the proof of the lemma.
  \end{proof}

\medskip

\subsection{Proof of Corollary \ref{corollary1}}
Consider a family of Beta random walks with parameters
$(\upalpha_{t,1},\upalpha_{t,2})_{t\ge 0}$, with
$\upalpha_{t,1}=\upalpha_{t,2}= t^r$. Note that
since $r\in (0,1)$
 conditions (\ref{as1}) and (\ref{as2}) are satisfied.
 Now, using the fact that

 $$
\sigma(\theta)\sim t^{-r} \qquad \mathfrak g(\upalpha,\upbeta)=\frac{1}{2}t^{-r},
$$
as $t\to\infty$ (see part $(iii)$  of Corollary \ref{corcor}),note that an integer $k$ satisfies (\ref{tta}) if and only if

$$
\left\lceil\frac{k}{2}\right\rceil>\frac{5}{3r}+1.
$$
Let us denote
by $\mathbb E'_{\upalpha_t}$ the expectation with respect
to the environment of this Beta random walk. Recall
that the parameters of the Dirichlet environment
are $\upalpha_{t,1}=\upalpha_{t,2}=t^r$, $|\upalpha_{t,3}|\le
t^{-p}$ and $\upalpha_{t,4}\le t^{-p}$.
On the other hand, note that

\begin{align}
\nonumber &\left|\mathbb E_{\upalpha_t}[\bar\xi_+]-\mathbb E'_{\upalpha_t}[\bar\xi_+]\right|
=\left|\Psi(\upalpha_{t,1})-\Psi(2\upalpha_{t,1}+2\upalpha_{t,3})
                    -\Psi(\upalpha_{t,1})+\Psi(2\upalpha_{t,1})\right|\\
 \label{bbbb} &\le \oldc{kkk}\frac{\upalpha_{t,3}}{\upalpha_{t,1}}\le \oldc{kkk}
    t^{-(r+p)}\le\oldc{kkk2}  t^{-r\left\lceil\frac{k}{2}\right\rceil},
\end{align}
for some constants $\newc\label{kkk}>0$, $\newc\label{kkk2}>0$,
and the last inequality is satisfied only when

$$
p\ge r\left\lceil\frac{k}{2}\right\rceil-r.
$$
For higher order moments, using the recursions (\ref{recursion}), note that the
shifted logarithmic moments of the random walk in Dirichlet
environment
can be obtained from the shifted logarithmic moments of the Beta
random walk
by changing in all the expressions involving the polygamma functions
the sum
$\sum_{i=1}^2\upalpha_{t,i}$ by $\sum_{i=1}^4\upalpha_{t,i}$. Hence, the same
bound (\ref{bbbb}) will be  satisfied for the differences between
higher
order moments. Now choose $k\ge 1$ so that

$$
\left\lceil\frac{k}{2}\right\rceil =\left\lceil\frac{5}{3r}-\frac{1}{3}\right\rceil. 
$$
Note that this is possible because $\frac{5}{3r}-\frac{1}{3}>0$ for
all $r\in (0,1)$.
Hence, it is enough to choose $p$ so that

$$
p\ge r \left\lceil\frac{5}{3r}-\frac{1}{3}\right\rceil-r. 
$$


\medskip






\begin{thebibliography}{9}



  
\bibitem[AKQ14]{AKQ14} T. Alberts, K. Khanin and J. Quastel. 
  \textit{The intermediate disorder regime for directed polymers 
  in dimension 1+1.} Ann. Probab. 42, no. 3, 1212–1256 (2014).
  
\bibitem[AC16]{AC16} A. Auffinger and W.-K. Chen.
\textit{Universality of chaos and ultrametricity in mixed p-spin
  models.} Comm. Pure Appl. Math. 69, no. 11, 2107-2130 (2016).

\bibitem[BC17]{Barr-Cor} 
G. Barraquand and I. Corwin,
\textit{Random walk in Beta-distributed random environment.}
Probab. Theory  Related Fields, 167, 1057-1116 (2017).


\bibitem[BMRS19]{BMRS19}
  R. Bazaes, C. Mukherjee, A. F. Ram\'\i rez, and S. Saglietti. 
  \textit{Quenched and averaged large deviation rate functions for random walks in 
random environments: the impact of disorder.} arXiv preprint arXiv:1906.05328, 
(2019).

\bibitem[BC14]{Bor-Cor} A. Borodin and I. Corwin.
  \textit{Macdonald processes.} Probab. Theory  Related Fields
  158, 225-400 (2014).


  \bibitem[BCF14]{BCF14} A. Borodin, I. Corwin and P. Ferrari. 
    \textit{Free energy fluctuations for directed polymers in random 
      media in $1+1$ dimension.}
    Comm. Pure Applied Math. 68, 1129-1214 (2014).


  \bibitem[BCR13]{BCR13} A. Borodin, I. Corwin and D. Remenik.
    \textit{Log-Gamma polymer free energy fluctuations via
      a Fredholm determinantal identity.}
    Comm. Math. Phys. 324, 215-232 (2013).

    
  \bibitem[CG17]{CG17} I. Corwin and Y. Gu. 
  \textit{Kardar-Parisi-Zhang equation and large deviations for 
    random walks in weak random environments}. 
  J. Stat. Phys. 166, 150-168 (2017).


\bibitem[EK86]{EK86} S. Ethier and T. Kurtz.
  \textit{Markov Processes, Charcaterization and Convergence.}
  Wiley, New York, (1986).

\bibitem[K21]{K21} S. Korotkikh.
  \textit{ Hidden diagonal integrability of $q$-hahn vertex model and
    beta polymer model.} arXiv:2105.05058 (2021).
  
\bibitem[KQ18]{KQ18} A. Krishnan and J. Quastel. 
  \textit{Tracy-Widom fluctuations for perturbations of the log-gamma 
    polymer 
    in intermediate disorder}. Ann. Appl. Probab. 28, 3736-3764 (2018). 


  
\bibitem[RSY13]{RasA-Sepp-Y} 
F. Rassoul-Agha, T. Seppäläinen , and A. Yilmaz. 
\textit{Quenched free energy and large deviations for 
random walks in random potentials.} 
Comm. Pure Appl. Math. 66, no. 2, 202–244 (2013).


\bibitem[RS14]{RasA-Sepp} 
F. Rassoul-Agha and T. Seppäläinen.
\textit{Quenched point-to-point free energy for random walks in random
potentials.}
Probab. Theory Related Fields 158, 711–750 (2014).

\bibitem[S12]{S12} T. Sepp\"al\"ainen.
  \textit{Scaling for a one-dimensional directed polymer
    with boundary conditions.}
  Ann. Probab. 40, 19-73 (2012).

\end{thebibliography}
\end{document}